\documentclass[a4paper,12pt]{article}
\usepackage{amsmath,amssymb,amsthm,amsxtra}
\newcommand{\C}{{\mathbb C}}
\newcommand{\K}{{\mathbb K}}
\newcommand{\N}{{\mathbb N}}
\newcommand{\R}{{\mathbb R}}

\newcommand{\abs}[2][\empty]{\ifx#1\empty\left|#2\right|%
\else#1\vert #2 #1\vert\fi}
\newcommand{\Alg}{{\mathcal A}}

\newcommand{\cLin}{\mathop{\mathcal L}\nolimits}%continuous linear operators
\newcommand{\Cnt}[1][]{{\cal C}^{#1}}

\newcommand{\co}[1]{{#1}^{c}}%complement
\newcommand{\comp}{\circ}%composition
\newcommand{\Completion}[1]{\widehat{#1}}
\newcommand{\csub}{\subset\subset}
\newcommand{\defstyle}[1]{\emph{#1}}
\newcommand{\eps}{\varepsilon}

\newcommand{\functspec}{\mathop{\mathrm{fsp}}\nolimits}%spectrum of an element

\newcommand{\Hom}{\mathop{\mathrm{Hom}}\nolimits}%homomorphisms
\newcommand{\id}{\mathop{\mathrm{id}}}
\newcommand{\ideal}{\unlhd}
\newcommand{\idealproper}{\lhd}
\renewcommand{\iff}{\Leftrightarrow}
\renewcommand{\Im}{\mathop{\mathrm{Im}}}
\renewcommand{\implies}{\Rightarrow}
\newcommand{\inner}[3][\empty]{\ifx#1\empty\left\langle #2,#3\right\rangle\else#1\langle #2,#3 #1\rangle\fi}%optional arg=size
\newcommand{\interl}{\mathop{\mathrm{interl}}\nolimits}
\newcommand{\interior}[1]{{#1}^\circ}
\newcommand{\intspec}{\mathop{\mathrm{intsp}}\nolimits}%interior spectrum of an element
\newcommand{\inv}[1]{{#1}^{-1}}%inverse of a map

\newcommand{\Ker}{\mathop{\mathrm{Ker}}}%kernel
\renewcommand{\Re}{\mathop{\mathrm{Re}}}
\newcommand{\Smlf}[1]{\Hom(#1,\GenC)}%space of all surjective multiplicative linear functionals
\newcommand{\norm}[2][\empty]{\ifx#1\empty\left\Vert#2\right\Vert%
\else#1\Vert #2 #1\Vert\fi}%optional arg=size

\newcommand{\pseudonorm}[2][\empty]{\ifx#2.{\mathcal P}\else\ifx#1\empty{\mathcal P}(#2)%
\else{\mathcal P}#1( #2 #1)\fi\fi}%optional arg=size%because of ugly spacing with \left...\right, we chose the default to be no sizing at all
\newcommand{\pseudonormprime}[2][\empty]{\ifx#2.{\mathcal P}'\else\ifx#1\empty{\mathcal P}'(#2)%
\else{\mathcal P}'#1( #2 #1)\fi\fi}
\newcommand{\resolv}{\rho}%resolvent set

\newcommand{\spec}{\mathop{\mathrm{sp}}\nolimits}%spectrum of an element
\newcommand{\specrad}{\mathop{r}\nolimits}%spectral radius
\newcommand{\val}{\ensuremath{\mathrm v}}%valuation
\newcommand{\wt}{\widetilde}

\newcommand{\Gen}{{\mathcal G}}
\newcommand{\GenC}{\widetilde\C}
\newcommand{\GenK}{\widetilde\K}
\newcommand{\GenR}{\widetilde\R}
\newcommand{\EMod}{{{\mathcal E}_M}}
\newcommand{\Mod}{{\mathcal M}}
\newcommand{\Null}{{\mathcal N}}
\newcommand{\sharpnorm}[2][\empty]{\abs[#1]{#2}_{\mathrm e}}
\newcommand{\caninf}{\alpha}%canonical infinitesimal

\newcommand{\Gencnt}[1]{\Gen_{\Cnt(#1)}}
\newcommand{\Gensharp}[1]{\Gen^\#_{\Cnt(#1)}}
\newcommand{\tGen}{\widetilde\Gen}

\newtheorem{thm}{Theorem}[section]
\newtheorem{prop}[thm]{Proposition}
\newtheorem{lemma}[thm]{Lemma}
\newtheorem{df}[thm]{Definition}
\newtheorem{cor}[thm]{Corollary}
\newtheorem{ex}[thm]{Example}

\theoremstyle{remark}%styles: plain, definition, remark
\newtheorem*{rem}{Remark}

\begin{document}
\title{Banach $\GenC$-algebras}
\author{Hans Vernaeve\footnote{Supported by FWF (Austria), grants M949-N18 and  Y237-N13.}\\Dept.\ of Pure Mathematics and Computer Algebra\\Ghent University\\
{\tt hvernaev@cage.ugent.be}}
\date{}
\maketitle

\begin{abstract}
We study Banach $\GenC$-algebras, i.e., complete ultra-pseudo-normed algebras over the ring $\GenC$ of Colombeau generalized complex numbers. We develop a spectral theory in such algebras. We show by explicit examples that important parts of classical Banach algebra theory do not hold for general Banach $\GenC$-algebras and indicate a particular class of Banach $\GenC$-algebras that overcomes these limitations to a large extent. We also investigate $C^*$-algebras over $\GenC$.
\end{abstract}



\section{Introduction}
Over the past thirty years, nonlinear theories of generalized functions have been developed by many authors \cite{AR91,GKOS,NPS98} mainly inspired by the work of J.\ F.\ Colombeau \cite{Colombeau84,Colombeau85}. They have proved to be a valuable tool for treating partial differential equations with singular data or coefficients \cite{GGO03,HdH01,HOP05,O92}. In recent research on the subject, a variety of algebras of generalized functions have been introduced in addition to the original construction by Colombeau, culminating in the definition of the Colombeau space $\Gen_E$ constructed on an arbitrary locally convex topological vector space $E$ (see \cite{Garetto2005} and the references therein). These spaces carry a natural structure as modules or algebras over the ring $\GenC$ of generalized complex numbers. In particular, if $E$ is a Banach algebra, then $\Gen_E$ turns out to be a Banach $\GenC$-algebra, i.e., a complete ultra-pseudo-normed $\GenC$-algebra. Thus several classes of generalized functions and generalized linear operators carry the structure of a Banach $\GenC$-algebra.\\
This paper is devoted to the study of Banach $\GenC$-algebras and Colombeau $C^*$-algebras (i.e., the analogues of $C^*$-algebras over $\GenC$) and their spectral theory. In order to ensure some of the most basic properties of the spectrum of an element of such an algebra, such as the facts that spectrum is bounded and that the spectrum of $0$ is $\{0\}$, the spectrum is not defined as the complement of the resolvent set, but by means of a slightly more restrictive property (definition \ref{df_spectrum}, theorem \ref{thm_resolv_determines_spec}). This also allows for analogues of the classical interpretations of the spectrum in the algebra of continuous functions on a compact set and in matrix algebras (propositions \ref{prop_spec_Gencbd} and \ref{prop_spec_matrix}).\\
Although restrictive versions of the Gelfand-Mazur theorem (\S \ref{section_GelfandMazur}) and the spectral mapping theorem (\S \ref{section_spectralmapping}) hold for general Banach $\GenC$-algebras, the strength of the classical results is not reached, and we give explicit counterexamples, showing, e.g., that the spectrum of an element can be empty (example \ref{ex_empty_spectrum_wrt_S}). Since most of the algebras of generalized functions and operators arise as subalgebras of $\Gen_B$, where $B$ is an algebra over $\C$, we restrict our attention to such algebras. Still, e.g., the spectrum of a self-adjoint element in a Colombeau $C^*$-subalgebra of $\Gen_B$ (with $B$ a $C^*$-algebra) can contain nonreal generalized numbers (example \ref{ex_spec_Cstar_subalgebra}).\\
This leads us to consider a particular class of Banach $\GenC$-algebras and Colombeau $C^*$-algebras, called strictly internal algebras (and more generally, algebras obtained by set-theoretic constructions on strictly internal algebras), in which a large part of the classical Banach algebra theory holds true. E.g., we obtain a stronger version of the spectral mapping theorem (theorem \ref{thm_poly_spec_calc_internal}), a spectral radius formula (theorem \ref{thm_specrad_formula}, using the theory of generalized holomorphic functions), conditions under which the spectrum of an element does not depend on the algebra to which the element belongs (proposition \ref{prop_stable_spec_subset_GenR} and theorem \ref{thm_spec_of_Cstar_subalgebra}) and an analogue of the Gelfand-Neumark theorem (proposition \ref{prop_GNS}). By means of the example $\Gensharp X$ of sharply continuous generalized functions on a compact metric space $X$ (\S \ref{section_Gensharp}), we indicate that this class of Banach $\GenC$-algebras contains algebras that arise naturally in the study of nonlinear generalized functions.

\section{Preliminaries}\label{section_prelims}
In a topological space $X$, $K\csub X$ means that $K$ is a compact subset of $X$. The topological closure of $A\subset X$ is denoted by $\overline A$, the topological interior by $\interior A$. We denote by $\Cnt(X)$ the ring of continuous $\C$-valued functions on $X$. We denote $\R^+:= \{x\in\R: x> 0\}$.

\subsection{Algebras}
In this paper, a $\GenC$-algebra will always be an \emph{associative} algebra over $\GenC$ \emph{with $1$}. Commutativity is not assumed, unless explicitly stated.\\
Let $\Alg$ be a $\GenC$-algebra. By an ideal, we mean a two-sided ideal; otherwise, the adjective `one-sided' (or `left', resp.\ `right') will be added. We write $I\idealproper \Alg$ iff $I$ is a proper ideal of $\Alg$ (i.e., an ideal different from $\Alg$ itself). A not necessarily proper ideal is denoted by $I\ideal\Alg$.
By an inverse of $u\in\Alg$, we mean a two-sided inverse, i.e., $v\in\Alg$ such that $uv=vu=1$; in that case, $u$ is called invertible and we denote $v=\inv{u}$ (which is uniquely determined).\\
$\Alg$ is called faithful iff for each $\lambda\in \GenC\setminus\{0\}$, $\lambda 1\ne 0$ in $\Alg$.
%cf.\ Hungerford: every ring is a $\Z$-algebra; so faithfulness is not usually assumed in the definition of an algebra)
More generally, a $\GenC$-module $\Gen$ is faithful iff for each $\lambda\in\GenC\setminus\{0\}$, there exists $u\in\Gen$ such that $\lambda u \ne 0$.

In a faithful $\GenC$-algebra, we identify the subring $\{\lambda 1: \lambda\in\GenC\}$ with $\GenC$.\\
We denote by $\Smlf\Alg$ the set of all $\GenC$-algebra morphisms $\Alg\to\GenC$ preserving $1$, i.e., the set of all surjective multiplicative $\GenC$-linear functionals on $\Alg$.

\subsection{Generalized function algebras, generalized numbers}
Let $E$ be a normed vector space over $\C$. Then the Colombeau space $\Gen_E:= \Mod_E/\Null_E$ \cite{Garetto2005}, where
\begin{align*}
\Mod_E &= \{(u_\eps)_\eps\in\ E^{(0,1)}: (\exists N\in\N) (\norm{u_\eps}\le \eps^{-N}, \text{ for small }\eps)\}\\
\Null_E &= \{(u_\eps)_\eps\in\ E^{(0,1)}: (\forall m\in\N) (\norm{u_\eps}\le \eps^{m}, \text{ for small }\eps)\}.
\end{align*}
Elements of $\Mod_{E}$ are called moderate, elements of $\Null_{E}$ negligible.
The element of $\Gen_E$ with $(u_\eps)_\eps$ as a representative is denoted by $[(u_\eps)_\eps]$. $\GenR:=\Gen_{\R}$ and $\GenC:=\Gen_\C$ are the so-called Colombeau generalized numbers. With the componentwise operations, $\Gen_E$ is a $\GenC$-module. If $E$ is a normed $\C$-algebra, $\Gen_E$ is a $\GenC$-algebra with the componentwise multiplication.\\
Let $S\subseteq(0,1)$. Then $e_S:=[(\chi_S(\eps))_\eps]\in\GenR$, where $\chi_S$ is the characteristic function on $S$. An element $e\in\GenC$ is idempotent (i.e., $e^2 = e$ holds in $\GenC$) iff there exists $S\subseteq (0,1)$ such that $e = e_S$ \cite{AJOS2006}. 
The complement of $S\subseteq(0,1)$ is denoted by $\co S$. By $\caninf\in\GenR$ we denote the element with $(\eps)_\eps$ as a representative.\\
For $(z_\eps)_\eps\in\C^{(0,1)}$, the valuation $\val(z_\eps):= \sup\{b\in\R: \abs{z_\eps}\le \eps^b$, for small $\eps \}$ and the so-called sharp norm $\sharpnorm{z_\eps}:= e^{- \val(z_\eps)}$. For $\tilde z=[(z_\eps)_\eps]\in\GenC$, $\val(\tilde z):= \val(z_\eps)\in (-\infty, \infty]$ and $\sharpnorm{\tilde z}:= \sharpnorm{z_\eps}\in [0,+\infty)$ are defined independent of the representative of $\tilde z$. For $a,b\in\GenR$, we write $a\gg b$ (eq., $b\ll a$) if $a-b\ge \caninf^m$, for some $m\in\N$ (i.e., if $a\ge b$ and $a-b$ is invertible in $\GenR$).\\
An ultra-pseudo-seminorm on a $\GenC$-module $\Gen$ is a map $\pseudonorm{.}$: $\Gen\to [0,+\infty)\subset \R$ satisfying \cite{Garetto2005,GV_Hilbert} $\pseudonorm{u+v}\le \max(\pseudonorm u,\pseudonorm v)$ and $\pseudonorm{\lambda u}\le \sharpnorm{\lambda}\pseudonorm{u}$, for each $u,v\in\Gen$ and $\lambda\in\GenC$. It follows that $\pseudonorm{\caninf^r u} = \sharpnorm{\caninf^r}\pseudonorm u = e^{-r}\pseudonorm u$, $\forall u\in\Gen$, $\forall r\in\R$. An ultra-pseudo-norm $\pseudonorm{.}$ on a $\GenC$-module $\Gen$ is an ultra-pseudo-seminorm that satisfies $\pseudonorm u\ne 0$, $\forall u\in\Gen\setminus\{0\}$. A Banach $\GenC$-module \cite{Garetto2005} is a complete pseudonormed $\GenC$-module. Typical examples of Banach $\GenC$-modules are $\Gen_E$ ($E$ a normed $\C$-vector space) with $\pseudonorm{.}$: $\Gen_E\to\GenC$: $\pseudonorm u = \sharpnorm[\big]{\norm{u_\eps}}$ \cite[Prop.\ 3.4]{Garetto2005}. For $u = [(u_{\eps})_{\eps}]\in\Gen_{E}$, we will denote $\norm u:= [(\norm{u_\eps})_{\eps}]\in\GenR$. Similarly, for $\tilde z=[(z_\eps)_\eps]\in\GenC$, we will denote $\abs{\tilde z}:= [(\abs{z_\eps})_\eps]\in\GenR$.\\
If $(\Gen,\pseudonorm{.})$ is an ultra-pseudo-normed $\GenC$-module, then it is in particular an ultrametric space with the translation invariant distance $d(u,v)=\pseudonorm{u-v}$. Hence $\Gen$ has a completion $\Completion\Gen$ \cite[\S 4]{Koethe69}. It is easy to check that also $\Completion\Gen$ is a $\GenC$-module and that the distance on $\Completion\Gen$ is translation invariant and induced by an ultra-pseudo-norm that extends $\pseudonorm{.}$. Hence $\Completion\Gen$ is a Banach $\GenC$-module \cite{Garetto2005}. For ultra-pseudo-normed $\GenC$-modules $\Gen_1$, $\Gen_2$, we denote by $\cLin(\Gen_1,\Gen_2)$ the set of all continuous $\GenC$-linear maps $\Gen_1\to\Gen_2$. Further, $\cLin(\Gen_1):= \cLin(\Gen_1, \Gen_1)$.\\
Let $A_\eps\subseteq E$, $\forall \eps\in (0,1)$. Then the set
\[
[(A_\eps)_\eps] := \{u\in \Gen_E: (\exists \text{ repres.\ $(u_\eps)_\eps$ of } u) (u_\eps\in A_\eps, \text{ for small }\eps)\}
\]
is called the \defstyle{internal subset} of $\Gen_E$ with representative $(A_\eps)_\eps$ \cite{OV_internal}. For $A\subseteq E$, we denote $\widetilde A:=[(A)_\eps]$. The interleaved closure \cite{OV_internal} of a subset $A$ of $\Gen_{E}$ is the set
\[
\interl(A) := \Big\{\sum_{j=1}^m e_{S_j} a_{j}: m\in\N, \{S_{1},\dots, S_{m}\} \text{ a partition of }(0,1), a_{j}\in A\Big\}.
\]
$A\subseteq \Gen_E$ is called closed under interleaving if $A=\interl(A)$, i.e., if $\forall \lambda,\mu\in A$ and $T\subseteq (0,1)$, also $e_T\lambda+e_{\co T}\mu\in A$. Internal sets are closed under interleaving.\\
A subset $A$ of $\Gen_E$ is called sharply bounded if $\sup_{u\in A}\pseudonorm u<+\infty$. An internal set $A$ is sharply bounded iff $A$ has a sharply bounded representative, i.e., a representative $[(A_\eps)_\eps]$ for which there exists $M\in\N$ such that $\sup_{u\in A_\eps}\norm{u}\le \eps^{-M}$, for small $\eps$ \cite[Lemma 2.4]{OV_internal}.

\subsection{Invertibility w.r.t.\ $S\subseteq (0,1)$}
The following concept was introduced in \cite{HV_isom}:
\begin{df}
Let $\Alg$ be a $\GenC$-algebra. Let $u\in\Alg$ and $S\subseteq (0,1)$ with $e_S 1\ne 0$ (in $\Alg$). Then $u$ is called \defstyle{invertible w.r.t.\ $S$} iff
\[(\exists v\in \Alg)(uv=vu=e_S1).\]
Such $v$ is called an inverse of $u$ w.r.t.\ $S$. Similarly, $l$ is a left inverse of $u$ w.r.t.\ $S$ iff $lu=e_S1$ and $r$ is a right inverse of $u$ w.r.t.\ $S$ iff $ur=e_S 1$.
\end{df}

\begin{lemma}\label{lemma_invertible_wrt_S}
Let $\Alg$ be a $\GenC$-algebra. Let $S\subseteq (0,1)$ with $e_S 1\ne 0$. 
\begin{enumerate}
\item Let $u\in\Alg$ have a left and a right inverse w.r.t.\ $S$. Then for every left inverse $l$ of $u$ w.r.t.\ $S$ and every right inverse $r$ of $u$ w.r.t.\ $S$, $e_S l = e_S r$ is a two-sided inverse of $u$ w.r.t.\ $S$.
\item $u\in\Alg$ is invertible w.r.t.\ $S$ iff $e_S u + e_{\co S}1$ is invertible.\\
In particular, $u$ is invertible w.r.t.\ $S$ iff $e_S u$ is invertible w.r.t.\ $S$.
\item Let $u_1, u_2\in\Alg$. Then $u_1u_2$ and $u_2 u_1$ are both invertible w.r.t.\ $S$ iff $u_1$ and $u_2$ are both invertible w.r.t.\ $S$.
\item Let $\Alg$ be faithful and $c = [(c_{\eps})_{\eps}]\in\GenC$. Then $c$ is invertible w.r.t.\ $S$ in $\GenC$ iff $c$ is invertible w.r.t.\ $S$ in $\Alg$ iff there exists $m\in\N$ such that $\abs{c_{\eps}}\ge \eps^m$, for small $\eps\in S$.
\end{enumerate}
\end{lemma}
\begin{proof}
(1) If $lu=ur=e_S 1$, then $e_S l=l(ur) = (lu)r= e_S r$. Hence also $(e_S l)u = e_S 1 = u(e_S r)$, so $e_S l=e_S r$ is an inverse of $u$ w.r.t.\ $S$.\\
(2) If $uv=e_S1$, for some $v\in \Alg$, then $(e_S u+e_{\co S}1)(e_S v + e_{\co S}1)= e_Suv + e_{\co S}1=1$.\\
Conversely, if $(e_Su+e_{\co S}1)v=1$, for some $v\in\Alg$, then multiplying by $e_S$ shows that $u (e_S v) = e_S 1$, and similarly for left inverses.\\
(3) If $(u_1 u_2) v_1 = v_2 (u_2 u_1) = e_S 1$, then $u_1$ has a left and right inverse w.r.t.\ $S$.\\
Conversely, if $u_1 v_1= e_S1$ and $u_2 v_2=e_S1$, then $u_1 u_2 v_2 v_1 = e_S u_1 v_1 = e_S 1$.\\
(4) Clearly, if $d\in\GenC$ and $cd=e_S$ holds in $\GenC$, then it also holds in $\Alg$. Conversely, if $c$ is not invertible w.r.t.\ $S$ in $\GenC$, then there exists $T\subseteq S$ with $e_T\ne 0$ such that $c e_T = 0$ \cite[Lemma 4.1]{HV_ideals}. So if $u\in\Alg $ and $c u = e_S$ holds in $\Alg$, then $e_T = e_T e_S = c e_T u = 0 $ (in $\Alg$, hence also in $\GenC$, by faithfulness of $\Alg$), a contradiction. The second equivalence follows by \cite[Lemma 4.1]{HV_ideals}.
\end{proof}
We also recall that for $c\in\GenC$, $c = 0$ iff $c$ is not invertible w.r.t.\ $S$, for each $S\subseteq (0,1)$ with $e_{S}\ne 0$ \cite[Lemma 4.1]{HV_ideals}.

\section{Banach $\GenC$-algebras}
\begin{df}
An ultra-pseudo-seminorm $\pseudonorm{.}$ on a $\GenC$-algebra $\Alg$ is called \defstyle{submultiplicative} if $\pseudonorm{uv}\le \pseudonorm{u}\pseudonorm{v}$, $\forall u,v\in\Alg$.\\
An \defstyle{ultra-pseudo-normed $\GenC$-algebra} $\Alg$ is a $\GenC$-algebra that is provided with a submultiplicative ultra-pseudo-norm $\pseudonorm{.}$ satisfying $\pseudonorm{1}=1$. We denote by the pair $(\Alg,\pseudonorm{.})$ the algebra $\Alg$ provided with the ultra-pseudo-norm $\pseudonorm{.}$.\\
A \defstyle{Banach $\GenC$-algebra} $\Alg$ is a complete ultra-pseudo-normed $\GenC$-algebra (in particular, $\Alg$ is a Banach $\GenC$-module \cite{Garetto2005}).
\end{df}
Typical examples of faithful Banach $\GenC$-algebras are $\Alg=\Gen_B$, where $B$ is a Banach $\C$-algebra.%(see section \ref{section_internal_Banach_algebra})

We recall that two ultra-pseudo-seminorms $\pseudonorm{.}$ and $\pseudonormprime{.}$ on a $\GenC$-module $\Gen$ are \defstyle{equivalent} if there exists $C\in\R$ such that $\pseudonorm{u}\le C\pseudonormprime{u}$ and $\pseudonormprime{u}\le C\pseudonorm{u}$, $\forall u\in\Gen$ \cite[Def.~4.7]{Garetto2007}.

\begin{lemma}\label{lemma_seminorm_image_of_1}
Let $\pseudonorm{.}\ne 0$ be an ultra-pseudo-seminorm on a $\GenC$-algebra $\Alg$ for which there exists a constant $C\in\R$ such that $\pseudonorm{uv}\le C\pseudonorm{u}\pseudonorm{v}$, $\forall u,v\in\Alg$. Then there exists an equivalent submultiplicative ultra-pseudo-seminorm $\mathcal P'$ on $\Alg$ with ${\mathcal P}'(1)=1$.
\end{lemma}
\begin{proof}
We let ${\mathcal P}'(u)=\sup_{\pseudonorm v = 1}\pseudonorm{uv}$. If $v\in\Alg$ and $\pseudonorm{v}\ne 0$, then there exists $r\in\R$ such that $\pseudonorm{\caninf^r v} =1$. Thus (as classically) $\pseudonorm{uv}\le\pseudonormprime u\pseudonorm v$, $\forall u,v\in\Alg$. Further, $\pseudonorm{uvw}\le \pseudonormprime u \pseudonormprime v\pseudonorm w$, $\forall u,v,w\in\Alg$, hence $\pseudonormprime{uv}\le \pseudonormprime u \pseudonormprime v$, $\forall u,v \in\Alg$. Clearly, $\pseudonormprime 1=1$. The equivalence of the ultra-pseudo-seminorms follows since $\pseudonorm u\le \pseudonorm 1 \pseudonormprime u$, $\forall u\in \Alg$ and $\pseudonormprime u\le C \pseudonorm u$, $\forall u \in\Alg$.
\end{proof}

\begin{prop}
Let $\Alg$ be a $\GenC$-algebra provided with an ultra-pseudo-norm $\pseudonorm{.}$ for which the ring multiplication is $\cdot$: $\Alg\times\Alg\to\Alg$ is continuous. Then there exists an equivalent ultra-pseudo-norm $\pseudonormprime{.}$ on $\Alg$ such that $(\Alg,\pseudonormprime{.})$ is an ultra-pseudo-normed $\GenC$-algebra.
\end{prop}
\begin{proof}
As the ring multiplication is a $\GenC$-bilinear map, the continuity implies that there exists $C\in\R$ such that $\pseudonorm{uv}\le C\pseudonorm{u}\pseudonorm{v}$, $\forall u,v\in\Alg$. The result follows then by lemma \ref{lemma_seminorm_image_of_1}.
\end{proof}

\begin{prop}\label{prop_completion_of_normed}
Let $\Alg$ be an ultra-pseudo-normed $\GenC$-algebra. Then the completion $\Completion\Alg$ is a Banach $\GenC$-algebra. If $\Alg$ is faithful, then also $\Completion\Alg$ is faithful.
\end{prop}
\begin{proof}
As explained in section \ref{section_prelims}, $\Completion\Alg$ is a Banach $\GenC$-module. The submultiplicativity of the pseudonorm $\pseudonorm{.}$ ensures that the ring multiplication can be extended to $\Completion\Alg$. In this way, $\Completion\Alg$ becomes a $\GenC$-algebra. By continuity, also the extended $\pseudonorm{.}$ is submultiplicative.
\end{proof}

\begin{df}
Let $(\Gen,\pseudonorm{.})$ be a Banach $\GenC$-module. Let $(u_\lambda)_{\lambda\in\Lambda}$ be a family of elements of $\Gen$. We say that $(u_\lambda)_{\lambda\in\Lambda}$ is \defstyle{summable} if there exists $v\in\Gen$ such that for each $r\in\R^+$ there exists a finite $F_0\subseteq\Lambda$ such that $\pseudonorm{v - \sum_{\lambda\in F} u_\lambda}\le r$, for each finite $F\subseteq \Lambda$ with $F_0\subseteq F$. Since $v$ is unique with this property, $\sum_{\lambda\in \Lambda} u_\lambda:= v$.
\end{df}

\begin{lemma}\label{lemma_convergence_Banach_module}
Let $(\Gen,\pseudonorm{.})$ be a Banach $\GenC$-module.
\begin{enumerate}
\item A family $(u_\lambda)_{\lambda\in\Lambda}$ of elements of $\Gen$ is summable iff for each $r\in\R^+$, there exists a finite $F\subseteq\Lambda$ such that for each $\lambda\in\Lambda\setminus F$, $\pseudonorm{u_\lambda}<r$. In that case:
\begin{enumerate}
\item there is a countable subset $\Lambda_0$ such that $u_\lambda = 0$ for each $\lambda\in\Lambda\setminus\Lambda_0$.
\item $\sum_{\lambda\in\Lambda} u_\lambda = \sum_{n\in\N} u_{\lambda_n}$ for any enumeration $(\lambda_n)_{n\in\N}$ of $\Lambda_0$.
\item $\pseudonorm{\sum_{\lambda\in\Lambda} u_\lambda}\le \max_{\lambda\in\Lambda} \pseudonorm{u_\lambda}$.
\end{enumerate}
\item Let $u_n\in \Gen$, for each $n\in\N$. Then $\sum_{n\in\N} u_n$ converges iff $(u_n)_{n\in\N}$ is summable iff $u_n\to 0$. If $\sum_{n\in\N} u_n$ converges, then $\sum_{n\in\N} u_{\sigma(n)} = \sum_{n\in\N} u_{n}$, for each bijection $\sigma$: $\N\to\N$, and $\pseudonorm{\sum_{n\in\N} u_n}\le \max_{n\in\N} \pseudonorm{u_n}$.
\item Let $u_{n,m}\in\Gen$, for each $n,m\in\N$ and let the family $(u_{n,m})_{n,m\in\N}$ be summable (we denote the sum of the family by $\sum_{n,m\in\N} u_{n,m}$). Then $\sum_{m\in\N} \sum_{n\in\N} u_{n,m} = \sum_{n\in\N} \sum_{m\in\N} u_{n,m} = \sum_{n,m\in\N} u_{n,m}$.
\item Let $\Gen$ be a Banach $\GenC$-algebra. Let $u_n$, $v_n\in\Gen$, for each $n\in\N$ and let $u_n, v_n\to 0$. Then $\big(\sum_{n\in\N} u_n\big)\big(\sum_{n\in\N} v_n\big) = \sum_{n,m\in\N}u_n v_m$.
\end{enumerate}
\end{lemma}
\begin{proof}
(1) As for Banach spaces over ultrametric fields \cite[Ex.\ 3.K]{VanRooij}.\\
(2) By the ultrametric property of $\pseudonorm{.}$, $\sum_{n\in\N} u_n$ converges iff $(u_n)_{n\in\N}$ is summable. The other assertions follow by part~1.\\
(3) By parts 1 and 2, all series involved are convergent. Further, the three expressions converge to $\lim_{N\to\infty} \sum_{n,m \le N} u_{n,m}$.\\
(4) By parts 1 and 2, all series involved are convergent. As the product in $\Gen$ is continuous, both expressions are equal to $\lim_{N\to\infty} \sum_{n,m\le N}u_n v_m$.
\end{proof}

\begin{lemma}\label{lemma_invertible_in_Banach_algebra}
Let $(\Alg, \pseudonorm{.})$ be a Banach $\GenC$-algebra. Let $u\in\Alg$ with $\pseudonorm u<1$. Then $v=\sum_{n=0}^\infty u^n$ exists in $\Alg$ and $v=\inv{(1-u)}$.
\end{lemma}
\begin{proof}
As for Banach $\C$-algebras (e.g.\ \cite[Lemma 3.1.5]{Kadison}).
\end{proof}
%\begin{cor}
%Let $(\Alg, \pseudonorm{.})$ be a Banach $\GenC$-algebra. Let $a\in \Alg$. If there exists $b\in\Alg$ such that $\pseudonorm{b a-1}< 1$, then $a$ has a left inverse in $\Alg$. (Similarly for right inverses.)
%\end{cor}

\begin{prop}\label{prop_invertible_in_Banach_algebra}
Let $(\Alg, \pseudonorm{.})$ be a Banach $\GenC$-algebra.
\begin{enumerate}
\item The set of invertible elements of $\Alg$ is open.
%\item Every maximal ideal is closed.
\item Denote by $Q$ the set of invertible elements of $\Alg$. The map $\inv{.}$: $Q\to Q$ is continuous.
\item Let $S\subseteq(0,1)$ with $e_S1\ne 0$ and $u\in\Alg$ with $\pseudonorm{e_S u}<1$. Then $1-u$ is invertible w.r.t.\ $S$.
\item Let $S\subseteq(0,1)$ with $e_S1\ne 0$. Then $\{u\in\Alg: u$ invertible w.r.t.\ $S\}$ is open.
\end{enumerate}
\end{prop}
\begin{proof}
(1), (2) As for Banach $\C$-algebras (e.g.\ \cite[Prop.\ 3.1.6]{Kadison}).\\
(3) By lemma \ref{lemma_invertible_in_Banach_algebra}, $1-e_Su =e_S(1-u)+e_{\co S}1$ is invertible. So by lemma \ref{lemma_invertible_wrt_S}, $1-u$ is invertible w.r.t.\ $S$.\\
(4) By lemma \ref{lemma_invertible_wrt_S}, $\{u\in\Alg: u$ invertible w.r.t.\ $S\} = \inv\Phi(Q)$, where $Q$ is the set of invertible elements of $\Alg$ and $\Phi(u)= e_S u + e_{\co S}1$. The result follows by part 1 and the continuity of $\Phi$.
\end{proof}

%\begin{prop}
%Let $(\Alg,\pseudonorm{.})$ be a Banach $\GenC$-algebra. Let $u, u_n\in \Alg$ with $u_n\to u$ and $u_n$ invertible in $\Alg$. Then $u$ is invertible in $\Alg$ iff $\sup_{n\in\N}\pseudonorm{\inv u_n}<+\infty$.
%\end{prop}
%\begin{proof}
%$\Rightarrow$: if $u$ is invertible, then $\inv u_n\to \inv u$ by proposition \ref{prop_invertible_in_Banach_algebra}.\\
%$\Leftarrow$: let $\sup_{n\in\N}\pseudonorm{\inv u_n}\le C$. Then $\pseudonorm{1 - u \inv u_n}\le \pseudonorm{u_n - u}\pseudonorm{\inv u_n}\le C\pseudonorm{u_n - u}<1$ for large $n$. By lemma \ref{lemma_invertible_in_Banach_algebra}, $u\inv u_n$ is invertible. Hence also $u$ is invertible.
%\end{proof}

\begin{lemma}\label{lemma_closed_subalgebra_in_Banach_algebra}
Let $(\Alg, \pseudonorm{.})$ be a Banach $\GenC$-algebra and let ${\mathcal B}\subseteq \Alg$ be a closed sub-$\GenC$-algebra (with $1$). Then $({\mathcal B}, \pseudonorm{.})$ is a Banach $\GenC$-algebra. If $\Alg$ is faithful, then also $\mathcal B$ is faithful.
\end{lemma}
\begin{proof}
Elementary.
\end{proof}

\begin{lemma}\label{lemma_GenC_proper_ideal}
Let $\Alg$ be a faithful $\GenC$-algebra. Let $I\idealproper\Alg$. The following are equivalent:
\begin{enumerate}
\item $I\cap\GenC =\{0\}$
\item $\Alg/I$ is a faithful $\GenC$-algebra
\item $(\forall S\subseteq (0,1) \text{ with }e_S\ne 0)(e_S\notin I)$
\item $(\forall S\subseteq(0,1) \text{ with }e_S\ne 0) (I\cap \{u\in \Alg: u\text{ is invertible w.r.t.\ }S\}=\emptyset)$.
\end{enumerate}
\end{lemma}
\begin{proof}
$(1)\implies (2)$: if $\lambda\in\GenC$ and $\lambda 1=0$ in $\Alg/I$, then $\lambda 1\in I\cap\GenC$, hence $\lambda 1=0$ in $\Alg$. As $\Alg$ is faithful, this implies that $\lambda=0$.\\
$(2)\implies (3)$: if there exists $S\subseteq(0,1)$ with $e_S\ne 0$ such that $e_S\in I$, then $e_S 1=0$ in $\Alg/I$, and $\Alg/I$ is not faithful (as $\Alg$ is faithful, $e_S\ne 0$ in $\GenC$ iff $e_S\ne 0$ in $\Alg$).\\
$(3)\implies (4)$: let $u\in\Alg$ be invertible w.r.t.\ $S$. Should $u\in I$, then also $0\ne e_S\in I$.\\
$(4)\implies (1)$: let $\lambda\in\GenC\setminus\{0\}$. Then there exists $S\subseteq(0,1)$ with $e_S\ne 0$ (also in $\Alg$, as $\Alg$ is faithful) such that $\lambda$ is invertible w.r.t.\ $S$, so $\lambda\notin I$.
\end{proof}

The following proposition is an extension of \cite[1.12]{Garetto2005}.
\begin{prop}\label{prop_closed_ideal_in_Banach_algebra}
Let $(\Alg, \pseudonorm{.})$ be a Banach $\GenC$-algebra and let $I\idealproper \Alg$ be closed. Then $(\Alg/I, \pseudonormprime{.})$ is a Banach $\GenC$-algebra with
\[
\pseudonormprime{u+I}= \inf_{u-v\in I} \pseudonorm v.
\]
If $\Alg$ is faithful, then $\Alg/I$ is faithful iff $I\cap\GenC = 0$.
\end{prop}
\begin{proof}
As for Banach $\C$-algebras (e.g., \cite[Thm.\ 1.5.3, Prop.\ 3.1.8]{Kadison}).
%In fact, completeness follows more easily using the ultrametric property: if $(\bar u_n)_{n}$ is Cauchy in $\Alg/I$, w.l.o.g.\ $\pseudonorm{\bar u_{n+1} - \bar u_n}\to 0$ (by considering a subsequence). Thus we can find representatives $u_n\in\Alg$ with $u_{n+1}-u_n\to 0$. By lemma \ref{lemma_convergence_Banach_module}, $\sum_n u_{n+1 - u_n}$ converges, i.e., $v:= \lim u_n$ exists in $\Alg$. As $\pseudonormprime{\bar u}\le \pseudonorm u$, also $\bar v = \lim \bar u_n$.
$\pseudonormprime{1+I}=1$ by lemma \ref{lemma_invertible_in_Banach_algebra}, since $I$ is a proper ideal.
Faithfulness follows by lemma \ref{lemma_GenC_proper_ideal}.
\end{proof}

The following proposition is an easy extension of \cite[Props.\ 3.17, 3.19]{Garetto2005}:
\begin{prop}\label{prop_lin_maps_on_normed_modules}
Let $(\Gen_1,\mathcal P_1)$, $(\Gen_2,\mathcal P_2)$ be ultra-pseudo-normed $\GenC$-modules. Then $\cLin(\Gen_1,\Gen_2)$ is an ultra-pseudo-normed $\GenC$-module with ultra-pseudo-norm $\pseudonorm{.}$ defined by
\begin{align*}
\pseudonorm{T} &= \inf\{C\in\R, C\ge 0: (\forall u\in\Gen_1) (\mathcal P_2(T u)\le C \mathcal P_1(u))\}\\
&= \sup_{\mathcal P_1(u)\le 1}\mathcal P_2(T u) = \sup_{\mathcal P_1(u) = 1}\mathcal P_2(T u) = \sup_{u \ne 0} \frac{\mathcal P_2(T u)}{\mathcal P_1(u)},
\end{align*}
and $\mathcal P_2(Tu)\le \pseudonorm T \mathcal P_1(u)$, $\forall u\in\Gen_1$.\\
If $(\Gen_2,\mathcal P_2)$ is complete, then also $(\cLin(\Gen_1,\Gen_2), \pseudonorm{.})$ is complete.
\end{prop}

\begin{prop}\label{prop_cLin_of_Banach_module}
Let $(\Gen,\pseudonorm{.})$ be a Banach $\GenC$-module. Then $(\cLin(\Gen),\pseudonorm{.})$ is Banach $\GenC$-algebra. If $\Gen$ is a faithful $\GenC$-module, then $\cLin(\Gen)$ is a faithful $\GenC$-algebra.
\end{prop}
\begin{proof}
As for Banach $\C$-algebras. If $\lambda 1= 0$ in $\cLin(\Gen)$, then $\lambda u = 0$, $\forall u\in\Gen$, so $\lambda = 0$ in $\GenC$ if $\Gen$ is faithful.
\end{proof}

\begin{prop}\label{prop_G_L(B)_is_Banach_subalgebra}
Let $B$ be a (classical) Banach space and $\cLin(B)$ the Banach $\C$-algebra of continuous $\C$-linear operators on $B$. Then $\Gen_{\cLin(B)}$ is isomorphic with a Banach $\GenC$-subalgebra of $\cLin(\Gen_B)$.
\end{prop}
\begin{proof}
As $\cLin(B)$ is a Banach $\C$-algebra, $\Gen_{\cLin(B)}$ is a Banach $\GenC$-algebra with ultra-pseudonorm $\pseudonorm T=\sharpnorm[\big]{\norm{T_\eps}}$, for $T=[(T_\eps)_\eps]\in\Gen_{\cLin(B)}$. By \cite[\S 1.1.2]{GV_Hilbert}, $(T_\eps)_\eps\in \Mod_{\cLin(B)}$ induces a well-defined (so-called basic) continuous $\GenC$-linear map $T$: $\Gen_B\to\Gen_B$ by means of $T(u)=[(T_\eps u_\eps)_\eps]$, for $u=[(u_\eps)_\eps]\in \Gen_B$. Further, as in \cite[Prop.\ 3.22]{Garetto2005}, one sees that $(T_\eps)_\eps\in\Null_{\cLin(B)}$ iff $(T_\eps)_\eps$ induces the $0$-map on $\Gen_B$. Hence we can identify $T\in\Gen_{\cLin(B)}$ with an element of $\cLin(\Gen_B)$. Clearly, the algebraic operations on $\Gen_{\cLin(B)}$ coincide with those on $\cLin(\Gen_B)$. We show that $\pseudonorm T =\mathcal P_{\cLin(\Gen_B)} (T)$. For $u\in\Gen_B$, $\pseudonorm{Tu}= \sharpnorm{\norm{T_\eps u_\eps}_B}\le \sharpnorm{\norm{T_\eps}}\sharpnorm{\norm{u_\eps}_B} = \pseudonorm T \pseudonorm u$, so $\mathcal P_{\cLin(\Gen_B)} (T)\le \pseudonorm{T}$. Conversely, for each $\eps\in (0,1)$, there exists $u_\eps\in B$ with $\norm{u_\eps} = 1$ and $\norm{T_\eps u_\eps}\ge \norm{T_\eps}-\eps^{1/\eps}$. Hence $u:=[(u_\eps)_\eps]\in\Gen_B$, $\pseudonorm u= 1$ and $\pseudonorm T\le \pseudonorm{Tu} \le \sup_{\pseudonorm u =1}\pseudonorm{Tu} =\mathcal P_{\cLin(\Gen_B)} (T)$.
\end{proof}

\section{Spectral values}
\subsection{Definition and elementary properties}
\begin{df}\label{df_spectrum}
Let $\Alg$ be a faithful $\GenC$-algebra. An element $u\in\Alg$ is called \defstyle{strictly non-invertible} if
\[(\forall S\subseteq(0,1) \text{ with } e_S \ne 0) (u \text{ is not invertible w.r.t.\ }S).\]
An element $\lambda\in\GenC$ is called a \defstyle{spectral value} for $u\in\Alg$ if $u-\lambda$ is strictly non-invertible.\\
The set of all spectral values of $u\in\Alg$ is called the \defstyle{spectrum} of $u$ and is denoted by $\spec_\Alg(u)$ (or by $\spec(u)$ if the algebra is clear from the context).\\
The set of all $\lambda\in\GenC$ such that $u-\lambda$ is invertible in $\Alg$ is called the \defstyle{resolvent set} of $u$ and is denoted by $\resolv_\Alg(u)$.
\end{df}
\begin{rem}
By definition, $\spec(u)\cap\resolv(u)=\emptyset$. But $\interl(\spec(u)\cup\resolv(u))\subsetneqq\GenC$ in general. E.g., by lemma \ref{lemma_invertible_wrt_S}(4), $\spec(0)=\{0\}$ and $\resolv(0)=\{\lambda\in\GenC: \lambda$ is invertible$\}$. If $\lambda\GenC$ is not generated by an idempotent (such $\lambda$ exists by \cite{AJOS2006}), then $\lambda\notin\interl(\spec(0)\cup\resolv(0))$. Since $\{\abs\lambda\in\GenR: \lambda\in\resolv(0)\}$ does not reach a minimum, $\resolv(0)$ is also not an internal set of $\GenC$ \cite[Cor.\ 2.6]{OV_internal}.
\end{rem}
We can motivate the definition of the spectrum by two classical examples of Banach $\GenC$-algebras.

Let $X$ be a compact topological space. Since $\Cnt(X)$ (with the sup-norm) is a commutative Banach $\C$-algebra, $\Gencnt X$ is a faithful commutative Banach $\GenC$-algebra.
\begin{df}
For $u = [(u_\eps)_\eps]\in\Gencnt X$ and $x=(x_\eps)_\eps\in X^{(0,1)}$, the \defstyle{point value} $u(x)$ is defined as $[(u_\eps(x_\eps))_\eps]\in\GenC$ (independent of the representative $(u_\eps)_\eps$).
\end{df}
Classically, for $u\in\Cnt(X)$, $\spec(u)=\{u(x): x\in X\}$.

\begin{prop}\label{prop_invertible_in_Gencnt}
Let $X$ be a compact topological space, $u=[(u_\eps)_\eps]\in\Gencnt X$ and $S\subseteq (0,1)$ with $e_S\ne 0$. Then the following are equivalent:
\begin{enumerate}
\item $u$ is invertible w.r.t.\ $S$ in $\Gencnt X$ (i.e., $\exists v\in\Gencnt{X}$ such that $uv= e_S$).
\item $u(x)$ is invertible w.r.t.\ $S$ in $\GenC$, for each $x\in X^{(0,1)}$.
\item There exists $m\in\N$ such that $\inf_{x\in X} \abs{u_\eps(x)} \ge \eps^m$, for small $\eps\in S$.
\item $[(\inf_{x\in X} \abs{u_\eps(x)})_\eps]$ is invertible w.r.t.\ $S$ in $\GenC$.
\end{enumerate}
\end{prop}
\begin{proof}
Let first $S=(0,1)$.\\
$1 \Rightarrow 2$: if $uv=1$, then $u(x)v(x)=1$, for each $x\in X^{(0,1)}$.\\
$2 \Rightarrow$ 3: supposing that the conclusion is not true, we find a decreasing sequence $(\eps_n)_{n\in\N}$ tending to $0$ and $x_{\eps_n}\in X$ with $\abs{u_{\eps_n}(x_{\eps_n})}<\eps_n^n$, for each $n\in\N$. Choose $x_\eps\in X$, if $\eps\notin\{\eps_n: n\in\N\}$. Then $x:=(x_\eps)_\eps\in X^{(0,1)}$, but $u(x)$ is not invertible in $\GenC$ by lemma \ref{lemma_invertible_wrt_S}(4).\\
$3 \Rightarrow 1$: let $v_\eps\in\Cnt(X)$ with $v_\eps(x)=1/u_\eps(x)$, for small $\eps$. Since $\sup_{x\in X}\abs{v_\eps(x)}\le \eps^{-m}$, $v=[(v_\eps)_\eps]\in\Gencnt X$ and $uv=1$.\\
For arbitrary $S$, the equivalences then follow by lemma \ref{lemma_invertible_wrt_S}(2). Finally, $(3)\iff(4)$ by lemma \ref{lemma_invertible_wrt_S}(4).
\end{proof}

\begin{prop}\label{prop_spec_Gencbd}
Let $X$ be a compact topological space and $u\in\Gencnt X$. Then
\[\spec(u)=\{u(x): x\in X^{(0,1)}\}.\]
\end{prop}
\begin{proof}
$\subseteq$: let $\lambda\in\spec(u)$. Then for each $S\subseteq (0,1)$ with $e_S\ne 0$, $u-\lambda$ is not invertible w.r.t.\ $S$. By proposition \ref{prop_invertible_in_Gencnt}, $[(\inf_{x\in X}\abs{u_\eps(x)-\lambda_\eps})_\eps]$ is not invertible w.r.t.\ $S$, for each $S\subseteq(0,1)$. Hence $[(\inf_{x\in X}\abs{u_\eps(x)-\lambda_\eps})_\eps]=0$ in $\GenC$. Choosing $x_\eps\in X$ such that $(\abs{u_\eps(x_\eps)-\lambda_\eps})_\eps$ is negligible, we find $x:=(x_\eps)_\eps\in X^{(0,1)}$ with $\lambda = u(x)$.\\
$\supseteq$: let $S\subseteq (0,1)$ with $e_S\ne 0$ and $x\in X^{(0,1)}$. Then $u-u(x)$ is not invertible w.r.t.\ $S$ by proposition \ref{prop_invertible_in_Gencnt}, since its point value at $x$ is not invertible w.r.t.\ $S$.
\end{proof}

Classically, for $A\in\C^{d\times d}$, $\spec(A)$ is the set of eigenvalues of $A$.
\begin{prop}\label{prop_matrices_elementary}\leavevmode
\begin{enumerate}
\item $\GenC^{d\times d}$ is a Banach $\GenC$-algebra for the usual matrix operations and the ultra-pseudo-norm $\pseudonorm{(a_{i,j})_{i,j=1,\dots,d}} := \max_{i,j=1,\dots,d}\sharpnorm{a_{i,j}}$.
\item $\GenC^{d\times d}\cong \Gen_{\C^{d\times d}} \cong \cLin(\GenC^d)$ as Banach $\GenC$-algebras.
\end{enumerate}
Every $\GenC$-linear map $\GenC^d\to\GenC^d$ is continuous.
\end{prop}
\begin{proof}
Clearly, $\GenC^{d\times d}$ is a $\GenC$-algebra. Let $(A_{\eps})_{\eps}\in(\C^{d\times d})^{(0,1)}$ with $A_{\eps} = (a_{i,j,\eps})_{i,j=1,\dots,d}$, for each $\eps$. Since all norms on $\C^{d\times d}$ are equivalent, $(A_{\eps})_{\eps}\in\Mod_{\C^{d\times d}}$ iff there exists $N\in\N$ such that $\max_{i,j=1,\dots,d} \abs{a_{i,j,\eps}} \le \eps^{-N}$, for small $\eps$ iff $(a_{i,j,\eps})_{\eps} \in\Mod_{\C}$, for $i,j=1,\dots, d$. Similarly, $(A_{\eps})_{\eps}\in\Null_{\C^{d\times d}}$ iff $(a_{i,j,\eps})_{\eps} \in\Null_{\C}$, for $i,j=1,\dots, d$. Thus the map $\Gen_{\C^{d\times d}}\to\GenC^{d\times d}$: $[(A_{\eps})_{\eps}] \mapsto ([(a_{i,j,\eps})_{\eps}])_{i,j=1,\dots,d}$ is a well-defined algebraic isomorphism. It also preserves the ultra-pseudo-norm, since
\[
\pseudonorm{[(A_{\eps})_{\eps}]} = \sharpnorm[\big]{\norm{A_{\eps}}} = \sharpnorm[\big]{\max_{i,j=1,\dots,d}\abs{a_{i,j,\eps}}} = \max_{i,j=1,\dots,d}\sharpnorm{a_{i,j,\eps}} = \pseudonorm{([(a_{i,j,\eps})_{\eps}])_{i,j=1,\dots,d}}.
\]
(Notice that $\sharpnorm[\big]{\norm{A_{\eps}}}$ does not depend on the chosen equivalent norm on $\C^{d\times d}$.) By the isomorphism, $\GenC^{d\times d}$ is a Banach $\GenC$-algebra and $\pseudonorm{(a_{i,j})_{i,j=1,\dots,d}}$ is submultiplicative.\\
Let $\delta_1 = (1,0,\dots,0)$, \dots, $\delta_d = (0,\dots,0,1)$.
If $f$: $\GenC^d\to\GenC^d$ is $\GenC$-linear, then $f(\tilde z_1,\dots, \tilde z_d) = (\sum_{j=1}^d a_{ij} \tilde z_j)_{i=1,\dots,d}$, where $(a_{1j},\dots, a_{dj}) := f(\delta_j)$ by $\GenC$-linearity. In particular, $f(\tilde z) = [(A_\eps z_\eps)_\eps]$, for each $\tilde z=[(z_\eps)_\eps]\in\GenC^d$, with $A_\eps z = \sum_{j=1}^d a_{i,j,\eps} z_j$. By the isomorphism between $\GenC^{d\times d}$ and $\Gen_{\C^{d\times d}}$, $(A_\eps)_\eps\in\Mod_{\C^{d\times d}}$, so $(\norm{A_\eps})_\eps$ is moderate for the operator norm on $\C^{d\times d}$. Hence $f$ is a so-called basic $\GenC$-linear map, and therefore continuous \cite[\S 1.1.2]{GV_Hilbert}.
Conversely, for $A\in\GenC^{d\times d}$, the map $\GenC^d\to\GenC^d$: $\tilde z\mapsto (\sum_{j=1}^d a_{ij} \tilde z_j)_{i=1,\dots,d}$ is $\GenC$-linear. This defines an algebraic isomorphism $\cLin(\GenC^d)\to \GenC^{d\times d}$. By proposition \ref{prop_G_L(B)_is_Banach_subalgebra}, the $\cLin(\GenC^d)$-ultra-pseudo-norm of $f$ coincides with $\sharpnorm[\big]{\norm{A_\eps}}$, so the isomrophism also preserves the ultra-pseudo-norm.
\end{proof}

For $A = (a_{ij})_{i,j}\in\GenC^{d\times d}$, $\det A := [(\det A_\eps)_\eps] = \sum_{\sigma} a_{1,\sigma(1)}\cdots a_{d,\sigma(d)} \in\GenC$, where $\sigma$ runs over all permutations of $\{1,\dots,d\}$.
\begin{prop}\label{prop_invertible_matrix}
Let $A\in\GenC^{d\times d}$ and $S\subseteq (0,1)$ with $e_S\ne 0$. Then the following are equivalent:
\begin{enumerate}
\item $A$ is invertible w.r.t.\ $S$ in $\GenC^{d\times d}$
\item $(\forall \tilde z\in\GenC^d)$ $(e_S A \tilde z=0  \implies e_S \tilde z = 0)$
\item $\det A$ is invertible w.r.t.\ $S$ in $\GenC$.
\end{enumerate}
\end{prop}
\begin{proof}
$(1)\implies (2)$: if $e_S A \tilde z = 0$, then by multiplying with an inverse of $A$ w.r.t.\ $S$, $e_S \tilde z = 0$.\\
$(2)\implies (3)$: let $B := e_S A + e_{\co S}$. If $B \tilde z = 0$, then $e_S A \tilde z = 0$, so by assumption $e_S \tilde z = 0$, and $e_{\co S} \tilde z = 0$. Hence $\tilde z = 0$. By \cite[Lemma 1.2.41]{GKOS}, $\det B = e_S \det A + e_{\co S}$ is invertible in $\GenC$. By lemma \ref{lemma_invertible_wrt_S}(2), $\det A$ is invertible w.r.t.\ $S$.\\
$(3)\implies (1)$: let $A = [(A_\eps)_\eps]$ and let $b\in\GenC$ with $b \cdot \det A = e_S$. Then $\det A_\eps \ne 0$ for small $\eps\in S$ by lemma \ref{lemma_invertible_wrt_S}(4). Hence $A_\eps C_\eps = C_\eps A_\eps = \det A_\eps$, for small $\eps\in S$, where $C_\eps$ is the cofactor matrix of $A_\eps$. Now $C := [(C_\eps)_\eps]\in \GenC^{d\times d}$ is well-defined, since the entries of $C$ are obtained as determinants of matrices in $\GenC^{(d-1)\times(d-1)}$. Hence $e_S A C = e_S C A = e_S \det A$, and $b e_S C$ is an inverse of $A$ w.r.t.\ $S$.
\end{proof}

\begin{lemma}\label{lemma_polynomial_eqn}
Let $a_0$, \dots, $a_{n-1}$ $\in\GenC$. Then there exist $\lambda_1$, \dots, $\lambda_n$ $\in\GenC$ such that $p(z):=z^n+a_{n-1}z^{n-1}+\cdots+a_1 z + a_0 = (z-\lambda_1)\cdots(z-\lambda_n)$. The solution set in $\GenC$ of the equation $p(z)=0$ is the sharply bounded internal set (with $\lambda_j=[(\lambda_{j,\eps})_\eps]$)
\begin{multline*}
[(\{\lambda_{1,\eps},\dots,\lambda_{n,\eps}\})_\eps]
= \interl(\{\lambda_1,\dots,\lambda_n\})\\
= \{\lambda_1 e_{S_1} + \cdots + \lambda_n e_{S_n}: \{S_1,\dots, S_n\}\text{ is a partition of }(0,1)\}.
\end{multline*}
\end{lemma}
\begin{proof}
Fix representatives $(a_{j,\eps})_\eps$ of $a_j$. For each $\eps$, let $p_\eps(z) = z^n + a_{n-1,\eps} z^{n-1} + \cdots + a_{1,\eps} z + a_{0,\eps} = (z-\lambda_{1,\eps})\cdots (z-\lambda_{n,\eps})$. If $\abs{\lambda_{j,\eps}}\ge 1$, then $\abs{\lambda_{j,\eps}}^n\le \abs{a_{n-1,\eps}} \abs{\lambda_{j,\eps}}^{n-1}+ \cdots + \abs{a_{0,\eps}}\le (\abs{a_{n-1,\eps}}+ \cdots + \abs{a_{0,\eps}}) \abs{\lambda_{j,\eps}}^{n-1}$. So always $\abs{\lambda_{j,\eps}}\le \abs{a_{n-1,\eps}}+ \cdots + \abs{a_{0,\eps}} + 1$, and $(\lambda_{j,\eps})\in\Mod_\C$. Let $\lambda_j:= [(\lambda_{j,\eps})_\eps]\in\GenC$. Then $p(z) = (z-\lambda_1)\cdots (z-\lambda_n)$.\\
Now let $z$ be any solution of the equation. By \cite[lemma~2.3]{HV_ideals}, there exists $S_1\subseteq (0,1)$ such that $(z-\lambda_1)e_{S_1}=0$ and $(z-\lambda_2)\cdots(z-\lambda_n)e_{\co S_1}=0$. By the same lemma, we find $T\subseteq (0,1)$ and $S_2:=T\setminus S_1$ such that $(z-\lambda_2)e_{\co S_1}e_T = (z-\lambda_2)e_{S_2} = 0$ and $(z-\lambda_3)\cdots(z-\lambda_n)e_{\co S_1} e_{\co T} = (z-\lambda_3)\cdots(z-\lambda_n)e_{\co{(S_1\cup S_2)}} = 0$. Inductively, we find a partition $\{S_1,\dots, S_n\}$ of $(0,1)$ such that $z e_{S_j}=\lambda_j e_{S_j}$, for $j=1,\dots, n$. Hence $z=z(e_{S_1}+\cdots e_{S_n})=\lambda_1 e_{S_1} + \cdots + \lambda_n e_{S_n}$.
\end{proof}

\begin{df}
Let $A\in\GenC^{d\times d}$. Then $\tilde z\in\GenC^d$ is an eigenvector for $A$ with eigenvalue $\lambda\in\GenC$ iff $\abs{\tilde z}\gg 0$ and $A \tilde z = \lambda \tilde z$.
\end{df}

\begin{prop}\label{prop_spec_matrix}
Let $A\in\GenC^{d\times d}$. Then
\[
\spec(A)=\{\lambda\in\GenC: \det(A-\lambda)=0\}=\{\lambda\in\GenC: \lambda \text{ is an eigenvalue for }A\}.
\]
\end{prop}
\begin{proof}
By proposition \ref{prop_invertible_matrix}, $\lambda\in\spec(A)$ iff $A-\lambda$ strictly not invertible in $\GenC^{d\times d}$ iff $\det(A-\lambda)$ strictly not invertible in $\GenC$ iff $\det(A-\lambda)= 0$ in $\GenC$.\\
If $\tilde z$ is an eigenvector for $A$ with eigenvalue $\lambda$, then for each $S\subseteq (0,1)$ with $e_S\ne 0$, $e_S \tilde z\ne 0$ and $e_S(A-\lambda)\tilde z = 0$, hence $A-\lambda$ is strictly not invertible by proposition \ref{prop_invertible_matrix}.\\
Let $A=[(A_\eps)_\eps]$. If $\lambda\in\GenC$ is a root of $\det(A-\lambda) = 0$, then by lemma \ref{lemma_polynomial_eqn} (and its proof), there exists a representative $(\lambda_\eps)_\eps$ of $\lambda$ such that $\det(A_\eps - \lambda_\eps) = 0$, for each $\eps$. Thus there exist $z_\eps\in\C^d$ with $\abs{z_\eps}=1$ and $A_\eps z_\eps = \lambda_\eps z_\eps$, for each $\eps$. Then $\tilde z:= [(z_\eps)_\eps]\in\GenC^d$ is an eigenvector for $A$ with eigenvalue $\lambda$.
\end{proof}

For many $\GenC$-algebras, $\spec(u)$ can alternatively be defined by means of $\resolv(u)$.
\begin{lemma}\label{lemma_far_away_set}
Let $\K$ be $\R$ or $\C$. Let $A\subseteq\GenK^d$. Then $\{\tilde x\in\GenK^d:$  $\abs{\tilde x - \tilde a}\gg 0$, $\forall \tilde a\in A\}$ $=$ $\{\tilde x\in\GenK^d:$ $(\forall S\subseteq (0,1)$ with $e_S\ne 0)$ $(\tilde x e_S\notin Ae_S)\}$.
\end{lemma}
\begin{proof}
$\subseteq$: if $\tilde x e_S = \tilde a e_S$, for some $\tilde a \in A$ and $e_S\ne 0$, then $\abs{\tilde x - \tilde a} = \abs{\tilde x - \tilde a} e_{\co S}\ngeq \caninf^m$, for any $m\in\N$.\\
$\supseteq$: if $\tilde x = [(x_\eps)_\eps]$ and there exists $\tilde a = [(a_\eps)_\eps] \in A$ such that $\abs{\tilde x - \tilde a} \ngeq \caninf^n$, for each $n\in\N$, then we can find a decreasing sequence $(\eps_n)_{n\in\N}$ tending to $0$ with $\abs{x_{\eps_n} - a_{\eps_n}}\le\eps_n^n$, for each $n\in\N$. Hence $\tilde x e_S = \tilde a e_S$ for $S:=\{\eps_n: n\in\N\}$.
\end{proof}

\begin{lemma}\label{lemma_resolv_set}
Let $(\Alg,\pseudonorm{.})$ be a faithful Banach $\GenC$-algebra and $u\in\Alg$.
Let $r\in\R$, $r > 0$ and $\lambda\in\GenC$ with $\abs{\lambda} \ge \caninf^{-\ln \pseudonorm{u} - r}$. Then $\lambda\in\resolv(u)$.
\end{lemma}
\begin{proof}
By lemma \ref{lemma_invertible_wrt_S}(4), $\lambda$ is invertible in $\GenC$.
Since $u- \lambda = -\lambda(1 - \inv\lambda u)$, $u- \lambda$ is invertible by lemma \ref{lemma_invertible_in_Banach_algebra} and the fact that $\pseudonorm{\inv\lambda u} \le \sharpnorm{\inv\lambda}\pseudonorm{u}\le \sharpnorm{\caninf^{\ln \pseudonorm{u} + r}}\pseudonorm{u} = e^{-r} < 1$.
\end{proof}

\begin{thm}\label{thm_resolv_determines_spec}
Let $\Alg$ be a faithful $\GenC$-algebra and $u\in\Alg$. If $\resolv(u)\ne\emptyset$ (e.g., by lemma \ref{lemma_resolv_set}, if $\Alg$ is a Banach $\GenC$-algebra), then
\begin{align*}
\spec(u) &= \{\lambda\in\GenC: \abs{\lambda - \mu}\gg 0, \ \forall \mu\in\resolv(u)\}\\
&= \{\lambda\in\GenC: (\forall S\subseteq (0,1) \text{ with } e_S\ne 0) (\lambda e_S \notin \resolv(u)e_S)\}.
\end{align*}
\end{thm}
\begin{proof}
This result follows by lemma \ref{lemma_far_away_set}, if we show that $\spec(u) = \{\lambda\in\GenC: (\forall S\subseteq (0,1) \text{ with } e_S\ne 0) (\lambda e_S \notin \resolv(u)e_S)\}$.\\
$\subseteq$: if $\lambda e_S = \mu e_S$, for some $\mu\in \resolv(u)$ and $S\subseteq(0,1)$ with $e_S\ne 0$, then $u-\lambda$ is invertible w.r.t.\ $S$.\\
$\supseteq$: let $\lambda\in\GenC\setminus\spec(u)$. There exists $S\subseteq (0,1)$ with $e_S\ne 0$ and $w\in \Alg$ such that $(u-\lambda)w= w (u-\lambda) =e_S$. Let $\mu\in\resolv(u)$ (by assumption, $\resolv(u)\ne\emptyset$). Then $(u-(\lambda e_S + \mu e_{\co S})) (w e_S + (u - \mu)^{-1} e_{\co S}) = (u-\lambda) w e_S + (u-\mu) (u-\mu)^{-1} e_{\co S} = 1$, and similarly for the left inverse. Thus $\lambda e_S\in \resolv(u) e_S$.
\end{proof}

\begin{lemma}\label{lemma_spec_inclusion}
Let $\Alg$ be a faithful $\GenC$-algebra.
\begin{enumerate}
\item If ${\mathcal B}$ is a sub-$\GenC$-algebra of $\Alg$, then $\spec_{\Alg}(u)\subseteq\spec_{\mathcal B}(u)$, $\forall u\in{\mathcal B}$.
\item Let $(\Alg_i)_{i\in I}$ be a chain (i.e., totally ordered for $\subseteq$) of sub-$\GenC$-algebras of $\Alg$. Then $\spec_{\bigcup_{i\in I}\Alg_i}(u)=\bigcap_{\substack{i\in I\\ u\in \Alg_i}}\spec_{\Alg_i}(u)$ and $\resolv_{\bigcup_{i\in I}\Alg_i}(u) = \bigcup_{\substack{i\in I\\ u\in \Alg_i}}\resolv_{\Alg_i}(u)$, $\forall u\in\bigcup_i \Alg_i$.
\item Let $(\Alg_i)_{i\in I}$ be a family of $\GenC$-algebras. Then $\resolv_{\bigcap_{i\in I}\Alg_i}(u) = \bigcap_{i\in I}\resolv_{\Alg_i}(u)$, $\forall u\in \bigcap_{i\in I}\Alg_i$.
\item If $I\ideal\Alg$ and $\Alg$, $\Alg/I$ are faithful, then $\spec_{\Alg/I}(u+I)\subseteq\spec_{\Alg}(u)$, $\forall u\in\Alg$.
\end{enumerate}
\end{lemma}
\begin{proof}
Elementary.
\end{proof}

\begin{lemma}\label{lemma_spec_resolv_interleave}
Let $\Alg$ be a faithful $\GenC$-algebra. Then
\begin{enumerate}
\item $\{u\in\Alg: u$ is invertible$\}$ is closed under finite interleaving.
\item $\{u\in\Alg: u$ is strictly non-invertible$\}$ is closed under finite interleaving.
\item Let $u\in\Alg$. Then $\resolv(u)$ and $\spec(u)$ are closed under finite interleaving.
\end{enumerate}
\end{lemma}
\begin{proof}
(1) Let $u,v\in\Alg$ be invertible and $T\subseteq (0,1)$. Then $\inv u e_T + \inv v e_{\co T}=(e_T u + e_{\co T} v)^{-1}$.\\
(2) Let $u,v\in\Alg$ be strictly non-invertible and $T\subseteq (0,1)$. Suppose $e_T u+e_{\co T}v$ is invertible w.r.t.\ $S$, for some $S\subseteq(0,1)$ with $e_S\ne 0$. Then $e_S e_T \ne 0$ or $e_S e_{\co T} \ne 0$. By symmetry of $(T,u)$ and $(\co T,v)$, we may suppose that $e_{S\cap T} = e_S e_T \ne 0$. As $e_T u+e_{\co T} v$ is invertible w.r.t\ $S\cap T$, also $e_{S\cap T} u$ and thus $u$ are invertible w.r.t.\ $S\cap T$ by lemma \ref{lemma_invertible_wrt_S}, a contradiction.\\
(3) By parts 1~and 2.
\end{proof}

\begin{lemma}\label{lemma_closed_spectrum}
Let $\Alg$ be a faithful Banach $\GenC$-algebra and $u\in\Alg$. Then $\resolv(u)$ is open and $\spec(u)$ is closed (in $\GenC$).
\end{lemma}
\begin{proof}
Let $\lambda\in\resolv(u)$, i.e., $u-\lambda$ is invertible. By proposition \ref{prop_invertible_in_Banach_algebra}(1), also $u-\lambda'$ is invertible, hence $\lambda'\in\resolv(u)$, as soon as $\sharpnorm{\lambda-\lambda'}$ is sufficiently small. It follows that also for $S\subseteq (0,1)$, $\{\tilde z\in\GenC: \tilde z e_S\in \resolv(u) e_S\}$ is open.
Hence $\spec(u)$ is closed by theorem \ref{thm_resolv_determines_spec}.
\end{proof}

\begin{thm}\label{thm_bounded_spectrum}
Let $(\Alg,\pseudonorm{.})$ be a faithful Banach $\GenC$-algebra and $u\in\Alg$. Then $\sharpnorm{\lambda}\le \pseudonorm u$, $\forall\lambda\in \spec_{\Alg}(u)$.
\end{thm}
\begin{proof}
Let $\lambda=[(\lambda_\eps)_\eps]\in\GenC$. Call $r=\sharpnorm{\lambda}$. Let $r>\pseudonorm u$, then in particular, $r>0$ and $\val(\lambda)=-\ln r$. Let $\delta\in\R$, $\delta > 0$ such that $\ln r - \delta \ge \ln{\pseudonorm{u}} + \delta$. By the definition of the valuation on $\GenC$,
\[
(\forall\eta\in(0,1))(\exists\eps\le\eta)(\abs{\lambda_\eps}\ge\eps^{-\ln r + \delta}).
\]
So we can construct a decreasing sequence $(\eps_n)_{n\in\N}$ tending to $0$ with $\eps_n^{-\ln r +\delta}\le\abs{\lambda_{\eps_n}}$, $\forall n\in\N$. Call $T=\{\eps_n:n\in\N\}$. Then $e_T\ne 0$. Further, $\abs{\lambda} e_T \ge \caninf^{-\ln r + \delta} e_T\ge \caninf^{-\ln{\pseudonorm{u}} - \delta} e_T$. Hence there exists $\mu\in\GenC$ with $\mu e_T =\lambda e_T$ and $\mu\in\resolv(u)$ by lemma \ref{lemma_resolv_set}. Thus $\lambda e_T\in\resolv(u) e_T$. By theorem \ref{thm_resolv_determines_spec}, $\lambda\notin\spec(u)$.
\end{proof}

\begin{lemma}\label{lemma_sp_uv_and_sp_vu}
Let $\Alg$ be a $\GenC$-algebra. Let $u,v\in\Alg$ and $S\subseteq(0,1)$ with $e_S 1\ne 0$. Then $1-uv$ is invertible w.r.t.\ $S$ iff $1-vu$ is invertible w.r.t.\ $S$.
\end{lemma}
\begin{proof}
If $1-uv$ is invertible, then a calculation (cf.\ \cite[Prop.\ 3.2.8]{Kadison}) shows that
\[\inv{(1-vu)}=v\inv{(1-uv)}u+1.\]
Further, by lemma \ref{lemma_invertible_wrt_S}, $1-uv$ is invertible w.r.t.\ $S$ iff $e_S(1-uv) + e_{\co S}1 = 1- e_S uv$ is invertible iff (by the first part) $1- e_S vu$ is invertible iff $1-vu$ is invertible w.r.t.\ $S$.
\end{proof}

\begin{prop}
Let $\Alg$ be a faithful $\GenC$-algebra. Let $u,v \in\Alg$. Then
\[
\spec(uv)\cap\{\lambda\in\GenC: \lambda \text{ invertible}\}
=\spec(vu)\cap\{\lambda\in\GenC: \lambda \text{ invertibe}\}.
\]
\end{prop}
\begin{proof}
Let $\lambda\in\spec(uv)$, $\lambda$ invertible. So for each $S\subseteq (0,1)$ with $e_S\ne 0$, $uv-\lambda$, hence also $\inv{\lambda}uv-1$, is not invertible w.r.t.\ $S$. By lemma \ref{lemma_sp_uv_and_sp_vu}, $\inv{\lambda}vu-1$, hence also $vu-\lambda$ is not invertible w.r.t.\ $S$. So $\lambda\in\spec(vu)$.
\end{proof}

\subsection{$\Smlf{\Alg}$ and the Gelfand-Mazur theorem}\label{section_GelfandMazur}
\begin{df}
Let $\Alg$ be a faithful $\GenC$-algebra.\\
We call the \defstyle{functional spectrum} of $u\in\Alg$
\[
\functspec(u)=\{m(u)\in\GenC: m\in\Smlf{\Alg}\}.
\]
\end{df}
\begin{lemma}\label{lemma_fsp_sub_sp}
Let $\Alg$ be a faithful $\GenC$-algebra. Then $\functspec(u)\subseteq\spec(u)$, $\forall u\in\Alg$.
\end{lemma}
\begin{proof}
Let $u\in\Alg$ and $\lambda\notin\spec(u)$. 
So there exists $S\subseteq(0,1)$ with $e_S\ne 0$ and $v\in\Alg$ such that $(u-\lambda)v=e_S$. Let $m\in\Smlf{\Alg}$. Then $m(u-\lambda)m(v)=e_S\ne 0$, hence also $m(u)-\lambda=m(u-\lambda)\ne 0$. So $\lambda\notin\functspec(u)$.
\end{proof}

\begin{prop}\label{prop_cnt_mlf_in_Banach_algebra}
Let $(\Alg,\pseudonorm{.})$ be a faithful Banach $\GenC$-algebra. Let $m\in\Smlf{\Alg}$. Then $\sharpnorm{m(u)}\le\pseudonorm u$, $\forall u\in\Alg$. In particular, $m$ is continuous.
\end{prop}
\begin{proof}
Follows by theorem \ref{thm_bounded_spectrum} and lemma \ref{lemma_fsp_sub_sp}.
\end{proof}

Classically, the Gelfand-Mazur theorem says that a Banach $\C$-algebra $B$ that is also a division algebra (i.e., $0$ is the only element in $B$ that is not invertible), is necessarily isomorphic with $\C$. Under the assumption that spectra of elements in $\Alg$ are not empty, we have the following analogon:

\begin{prop}[Gelfand-Mazur]\label{prop_Gelfand-Mazur}
Let $\Alg$ be a faithful Banach $\GenC$-algebra with $\spec(u)\ne\emptyset$, $\forall u\in \Alg$. %(e.g., $\Alg=\Gen_B$ with $B$ a classical Banach algebra, cf.\ theorem \ref{thm_non_empty_spectra}).
If $0$ is the only element in $\Alg$ that is strictly non-invertible, then $\Alg\cong\GenC$.
\end{prop}
\begin{proof}
As $\Alg$ is faithful, $\GenC\subseteq\Alg$. Let $u\in\Alg$. By assumption, there exists $\lambda\in\spec_\Alg(u)$, i.e., $u-\lambda$ is strictly non-invertible in $\Alg$, and thus $u-\lambda=0$. Hence $\Alg\subseteq \GenC$.
\end{proof}

\begin{df}(cf.\ \cite{HV_isom})
Let $\Alg$ be a $\GenC$-algebra. Let $I\idealproper\Alg$. Then $I$ is \defstyle{maximal with respect to $I\cap\GenC 1=\{0\}$} if $I\cap\GenC 1=\{0\}$ and
\[
J\idealproper\Alg,\ I\subseteq J \text{ and }J\cap\GenC 1 =\{0\}\text{ imply that }I=J.
\]
\end{df}
The Gelfand-Mazur theorem classically implies a bijective correspondence between maximal ideals of a commutative Banach algebra $B$ and non-zero multiplicative linear functionals on $B$. As motivated in \cite{HV_isom}, the maximal ideals have to be replaced in this context by ideals $I$ maximal w.r.t.\ $I\cap \GenC 1= 0$.

\begin{lemma}\label{lemma_closure_of_GenC_proper_ideal}
Let $\Alg$ be a faithful Banach $\GenC$-algebra. Let $I\idealproper\Alg$ with $I\cap\GenC =\{0\}$. Then $\overline I\idealproper\Alg$ with $\overline I\cap\GenC =\{0\}$.
\end{lemma}
\begin{proof}
By the continuity of the operations $+$ and $\cdot$, it follows that $\overline I$ is an ideal.\\
Let $S\subseteq (0,1)$ with $e_S\ne 0$. By lemma \ref{lemma_GenC_proper_ideal}, $\{u\in\Alg: u$ invertible w.r.t.\ $S\}\subseteq\Alg\setminus I$. By proposition \ref{prop_invertible_in_Banach_algebra}, $\{u\in\Alg: u$ invertible w.r.t.\ $S\}$ is open, so $\{u\in\Alg: u$ invertible w.r.t.\ $S\}\subseteq\interior{(\Alg\setminus I)}=\Alg\setminus\overline{I}$. As $S$ is arbitrary, again by lemma \ref{lemma_GenC_proper_ideal}, $\overline I\cap\GenC=\{0\}$.
\end{proof}

\begin{cor}\label{cor_closure_of_GenC_proper_ideal}
Let $\Alg$ be a faithful Banach $\GenC$-algebra. Let $I\idealproper\Alg$ maximal w.r.t.\ $I\cap\GenC =\{0\}$. Then $I$ is closed.
\end{cor}

The following theorem can be viewed as an analogue of the correspondence between maximal ideals of a commutative $\C$-algebra $\Alg$ and quotients of $\Alg$ that are fields.
\begin{prop}\label{prop_GenC_max_ideal}
Let $\Alg$ be a faithful $\GenC$-algebra. Let $I\ideal\Alg$. Consider the following statements:
\begin{enumerate}
\item $\Alg/I$ is a faithful $\GenC$-algebra and $0$ is the only element in $\Alg/I$ that is strictly non-invertible
\item $I$ is maximal w.r.t.\ $I\cap\GenC=\{0\}$.
\end{enumerate}
Then $(1)\implies (2)$. If $\Alg$ is commutative, then $(1)\iff(2)$.
\end{prop}
\begin{proof}
$(1)\implies(2)$: by lemma \ref{lemma_GenC_proper_ideal}, $I\cap\GenC=\{0\}$. Let $J\ideal\Alg$ with $I\subsetneqq J$. Let $u\in J\setminus I$. As $u+I\ne 0$ in $\Alg/I$, there exists $S\subseteq(0,1)$ with $e_S\ne 0$ and $v\in \Alg$ such that $uv+I=e_S+I$. Hence $e_S\in J\cap\GenC\setminus\{0\}$.\\
$(2)\implies(1)$: let $\Alg$ be commutative. By lemma \ref{lemma_GenC_proper_ideal}, $\Alg/I$ is faithful. Let $u\in\Alg\setminus I$. By the maximality of $I$, there exists $\lambda\in\GenC\setminus\{0\}$ with $\lambda\in I+u\Alg$. As $\lambda\ne 0$, there exists $S\subseteq(0,1)$ with $e_S\ne 0$ such that $\lambda$ is invertible w.r.t.\ $S$. Hence $e_S\in I+u\Alg$, i.e., there exists $v\in\Alg$ with $uv+I=e_S+I$, hence $u+I$ is invertible w.r.t.\ $S$ in $\Alg/I$.
\end{proof}
\begin{cor}\label{cor_Kerm_char}
Let $\Alg$ be a faithful $\GenC$-algebra. Let $m\in\Smlf{\Alg}$. Then $\Ker m$ is an ideal maximal w.r.t.\ $\Ker m\cap\GenC = \{0\}$ and $\spec(\bar u)\ne\emptyset$, for each $\bar u\in\Alg/ \Ker m$.
\end{cor}
\begin{proof}
The first assertion follows by proposition \ref{prop_GenC_max_ideal} and the fact that $\Alg/\Ker m\cong \GenC$. For the second assertion, let $\bar u= u + \Ker m \in\Alg/\Ker m$, for some $u\in \Alg$. Since $u = m(u) + (u-m(u))\in m(u) + \Ker m$, $\bar u = m(u)$ in $\Alg/\Ker m$. Thus
%$\bar u - m(u) = 0$ is strictly non-invertible in $\Alg/I$, and
$m(u)\in \spec(\bar u)$.
\end{proof}

The converse of the previous corollary is given by the following proposition:
\begin{prop}\label{prop_mlf_in_Banach_algebra}
Let $\Alg$ be a faithful commutative Banach $\GenC$-algebra. Let $I\ideal\Alg$ be maximal w.r.t.\ $I\cap\GenC=\{0\}$. If $\spec(u)\ne\emptyset$, $\forall u \in\Alg/I$, then $I=\Ker m$, for some $m\in\Smlf{\Alg}$.
\end{prop}
\begin{proof}
By corollary \ref{cor_closure_of_GenC_proper_ideal}, $I$ is closed. By proposition \ref{prop_closed_ideal_in_Banach_algebra}, $\Alg/I$ is a Banach $\GenC$-algebra. By proposition \ref{prop_GenC_max_ideal}, $\Alg/I$ satisfies the conditions of proposition \ref{prop_Gelfand-Mazur}, so $\Alg/I\cong\GenC$. Hence the canonical surjection $\Alg\to\Alg/I$ induces a surjective multiplicative $\GenC$-linear functional $\Alg\to\GenC$ with $I$ as its kernel.
\end{proof}

\begin{cor}
Let $\Alg$ be a faithful commutative Banach $\GenC$-algebra. If each strictly non-invertible $v\in\Alg$ is contained in an ideal $I\ideal\Alg$ maximal w.r.t.\ $I\cap\GenC=\{0\}$ with the property that $\spec(u)\ne\emptyset$, $\forall u \in\Alg/I$, then $\spec(v)=\functspec(v)$, for each $v\in\Alg$.
\end{cor}
\begin{proof}
Let $v\in\Alg$ and $\lambda\in\spec(v)$. Then $v-\lambda$ is strictly non-invertible, so there exists $I\ideal\Alg$, maximal w.r.t.\ $I\cap\GenC=\{0\}$ with the property that $\spec(u)\ne\emptyset$, $\forall u \in\Alg/I$ containing $v-\lambda$. By proposition \ref{prop_mlf_in_Banach_algebra}, $I=\Ker m$, for some $m\in\Smlf{\Alg}$. Hence $m(v-\lambda)=0$, and $\lambda = m(v)\in \functspec(v)$. The converse inclusion follows by lemma \ref{lemma_fsp_sub_sp}.
\end{proof}

\subsection{The spectral mapping theorem}\label{section_spectralmapping}
\begin{thm}\label{thm_poly_spec_calc}
Let $\Alg$ be a faithful $\GenC$-algebra and $u,v\in\Alg$.
\begin{enumerate}
\item Let $p\in\GenC[x]$ be a polynomial and let $p(C)=\{p(\lambda):\lambda\in C\}$, for $C\subseteq\GenC$.\\
Then $p(\spec(u))\subseteq\spec(p(u))$ and $p(\functspec(u))=\functspec(p(u))$. If $p(\lambda)\in\resolv(p(u))$, then $\lambda\in\resolv(u)$.
\item For $C\subseteq\{\lambda\in\GenC: \lambda$ is invertible$\}$, let $\inv C=\{\inv\lambda: \lambda\in C\}$. If $u$ is invertible, then
$\spec(\inv u)=\inv{(\spec(u))}$ and $\functspec(\inv u)=\inv{(\functspec(u))}$.
\item $\functspec(u+v)\subseteq\functspec(u)+\functspec(v)$ and $\functspec(uv)\subseteq\functspec(u)\functspec(v)$.
\end{enumerate}
\end{thm}
\begin{proof}
(1) Let $u\in\Alg$, $\lambda\in\GenC$ and $p(x)=a_n x^n + \cdots + a_0$, with $a_j\in\GenC$, $\forall j$. Then
\begin{align*}
p(u)-p(\lambda ) &=a_n(u^n-\lambda^n)+\cdots + a_1 (u-\lambda),\\
u^j - \lambda^j &=(u-\lambda) (u^{j-1}+\lambda u^{j-2} + \cdots + \lambda^{j-1})\\
&=(u^{j-1}+\lambda u^{j-2} + \cdots + \lambda^{j-1})(u-\lambda),
\end{align*}
for $j=1$, \dots, $n$, so there exists $v\in\Alg$ such that
$p(u)-p(\lambda) = (u-\lambda) v=v (u-\lambda)$. So if $p(\lambda)\in \resolv(p(u))$, i.e., $p(u)-p(\lambda)$ is invertible, then also $u-\lambda$ is invertible, i.e., $\lambda\in\resolv(u)$. Similarly, if $p(\lambda)\notin\spec(p(u))$, then there exists $S\subseteq(0,1)$ with $e_S\ne 0$ such that $p(u)-p(\lambda)$ is invertible w.r.t.\ $S$. By lemma \ref{lemma_invertible_wrt_S}, also $u-\lambda$ is invertible w.r.t.\ $S$, and $\lambda\notin\spec(u)$. By contraposition, $p(\spec(u))\subseteq\spec(p(u))$.\\
(2) Let $u\in\Alg$ be invertible. We show that $\spec(u)$ contains only invertible elements.\\
Let $\lambda\in\spec(u)$. Suppose that $\lambda\in\GenC$ is not invertible. Then there exists $S\subseteq(0,1)$ with $e_S\ne 0$ and $\lambda e_S =0$ \cite[Lemma 4.1]{HV_ideals}. As $u-\lambda$ is not invertible w.r.t.\ $S$, also $u=u-\lambda e_S$ is not invertible w.r.t.\ $S$ by lemma \ref{lemma_invertible_wrt_S}, a contradiction.\\
Let $\lambda\in\GenC$ be invertible and let $S\subseteq (0,1)$ with $e_S\ne 0$. Then by lemma \ref{lemma_invertible_wrt_S}, $\inv{u}-\lambda$ is invertible w.r.t\ $S$ iff $\inv{\lambda}u(\inv{u}-\lambda)=\inv{\lambda}-u$ is invertible w.r.t.\ $S$. So for $\lambda$ invertible, $\lambda\in\spec(\inv u)$ iff $\inv{\lambda}\in\spec(u)$. As $\spec(u)$ and $\spec(\inv u)$ only contain invertible elements, $\spec(\inv u)=\inv{(\spec(u))}$.\\
The assertions about the functional spectrum are elementary.
\end{proof}

\begin{lemma}\label{lemma_spec_au}
Let $\Alg$ be a faithful $\GenC$-algebra, $a\in\GenC$, $u\in\Alg$ and $\lambda\in\spec(au)$. Then $\lambda\in\overline{a\GenC}$ (here $\overline{\phantom{A}}$ denotes the topological closure).
\end{lemma}
\begin{proof}
Let $S\subseteq (0,1)$ such that $e_S a = 0$. Then $e_S \lambda \in \spec(e_S au) = \spec(0) = \{0\}$. The result follows by \cite[Prop.\ 4.3(1), Thm.\ 5.7(2)]{HV_ideals}.
\end{proof}

\begin{prop}\label{prop_spec_au}
Let $\Alg$ be a faithful $\GenC$-algebra and $u\in\Alg$ with $\spec(u)\ne\emptyset$. Let $a\in\GenC$.
\begin{enumerate}
\item Let $a\GenC$ be closed (i.e., there exists $S\subseteq (0,1)$ such that $a\GenC=e_S\GenC$ \cite[Thm.\ 5.7]{HV_ideals}). Then $a\spec(u) = \spec(au)$.
\item $a\spec(u)\subseteq \spec(au) \subseteq \overline{a\spec(u)}$.
\item If $\spec(u)$ is the intersection of a decreasing sequence of internal subsets of $\GenC$, then $a\spec(u) = \spec(au)$.
\end{enumerate}
\end{prop}
\begin{proof}
By theorem \ref{thm_poly_spec_calc}, $a\spec(u)\subseteq \spec(au)$, $\forall a\in\GenC$.\\
(1) Let first $a= e_S$. Let $\lambda\in\spec(e_S u)$. Let $\mu\in\spec(u)$ arbitrary. As $e_S u- e_S\lambda$, $e_{\co S}u- e_{\co S}\mu$ are strictly non-invertible, also $u - (e_S\lambda + e_{\co S}\mu)$ is strictly non-invertible by lemma \ref{lemma_spec_resolv_interleave}. Hence $e_S\lambda\in e_S\spec(u)$. As $e_{\co S}\lambda\in\spec(e_{\co S}e_S u)=\spec(0) = \{0\}$, $\lambda = e_S\lambda\in e_S\spec(u)$.\\
Let now $a\GenC=e_S\GenC$. Then there exists $b\in\GenC$ such that $ab=e_S$. Let $\lambda\in\spec(au)$. Then $e_S \lambda = a b \lambda\in a b \spec(a u) \subseteq a \spec(a b u)= a \spec(e_S u) = a e_S \spec(u) = a \spec(u)$. As $e_{\co S}\lambda\in\spec(e_{\co S}a u)=\spec(0) = \{0\}$, $\lambda = e_S\lambda\in a\spec(u)$.\\
(2) Let $\lambda\in\spec(au)$. Let $(S_n)_{n\in\N}$ be a sequence of level sets for $a$ \cite[\S 5]{HV_ideals}. Then $a e_{S_n} \GenC = e_{S_n}\GenC$, for each $n$. By part 1, $e_{S_n} \lambda\in \spec(a e_{S_n} u) = a e_{S_n}\spec(u)$. Let $\mu\in\spec(u)$ arbitrary. Then $e_{S_n}\lambda + a e_{\co S_n}\mu\in a (e_{S_n}\spec(u) + e_{\co S_n}\spec(u)\subseteq a \spec(u)$ by lemma \ref{lemma_spec_resolv_interleave}. By lemma \ref{lemma_spec_au} and \cite[Thm.\ 5.7(4)]{HV_ideals}, $\lim_{n\to\infty} (e_{S_n}\lambda + a e_{\co S_n}\mu)=\lambda$. Hence $\lambda\in \overline{a\spec(u)}$.\\
(3) By part 2, it is sufficient to show that $a\spec(u)$ is closed. Let $\spec(u) = \bigcap_{n\in\N} C_n$, where $C_n=[(C_{n,\eps})_\eps]$ are internal and $(C_n)_n$ is decreasing. By theorem \ref{thm_bounded_spectrum}, there exists $N\in\N$ such that $\abs{\lambda}\le\alpha^{-N}$, for each $\lambda\in\spec(u)$. Hence $\spec(u) = \bigcap_{n\in\N} [(C_{n,\eps}\cap B(0,(1+1/n)\eps^{-N}))_\eps]$, and we may assume that $C_n$ have sharply bounded representatives. By \cite[2.9]{OV_internal}, we may also assume that $C_{n+1,\eps}\subseteq C_{n,\eps}$, for each $n,\eps$. We show that then $a\spec(u) = \bigcap_{n\in\N} [(a_\eps C_{n,\eps})_\eps]$, and thus $a\spec(u)$ is closed by \cite[2.3]{OV_internal}.\\
Let $a=[(a_\eps)_\eps]$. If $\lambda\in a\spec(u)$, then $\lambda \in a C_n\subseteq [(a_\eps C_{n,\eps})_\eps]$, for each $n$. Conversely, if for each $n$, $\lambda = [(\lambda_\eps)_\eps]= [(a_\eps c_{n,\eps})_\eps]$, with $c_{n,\eps}\in C_{n,\eps}$, then there exist $\eps_n$ (w.l.o.g.\ decreasingly tending to $0$) such that $\abs{\lambda_\eps - a_\eps c_{n,\eps}}\le \eps^n$, if $\eps\le\eps_n$. Let $c_\eps := c_{n,\eps}$, if $\eps_{n+1}<\eps\le \eps_n$, for each $n$. As $C_{n,\eps}$ are sharply bounded, $(c_\eps)_\eps\in\Mod_\C$ and $\lambda = a c$, with $c := [(c_\eps)_\eps]\in\GenC$. As $C_{n+1,\eps}\subseteq C_{n,\eps}$, for each $n$, $\eps$, we have $c\in \bigcap_{n\in\N} C_n=\spec(u)$.
\end{proof}

\begin{thm}\label{thm_holomorphic_spec_calc}
Let $(\Alg,\pseudonorm{.})$ be a faithful Banach $\GenC$-algebra.
Let $f(z)=\sum_{n=0}^\infty a_n z^n$ with $a_n\in\GenC$ and with $1/R:=\limsup_{n\to\infty}\sqrt[n]{\sharpnorm{a_n}}<+\infty$. Then $f(\spec(u))\subseteq\spec(f(u))$ and $f(\functspec(u))=\functspec(f(u))$, $\forall u\in\Alg$ with $\pseudonorm{u}<R$.
\end{thm}
\begin{proof}
Let $u\in\Alg$ with $\pseudonorm{u}< R$. Let $\lambda\in\spec(u)$. By theorem \ref{thm_bounded_spectrum}, also $\sharpnorm\lambda< R$. So $\sum_n \pseudonorm{a_n u^n}\le\sum_n\sharpnorm{a_n}\pseudonorm{u}^n<+\infty$, and thus by lemma \ref{lemma_convergence_Banach_module}, $f(u)=\sum_n a_n u^n$ converges in $\Alg$. Similarly, $f(\lambda)=\sum_n a_n \lambda^n$ converges in $\GenC$. Then $f(u)-f(\lambda) = \sum_{n= 0}^\infty a_n(u^n-\lambda^n)=\sum_{n=0}^\infty a_n(u-\lambda) (u^{n-1}+\lambda u^{n-2}+\cdots + \lambda^{n-1})$. As $\pseudonorm{\lambda^{n-j-1}u^{j}} \le\sharpnorm{\lambda}^{n-j-1}\pseudonorm{u}^j \le\pseudonorm{u}^{n-1}$, $\forall j$, also $\sum_n \pseudonorm{a_n(u^{n-1}+\lambda u^{n-2}+\cdots + \lambda^{n-1})}\le \sum_n \sharpnorm{a_n} \pseudonorm{u}^{n-1}$ $<\infty$. Hence there exists $v\in\Alg$ such that $f(u)-f(\lambda) = (u-\lambda)v= v(u-\lambda)$. As in theorem \ref{thm_poly_spec_calc}, the assumption that $f(\lambda)\notin\spec(f(u))$ then implies that $\lambda\notin\spec(u)$, a contradiction. So $f(\spec(u))\subseteq\spec(f(u))$.\\
Let $m\in\Smlf{\Alg}$. By proposition \ref{prop_cnt_mlf_in_Banach_algebra}, $f(m(u))=\sum_n a_n m(u)^n$ converges in $\GenC$ and since $m$ is continuous, $f(m(u)) = m(f(u))$.
\end{proof}

\begin{df}
Let $\Alg$ be a faithful $\GenC$-algebra.
The \defstyle{spectral radius} of $u\in\Alg$ is
\[\specrad_\Alg(u)=\sup\{\sharpnorm{\lambda}: \lambda\in \spec(u)\}.\]
If the algebra is clear from the context, we simply write $\specrad(u)$. So by definition, $\specrad(u)\in [0,\infty]\cup\{-\infty\}$. By theorem \ref{thm_bounded_spectrum}, if $(\Alg,\pseudonorm{.})$ is a Banach algebra, $\specrad(u)\le \pseudonorm{u}$.%It makes a priori no sense to define a spectral radius with values in $\GenR$, since suprema do not necessarily exist.
\end{df}

\begin{prop}\label{prop_spectral_radius_bound}
Let $(\Alg,\pseudonorm{.})$ be a faithful Banach $\GenC$-algebra. Then $\specrad(u)\le\lim_{n\to\infty} \sqrt[n]{\pseudonorm{u^n}}$.
\end{prop}
\begin{proof}
Since $\pseudonorm{.}$ is submultiplicative, $\lim_{n\to\infty}\sqrt[n]{\pseudonorm{u^n}}$ exists (as in \cite[Ch.~VI, Problem 11]{ReedSimon}). If $\lambda\in\spec(u)$, then $\lambda^n\in\spec(u)^n\subseteq \spec(u^n)$ by theorem \ref{thm_poly_spec_calc}. By theorem \ref{thm_bounded_spectrum}, $\sharpnorm{\lambda}^n=\sharpnorm{\lambda^n}\le \pseudonorm{u^n}$. So $\specrad(u)\le\sqrt[n]{\pseudonorm{u^n}}$, $\forall n\in\N$.
\end{proof}

\section{Colombeau $*$-algebras}
\begin{df}
A $\GenC$-algebra $\Alg$ is called a \defstyle{Colombeau $*$-algebra} if there exists a map $*$: $\Alg\to\Alg$ (called \defstyle{involution}) satisfying
\begin{enumerate}
\item $(\forall u\in\Alg)$ $(u^{**}=u)$
\item $(\forall u, v\in\Alg)$ $((u+v)^*=u^* + v^*)$
\item $(\forall u\in\Alg)$ $(\forall \lambda\in\GenC)$ $((\lambda u)^*=\overline\lambda u^*)$ (where $\bar{.}$ denotes complex conjugation in $\GenC$)
\item $(\forall u, v\in\Alg)$ $((uv)^*=v^*u^*)$.
\end{enumerate}
As for $\C$-algebras, it follows that $1^*=1$ (e.g., \cite[\S 4.1]{Kadison}).
If $\Alg$ is faithful, then $\spec(u^*)=\{\overline \lambda: \lambda \in \spec(u)\}$, $\forall u\in\Alg$.\\
$u\in\Alg$ is called \defstyle{self-adjoint} if $u^*=u$; \defstyle{unitary} if $u u^*=u^* u =1$; \defstyle{normal} if $uu^*=u^*u$.\\
A $*$-algebra $\Alg$ is called \defstyle{symmetric} if $1+u u^*$ is invertible for each $u\in \Alg$ (cf.~\cite[2.24]{Husain}).
\end{df}
%E.g., the algebra $\Gen(A)$, where $A\subseteq \R^d$ is open or compact is a symmetric $\GenC$-algebra (the involution is given by complex conjugation). This follows from a characterization of the invertible elements \cite[Thm.~1.2.5]{GKOS}.

\begin{lemma}\label{lemma_real}
Let $\lambda\in\GenC$ such that $\lambda^2\in\GenR$ and $\lambda^2 \ge 0$. Then $\lambda\in\GenR$.
\end{lemma}
\begin{proof}
As $\lambda\in\GenC$, $\lambda= a+bi$, with $a$, $b$ $\in\GenR$, so $\lambda^2=a^2-b^2+2abi$. By assumption, $ab=0$ and $a^2\ge b^2$. Hence $0\le b^4\le a^2 b^2 = 0$, and $b=0$ since $0$ is the only nilpotent element in $\GenR$.
%By a characterization of zero divisors, there exists $S\subseteq(0,1)$ such that $a e_S=0$ and $b e_{\co S}=0$. Now $0 = a^2 e_S \ge b^2 e_S \ge 0$, so $b^2 e_S=0$ and $b e_S = 0$, by the fact that there are no nilpotent elements in $\GenR$. So $b=b e_S + b e_{\co S}= 0$.
\end{proof}

\begin{lemma}\label{lemma_hermitian_mlf}
For a multiplicative $\GenC$-linear functional $m$ on a Colombeau $*$-algebra $\Alg$, the following are equivalent:
\begin{enumerate}
\item $(\forall u\in\Alg)$ $(m(u^*)=\overline{m(u)})$%(where $\bar{.}$ denotes complex conjugation in $\GenC$)
\item $(\forall u\in\Alg, u$ self-adjoint$)$ $(m(u)\in\GenR)$
\item $(\forall u\in\Alg)$ $(m(uu^*)\ge 0)$.
\end{enumerate}
\end{lemma}
\begin{proof}
$(1)\implies(3)$: $m(uu^*)=\abs{m(u)}^2\ge 0$.\\
$(3)\implies(2)$: if $u=u^*$, then $m(u)^2=m(uu^*)\ge 0$, so $m(u)\in\GenR$ by lemma \ref{lemma_real}.\\
$(2)\implies(1)$: for $u\in\Alg$, $u=v+ iw$, where $v=(u + u^*)/2$ and $w=i(u^*-u)/2$ are self-adjoint. So $m(v)$, $m(w)\in\GenR$. Hence $m(u^*) = m(v - iw)=m(v)-im(w)=\overline{m(v) + im(w)}=\overline{m(u)}$.
\end{proof}

\begin{df}
If a multiplicative $\GenC$-linear functional $m$ satisfies one of the conditions of the previous lemma, then $m$ is called \defstyle{hermitian}.
\end{df}

\begin{thm}\label{thm_sp_sub_GenR_in_symm_algebra}
Let $\Alg$ be a faithful symmetric Colombeau $*$-algebra and let $u\in\Alg$ be self-adjoint. Then $\spec(u)\subseteq\GenR$.
\end{thm}
\begin{proof}
Let $u\in\Alg$ be self-adjoint and $\lambda\in\GenC$. Write $\lambda = a+ bi$ with $a$, $b$ $\in\GenR$. Suppose that $b\ne 0$. Then there exists $S\subseteq (0,1)$ with $e_S\ne 0$ such that $b$ is invertible w.r.t.\ $S$, i.e., there exists $c\in\GenR$ such that $bc=e_S$. Then
\[
e_S c^2 (u-\lambda )(u-\bar\lambda )=e_S c^2[(u-a)^2+b^2]=e_S[c^2(u-a)^2 + 1].
\]
As $u=u^*$, also $c(u-a)=(c(u-a))^*$, so $c^2(u-a)^2+1$ is invertible. Since $u-\lambda$ and $u-\bar \lambda$ commute, $u-\lambda$ is invertible w.r.t.\ $S$ by lemma \ref{lemma_invertible_wrt_S}. Hence $\lambda\notin\spec(u)$.
\end{proof}

\begin{cor}
Let $\Alg$ be a faithful symmetric Colombeau $*$-algebra and $m\in\Hom(\Alg,\GenC)$. Then $m$ is hermitian.
\end{cor}
\begin{proof}
Let $u\in\Alg$ be self-adjoint. By lemma \ref{lemma_fsp_sub_sp} and theorem \ref{thm_sp_sub_GenR_in_symm_algebra}, $\functspec(u)\subseteq\spec(u)\subseteq\GenR$. By lemma \ref{lemma_hermitian_mlf}, this means that each $m\in\Hom(\Alg,\GenC)$ is hermitian.
\end{proof}

\section{Colombeau $C^*$-algebras}
\begin{df}
A \defstyle{Colombeau $C^*$-algebra} $(\Alg,\pseudonorm{.})$ is a Banach $\GenC$-algebra that is also a Colombeau $*$-algebra and for which
\[\pseudonorm{u^* u}=\pseudonorm{u}^2,\quad\forall u\in\Alg.\]
\end{df}
As for classical $C^*$-algebras, this implies that $\pseudonorm{u^*}=\pseudonorm u$, $\forall u\in\Alg$; hence $*$: $\Alg\to\Alg$ is continuous and the set of self-adjoint elements of $\Alg$ is closed.

\begin{prop}\label{prop_Cstar_elementary}
Let $(\Alg,\pseudonorm{.})$ be a Colombeau $C^*$-algebra.
\begin{enumerate}
\item If $u\in\Alg$ is normal, then $\pseudonorm{u^n}=\pseudonorm{u}^n$, for each $n\in\N$.
\item If $u\in\Alg$ is unitary, then $\pseudonorm{u}=1$ and $\sharpnorm{\lambda}=1$, $\forall\lambda\in\spec(u)$.
\end{enumerate}
\end{prop}
\begin{proof}
(1) Since $\pseudonorm{.}$ is submultiplicative, $s(u):=\lim_{n\to\infty}\pseudonorm{u^n}^{1/n}$ exists (as in \cite[Ch.~VI, Problem 11]{ReedSimon}) and $s(u)\le \pseudonorm{u}$, for each $u\in\Alg$. If $u$ is self-adjoint, then $\pseudonorm{u^{2^n}}=\pseudonorm{u}^{2^n}$, for each $n\in\N$, hence $s(u)=\pseudonorm u$. If $u$ is normal, then
\[s(u^* u)=\lim_{n}\pseudonorm{{u^*}^n u^n}^{1/n}\le s(u^*)s(u) = s(u)^2,\]
and $\pseudonorm{u}^2=\pseudonorm{u^* u} = s(u^* u)\le s(u)^2\le \pseudonorm{u}^2$, hence $s(u)=\pseudonorm{u}$. The result follows, since $s(u)\le \pseudonorm{u^n}^{1/n}\le \pseudonorm{u}$, for each $n\in\N$.\\
(2) Since $1=\pseudonorm 1=\pseudonorm{u^* u}=\pseudonorm{u}^2$, $\pseudonorm u =\pseudonorm{u^*} = 1$. Let $\lambda\in\spec(u)$. By theorem \ref{thm_bounded_spectrum}, $\sharpnorm\lambda\le 1$. As $\inv u= u^*$, $\inv\lambda\in\spec(u^*)$ by theorem \ref{thm_poly_spec_calc}, hence also $\sharpnorm{\inv\lambda}\le 1$. So $1=\sharpnorm{\lambda\inv\lambda} \le \sharpnorm{\lambda}\sharpnorm{\inv\lambda}\le 1$, hence $\inv{\sharpnorm{\lambda}}=\sharpnorm{\inv\lambda}\le 1$, and $\sharpnorm\lambda= 1$.
\end{proof}
\begin{cor}
Let $\Alg$ be a Colombeau $C^*$-algebra. If $u\in\Alg$ is normal and $u^n=0$, for some $n\in\N$, then $u=0$.
\end{cor}

We recall the Cauchy-Schwarz inequality for $\GenC$-modules \cite{GV_Hilbert}:
\begin{prop}\label{prop_C-S}
Let $\Gen$ be a $\GenC$-module provided with a $\GenC$-sesquilinear map $\inner{.}{.}$: $\Gen\times\Gen\to\GenC$ such that $\inner{u}{v} = \overline{\inner{v}{u}}$ and $\inner{u}{u}\ge 0$ (in $\GenR$), for each $u,v\in\Gen$. Then $\abs{\inner{u}{v}}^2\le \inner{u}{u}\inner{v}{v}$, for each $u,v\in\Gen$.
\end{prop}
\begin{proof}
The proof of \cite[Prop.\ 2.2]{GV_Hilbert} also holds if the scalar product $\inner{.}{.}$ does not necessarily satisfy $\inner{u}{u}=0 \implies u=0$.
\end{proof}

Recall that a Hilbert $\GenC$-module $\mathcal H$ is a $\GenC$-module provided with an inner product $\inner{.}{.}$ satisfying the properties of proposition \ref{prop_C-S} together with $\inner{u}{u}=0 \implies u=0$, $\forall u\in\mathcal H$, such that it is complete (and hence a Banach $\GenC$-module) for the ultra-pseudo-norm $\pseudonorm{u}:=\sharpnorm[]{\sqrt{\inner{u}{u}}}$. We also denote $\norm{u}:=\sqrt{\inner{u}{u}}\in\GenR$.
\begin{df}
For a Hilbert $\GenC$-module $\mathcal H$, we denote $\mathcal B(\mathcal H):= \{T\in \cLin(\mathcal H): T$ has an adjoint$\}$.
\end{df}
It is not known if $\mathcal B(\mathcal H) = \cLin(\mathcal H)$ (cf.\ the conjecture in \cite[\S 4]{GV_Hilbert}).

\begin{prop}\label{prop_B(H)_is_C_star_algebra}
Let $\mathcal H$ be a Hilbert $\GenC$-module. Then $\mathcal B(\mathcal H)$ is a Colombeau $C^*$-algebra.
\end{prop}
\begin{proof}
By proposition \ref{prop_cLin_of_Banach_module}, $\cLin(\mathcal H)$ is a Banach $\GenC$-algebra. If $T$, $S\in\cLin(\mathcal H)$ have an adjoint, then also $T+S$, $\lambda T$ ($\lambda\in\GenC$), $TS$, $T^*$ have an adjoint and $\mathcal B(\mathcal H)$ is a Colombeau $*$-algebra. As for classical Hilbert spaces, by proposition \ref{prop_C-S},
\begin{multline*}
\pseudonorm{T}^2 = \sup_{\pseudonorm{u} = 1}\pseudonorm{Tu}^2 = \sup_{\pseudonorm{u} = 1}\sharpnorm{\norm{Tu}^2} = \sup_{\pseudonorm{u} = 1}\sharpnorm{\inner{u}{T^* T u}}\\
\le \sup_{\pseudonorm{u} = 1}\pseudonorm{u}\pseudonorm{T^* T u} \le \pseudonorm{T^* T},
\end{multline*}
whence $\pseudonorm{T^*}=\pseudonorm{T}$ and $\pseudonorm{T^* T}=\pseudonorm{T}^2$. So if for each $n\in\N$, $T_n\in\mathcal B(\mathcal H)$ and $T_n\to T\in\cLin(\mathcal H)$, then also $\lim_{n\to\infty}T_n^* \in \cLin(\mathcal H)$ exists by completeness of $\cLin(\mathcal H)$. By continuity of $\inner{.}{.}$, $T^* = \lim_{n\to\infty}T_n^*$. Thus $\mathcal B(\mathcal H)$ is a $C^*$-subalgebra of $\cLin(\mathcal H)$.
\end{proof}

If $B$ is a clasical $C^*$-algebra, then for $u=[(u_\eps)_\eps]\in\Gen_B$, $u^*:=[(u_\eps^*)_\eps]$ is well-defined and turns $\Gen_B$ into a Colombeau $C^*$-algebra with $\norm{u^*u}=\norm{u}^2$ and $\norm{u}=\norm{u^*}$, $\forall u\in\Alg$.

\begin{prop}\label{prop_G_B(H)_is_C_star_subalgebra}
If $H$ is a (classical) Hilbert space and $\cLin(H)$ is the $C^*$-algebra of continuous $\C$-linear operators on $H$, then $\Gen_{\cLin(H)}$ is (up to isomorphism) a $C^*$-subalgebra of $\mathcal \cLin(\Gen_H)$.
\end{prop}
\begin{proof}
As $\cLin(H)$ is a $C^*$-algebra, $\Gen_{\cLin(H)}$ is a Colombeau $C^*$-algebra with ultra-pseudonorm $\pseudonorm T=\sharpnorm[\big]{\norm{T_\eps}}$, for $T=[(T_\eps)_\eps]\in\Gen_{\cLin(H)}$. By proposition \ref{prop_G_L(B)_is_Banach_subalgebra}, $\Gen_{\cLin(H)}$ is isomorphic with a Banach $\GenC$-subalgebra of $\mathcal \cLin(\Gen_H)$. The isomorphism of proposition \ref{prop_G_L(B)_is_Banach_subalgebra} clearly also preserves the involution $*$.
\end{proof}

\section{Counterexamples}
The following example is in contrast with the situation in commutative Banach $\C$-algebras (e.g., \cite[\S 3.2]{Kadison}):
\begin{ex}\label{ex_empty_spectrum_wrt_S}
There exists $M\idealproper\Gen_{\Cnt([0,1])}$ maximal w.r.t.\ $M\cap\GenC = \{0\}$ and $u\in\Alg := \Gen_{\Cnt([0, 1])}/M$ (which is a faithful commutative Banach $\GenC$-algebra by proposition \ref{prop_closed_ideal_in_Banach_algebra} and corollary \ref{cor_closure_of_GenC_proper_ideal}) such that $\resolv_\Alg(u)=\GenC$. Consequently, we also have:
\begin{enumerate}
\item $\spec_\Alg(u)=\emptyset$.
\item $M\ne \Ker m$, for each $m\in\Smlf{\Gen_{\Cnt([0,1])}}$.
\item $\functspec_\Alg(u) = \emptyset$.
\item $\functspec_\Alg(v) \ne \spec_\Alg(v)$, for some $v\in\Alg$.
\end{enumerate}
\end{ex}
\begin{proof}
Let $\phi\in\Cnt[\infty](\C)$ with compact support and with $\phi(0)=1$ and $\phi(z)\ge 0$, for each $z\in\C$. For each $a=[(a_\eps)_\eps]\in\GenC$, let $u_a:=\big[\big(\phi(\frac{x-a_\eps}{\eps})\big)_\eps\big]\in\Gen_{\Cnt([0,1])}$. Let $I\idealproper\Gen_{\Cnt([0,1])}$ be the ideal generated by $\{u_a: a\in\GenC\}$.
Let $u\in I$. As $u$ is a finite $\Gen_{\Cnt([0,1])}$-linear combination of some of the $u_a$, it has a representative $(u_\eps)_\eps$ for which the Lebesgue measure of the support of $u_\eps$ is of order $\eps$. On the other hand, let $u\in\GenC\setminus\{0\}\subseteq \Gen_{\Cnt([0,1])}$. Then $u$ is invertible w.r.t.\ $S$, for some $S\subseteq(0,1)$ with $e_S\ne 0$. For each representative $(u_\eps)_\eps\in\Mod_{\Cnt([0,1])}$ of $u$, there exists $m\in\N$ such that $\inf_{x\in[0,1]}\abs{u_\eps(x)}\ge\eps^m$, for small $\eps\in S$ by lemma \ref{lemma_invertible_wrt_S}. Thus $I\cap\GenC =\{0\}$. By Zorn's lemma, there exists $M\idealproper\Gen_{\Cnt([0,1])}$, maximal w.r.t.\ $M\cap\GenC =\{0\}$ with $I\subseteq M$.
Let $u=\id + M\in\Alg$ where $\id=[(\id_{[0,1]})_\eps]\in\Gen_{\Cnt([0,1])}$ and $\id_{[0,1]}$ is the identity map on $[0,1]$. Let $\lambda=[(\lambda_\eps)_\eps]\in\GenC$. For $x\in [0,1]$,
\[\abs[\Big]{\phi\Big(\frac{x-\lambda_\eps}{\eps}\Big)-\phi(0)}\le \eps^{-1}\abs{x - \lambda_\eps}\sup_{z\in\C}\abs{\nabla\phi(z)},\]
so $\phi(\frac{x-\lambda_\eps}{\eps})\ge \eps$, as soon as $\abs{x - \lambda_\eps}\le\eps^2$ and $\eps$ small enough. Hence
\[
\inf_{x\in [0,1]}
\abs{x-\lambda_\eps}^2 + \phi\Big(\frac{x-\lambda_\eps}{\eps}\Big)\ge\eps^4,
\]
for small $\eps$, and $\abs{\id -\lambda}^2 + u_\lambda$ is invertible in $\Gen_{\Cnt([0,1])}$ by proposition \ref{prop_invertible_in_Gencnt}. As $u_\lambda\in M$, there exists $v\in\Gen_{\Cnt([0,1])}$ such that $(\id -\lambda)(\overline{\id -\lambda})v=1$ modulo $M$. Hence $u-\lambda$ is invertible in $\Alg$. Consequence 2 follows by corollary \ref{cor_Kerm_char}. Consequence 3 follows by lemma \ref{lemma_fsp_sub_sp}. Hence $\Smlf\Alg = \emptyset$, and consequence 4 follows since $\functspec_\Alg(0)=\emptyset\ne \{0\} = \spec_\Alg(0)$.
\end{proof}

For general Colombeau $C^*$-algebras, even if they are subalgebras of $\Gen_B$ (with $B$ a $C^*$-algebra), almost nothing of the classical spectral theory (e.g., \cite[\S 4.1]{Kadison}) remains valid:
\begin{ex}\label{ex_spec_Cstar_subalgebra}
There exists a commutative $C^*$-algebra $B$, a Colombeau $C^*$-subalgebra $\Alg$ of $\Gen_B$ and a self-adjoint $u\in \Alg$ such that
\begin{enumerate}
\item $\spec_\Alg(u)\ne\spec_{\Gen_B}(u)$.
\item $\spec_\Alg(u)\not\subseteq\{\lambda\in\GenC: \abs{\lambda}\le\norm{u}\}$.
\item $\spec_\Alg(u)\not\subseteq\GenR$.
\item $\Alg$ is not a symmetric Colombeau $*$-algebra.
\end{enumerate}
\end{ex}
\begin{proof}
Let $B=\Cnt([0,1])$, $u = [(\id)_\eps]\in\Gen_B$ where $\id$ is the identity map on $[0,1]$, and $\Alg=\overline{\GenC[u]}$. For $\lambda\in\C\setminus [0,1]$, $\id-\lambda$ is invertible in $B$, hence also $(u-\lambda)^{-1} = [((\id -\lambda)^{-1})_\eps]$ in $\Gen_B$. Suppose that $S\subseteq (0,1)$ with $e_S\ne 0$ and $u-\lambda$ is invertible w.r.t.\ $S$ in $\Alg$. Then $(u-\lambda)^{-1}e_S\in\Alg$ by lemma \ref{lemma_invertible_wrt_S}, and there exists $w\in\GenC[u]$ such that $\norm{w-(u - \lambda)^{-1}e_S}\le\caninf$. As $w=\sum_{j=0}^m a_j u^j$, for some $m\in\N$ and $a_j\in\GenC$, there exists a sequence of polynomials $p_n\in\C[x]$ of degree at most $m$ such that $\norm{p_n - (\id - \lambda)^{-1}}\to 0$. Since a finite-dimensional subspace of a Banach space is closed (e.g, \cite[Cor.~1.5.4]{Kadison}), $(\id - \lambda)^{-1}\in\C[x]$, a contradiction. Hence $\C\setminus [0,1] \subseteq \spec_\Alg(u)\setminus \spec_{\Gen_B}(u)$. The last assertion follows by theorem \ref{thm_sp_sub_GenR_in_symm_algebra}.%With $v=\sqrt x\in B$, we then also have that $1 + v v^* = 1 + u$ is not invertible in $\Alg$.
\end{proof}
%In fact, in the previous example, more precisely one can explicitly construct $m\in\Smlf(\Alg)$ such that $m(u)=\lambda$, for each $\lambda\in\GenC$ with $\sharpnorm\lambda \le 1$. From this, it follows by theorem \ref{thm_bounded_spectrum} that $\functspec_\Alg(u)=\spec_\Alg(u)=\{\lambda\in\GenC: \sharpnorm\lambda \le 1\}$.

\section{The $\GenC$-algebra $\Gensharp X$}\label{section_Gensharp}
In spite of the counterexamples in the previous section, for some Banach subalgebras of $\Gen_B$ ($B$ a Banach $\C$-algebra), the spectrum behaves to a large extent as in the classical theory. In this section, we briefly investigate the structure of such an algebra, which will provide an explanation why the spectrum in this algebra behaves well (see propositions \ref{prop_Gensharp_is_intersection_of_strictly_internal} and \ref{prop_GenB_invertible}, corollary \ref{cor_spec_Gensharp}).

Let $X$ be a compact metric space. Let $\widetilde X:= X^{(0,1)}/_{\sim}$, where
$(x_\eps)_\eps\sim (y_\eps)_\eps$ iff $(d(x_\eps,y_\eps))_\eps\in\Null_\R$ (this definition coincides with the definition of the internal subset $\widetilde K = [(K)_\eps]\subset\GenR^d$ when $K\csub\R^d$). Then $d$: $\widetilde X\times\widetilde X\to\GenR$: $d([(x_\eps)_\eps],[(y_\eps)_\eps]):=[(d(x_\eps,y_\eps)_\eps)]$ is a well-defined map that extends the metric $d$ on $X$ and $D$: $\widetilde X\times\widetilde X\to \R$: $D(\tilde x, \tilde y):= \sharpnorm{d(\tilde x, \tilde y)}$ defines an ultrametric on $\widetilde X$. The corresponding topology on $\widetilde X$ is called sharp topology. Also the (non-Hausdorff) topology on $X^{(0,1)}$ defined by the corresponding pseudometric $D$: $X^{(0,1)}\times X^{(0,1)}\to \R$: $D((x_\eps)_\eps, (y_\eps)_\eps):= \sharpnorm{d(x_\eps,y_\eps)}$ is called sharp topology.

Let $u\in\Gencnt{X}$. If $u(x) = u(y)$ in $\GenC$, for each $x$, $y$ $\in X^{(0,1)}$ with $x\sim y$, then $u$ defines a pointwise map $\widetilde X\to\GenC$ by means of $u(\tilde x):=[(u_\eps(x_\eps))_\eps]$ (with $u=[(u_\eps)_\eps]$, $\tilde x=[(x_\eps)_\eps]$, definition independent of representatives).

\begin{prop}\label{prop_sharp_cnt}
Let $X$ be a compact metric space and $u\in\Gencnt{X}$. Then the following are equivalent:
\begin{enumerate}
\item $u$ defines a map $\widetilde X\to\GenC$
\item $(\forall n\in\N)(\exists m\in\N)(\exists \eta_m>0)(\forall\eps\le\eta_m)$
\[(\forall x,y\in X) (d(x,y)\le\eps^m\implies\abs{u_{\eps}(x)-u_\eps(y)}\le \eps^n)
\]
\item $u$: $X^{(0,1)}\to\GenC$ is continuous (for the sharp topologies on $X^{(0,1)}$, resp.\ $\GenC$)
\item $u$ defines a continuous map $\widetilde X\to\GenC$ (for the sharp topologies on $\widetilde X$, resp.\ $\GenC$).
\end{enumerate}
\end{prop}
\begin{proof}
$(1)\implies (2)$: if $(2)$ doesn't hold, then we can find $n\in\N$, a decreasing sequence $(\eps_m)_{m\in\N}$ tending to $0$ and $x_{\eps_m}$, $y_{\eps_m}$ $\in X$ such that $d(x_{\eps_m},y_{\eps_m})\le\eps_m^m$ and $\abs{u_{\eps_m}(x_{\eps_m})-u_{\eps_m}(y_{\eps_m})}> \eps_m^n$, for each $m\in\N$. Let $x_\eps=y_\eps\in X$ arbitrary for $\eps\in (0,1)\setminus \{\eps_m:m\in\N\}$. For $x=(x_\eps)_\eps$, $y=(y_\eps)_\eps\in X^{(0,1)}$, we then have $x\sim y$, but $u(x)\ne u(y)$.\\
$(2)\implies (3)$: let $r\in\R^+$. For $n\ge -\ln r$, we find $m\in\N$ such that for $x=(x_\eps)_\eps, y=(y_\eps)_\eps \in X^{(0,1)}$, if $\sharpnorm{d(x,y)}<e^{-m}$, hence $d(x_\eps,y_\eps)\le\eps^m$, for sufficiently small $\eps$, then $\abs{u_\eps(x_\eps)-u_\eps(y_\eps)}\le\eps^n$, for sufficiently small $\eps$, so $\sharpnorm{u(x)-u(y)}\le e^{-n}\le r$.\\
$(3)\implies(1)$: if $x$, $y$ $\in X^{(0,1)}$ and $x\sim y$, then $\sharpnorm{d(x,y)}\le\delta$, $\forall\delta\in\R^+$. By continuity, $\sharpnorm{u(x)-u(y)}\le r$, $\forall r\in\R^+$; hence $u(x)=u(y)$ in $\GenC$.\\
$(1)\&(3)\implies(4)$, $(4)\implies(1)$: clear.
\end{proof}

\begin{df}
Let $X$ be a compact metric space. Then we denote by $\Gensharp X$ the set of all elements of $\Gencnt{X}$ that satisfy the equivalent properties of the previous proposition. We call $\Gensharp X$ the algebra of sharply continuous generalized functions on $X$.
\end{df}

\begin{prop}\label{prop_Gensharp_Banach}
Let $X$ be a compact metric space. Then $\Gensharp X =\bigcap_{x, y\in X^{(0,1)}, x\sim y}\{u\in\Gencnt X: u(x) = u(y)\}$ is a Banach $\GenC$-subalgebra of $\Gencnt X$.
\end{prop}
\begin{proof}
The equality holds by proposition \ref{prop_sharp_cnt}. Hence $\Gensharp X$ is a closed (and thus complete) subset of $\Gencnt X$, and also a sub-$\GenC$-algebra of $\Gencnt X$.
\end{proof}

\begin{lemma}\label{lemma_strictly_internal_Gencnt}
Let $X$ be a compact metric space. For $x = (x_\eps)_\eps$, $y=(y_\eps)_\eps\in X^{(0,1)}$,
\[
\{u\in\Gencnt X: u(x) = u(y)\} = [(\{u\in\Cnt(X): u(x_\eps)=u(y_\eps)\})_\eps].
\]
\end{lemma}
\begin{proof}
Let $u=[(u_\eps)_\eps]\in\Gencnt X$ with $u(x)=u(y)$, i.e., $(u_\eps(x_\eps) - u_\eps(y_\eps))_\eps\in\Null_{\C}$. Let $n_\eps(x):= (u_\eps(x_\eps) - u_\eps(y_\eps))\min\big(1, \frac{d(x,x_\eps)}{d(x_\eps,y_\eps)}\big)$, if $x_\eps\ne y_\eps$ and $n_\eps:= 0$, otherwise. Then $(n_\eps)_\eps\in\Null_{\Cnt(X)}$, and $u = [(u_\eps + n_\eps)_\eps]$ with $(u_\eps + n_\eps)(x_\eps) = (u_\eps + n_\eps)(y_\eps)$, $\forall \eps$. The converse inclusion holds by definition.
\end{proof}

This suggests the following definition:
\begin{df}
Let $B$ be a Banach $\C$-algebra. Then $\Alg\subseteq\Gen_B$ is a \defstyle{strictly internal subalgebra} of $\Gen_B$ iff $\Alg=[(A_\eps)_\eps]$ where $A_\eps$ are Banach subalgebras of $B$, for each $\eps$. Such a representative of $\Alg$ will be called a \defstyle{strict} representative.
\end{df}
\begin{rem}
Since $[(A_\eps)_\eps] = [(\overline A_\eps)_\eps]$ for internal subsets of $\Gen_B$ \cite[Cor.~2.2]{OV_internal}, any $\Alg=[(A_\eps)_\eps]$ where $A_\eps$ are subalgebras (with $1$) of $B$, is a strictly internal subalgebra. Since every internal set of $\Gen_B$ is closed \cite[Prop.~2.3]{OV_internal}, any internal subalgebra $\Alg$ of $\Gen_B$ is a Banach $\GenC$-algebra.
\end{rem}

\begin{prop}\label{prop_Gensharp_is_intersection_of_strictly_internal}
Let $X$ be a compact metric space. Then $\Gensharp X$ is an intersection of strictly internal subalgebras of $\Gencnt X$.
\end{prop}
\begin{proof}
By proposition \ref{prop_Gensharp_Banach} and lemma \ref{lemma_strictly_internal_Gencnt}.
\end{proof}

\section{Strictly internal Banach $\GenC$-algebras}\label{section_internal_Banach_algebra}
We now study the properties of strictly internal Banach $\GenC$-algebras, as introduced in the previous section.

\begin{lemma}\label{lemma_GenB_invertible}
Let $B$ be a Banach $\C$-algebra and $\Alg$ a strictly internal subalgebra of $\Gen_B$ with strict representative $(A_\eps)_\eps$. Let $u=[(u_\eps)_\eps]\in\Alg$ with $u_\eps\in A_\eps$, for small $\eps$. Let $S\subseteq (0,1)$ with $e_S\ne 0$. Then $u$ is invertible w.r.t.\ $S$ in $\Alg$ iff both the following conditions hold:
\begin{enumerate}
\item $u_\eps$ is invertible in $A_\eps$, for small $\eps\in S$;
\item there exists $N\in\N$ such that $\norm{\inv u_\eps}\le\eps^{-N}$, for small $\eps\in S$.
\end{enumerate}
\end{lemma}
\begin{proof}
$\Rightarrow$: let $uv = vu = e_S$, with $u=[(u_\eps)_\eps]$, $v=[(v_\eps)_\eps]\in\Alg$. W.l.o.g., $u_\eps, v_\eps\in A_\eps$, $\forall\eps$. Let $n_\eps:= u_\eps v_\eps - 1\in A_\eps$, for $\eps\in S$ and $n_\eps=0$ otherwise. Then $(n_\eps)_\eps\in \Null_B$. In particular, for small $\eps\in S$, $\norm{n_\eps} \le 1/2$, whence $\inv{(u_\eps v_\eps)}\in A_\eps$ exists, (since $A_\eps$ are Banach $\C$-algebras) and $\norm{\inv{(u_\eps v_\eps)}} \le \frac{1}{1-\norm{n_\eps}} \le 2$. Then $v_\eps \inv{(u_\eps v_\eps)}\in A_\eps$ is a right inverse for $u_\eps$, for small $\eps\in S$. Similarly, $\inv{(v_\eps u_\eps)}v_\eps\in A_\eps$ is a left inverse for $u_\eps$, for small $\eps\in S$. Hence $u_\eps$ is invertible in $A_\eps$, for small $\eps\in S$.\\
$\Leftarrow$: let $v_\eps:=\inv u_\eps\in A_\eps$, if $\eps\in S$ sufficiently small, and $v_\eps=0$ otherwise. Then $v:=[(v_\eps)_\eps]\in \Alg$ and $u v = v u = e_S$.
\end{proof}

\begin{cor}\label{cor_GenB_strictly_noninvertible}
Let $B$ be a Banach $\C$-algebra and $\Alg$ a strictly internal subalgebra of $\Gen_B$ with strict representative $(A_\eps)_\eps$. Let $u=[(u_\eps)_\eps]\in\Alg$ with $u_\eps\in A_\eps$, for small $\eps$. Then $u$ is strictly non-invertible in $\Alg$ iff
\[(\forall n\in\N) (\exists \eta_n\in (0,1)) (\forall \eps\le\eta_n) (u_\eps \text{ is not invertible in }A_\eps \text{ or } \norm[]{\inv u_\eps}\ge \eps^{-n}).\]
\end{cor}
\begin{proof}
$\Rightarrow$: If the conclusion does not hold, we find $N\in\N$ and a decreasing sequence $(\eps_n)_{n\in\N}$ tending to $0$ such that $u_{\eps_n}$ is invertible in $A_{\eps_n}$ and $\norm[]{\inv u_{\eps_n}}\le \eps_n^{-N}$, for each $n\in\N$. By lemma \ref{lemma_GenB_invertible}, $u$ is invertible w.r.t.\ $T:=\{\eps_n: n\in\N\}$, contradicting the hypotheses.\\
$\Leftarrow$: by lemma \ref{lemma_GenB_invertible}.
\end{proof}

\begin{prop}\label{prop_GenB_invertible}
Let $B$ be a Banach $\C$-algebra and $\Alg$ be an intersection of strictly internal subalgebras of $\Gen_B$. Let $u\in\Alg$.
\begin{enumerate}
\item Let $S\subseteq(0,1)$ with $e_S\ne 0$. If $\norm{u}e_S\ll 1$ in $\GenR$, then $1-u$ is invertible w.r.t.\ $S$ in $\Alg$.
\item $\spec_{\Alg}(u)\subseteq\{\lambda\in\GenC: \abs{\lambda}\le \norm{u}\}$.
\end{enumerate}
\end{prop}
\begin{proof}
(1) Let first $\Alg$ be strictly internal with strict representative $(A_\eps)_\eps$ and $u=[(u_\eps)_\eps]$ with $u_\eps\in A_\eps$, $\forall \eps$. There exists $m\in\N$ such that $\norm{u_\eps}\le 1 - \eps^m$, for small $\eps\in S$. Hence $(1-u_\eps)^{-1}\in A_\eps$ exists and $\norm{(1-u_\eps)^{-1}}\le \frac{1}{1-\norm{u_\eps}}\le\eps^{-m}$, for small $\eps\in S$. By lemma \ref{lemma_GenB_invertible}, $1-u$ is invertible w.r.t.\ $S$ in $\Alg$.\\
Now let $\Alg=\bigcap_{i\in I} \Alg_i$, with $\Alg_i$ strictly internal, for each $i$. Let $u\in\Alg$. By the previous case, there exists an inverse $v_i\in \Alg_i$ w.r.t.\ $S$ of $1-u$. By lemma \ref{lemma_invertible_wrt_S}, $v_i e_S$ does not depend on $i$, and thus is an inverse w.r.t.\ $S$ of $1-u$ in $\Alg$.\\
(2) Let $\lambda\in\GenC$ with $\abs{\lambda}\nleq \norm{u}$. Then there exists $S\subseteq(0,1)$ with $e_S\ne 0$ and $m\in\N$ such that $\abs{\lambda}e_S\ge\norm{u}e_S + \caninf^m e_S$. By lemma \ref{lemma_invertible_wrt_S}, $\lambda$ is invertible w.r.t.\ $S$, say $\lambda\mu=e_S$. Hence $\norm{\mu u}e_S \le (1 - \abs{\mu}\caninf^m) e_S\ll 1$. By part 1, $1 - \mu u$ is invertible w.r.t.\ $S$ in $\Alg$. Hence also $\lambda - u$ is invertible w.r.t.\ $S$ in $\Alg$ by lemma \ref{lemma_invertible_wrt_S}, and $\lambda\notin \spec_\Alg(u)$.
\end{proof}

\begin{prop}\label{prop_spec_GenB_internal}
Let $B$ be a Banach $\C$-algebra and $\Alg$ a strictly internal subalgebra of $\Gen_B$. Let $u\in\Alg$.
\begin{enumerate}
\item $\resolv_\Alg(u)$ is the union of an increasing sequence of internal subsets of $\GenC$.
\item $\spec_\Alg(u)$ is the intersection of a decreasing sequence of internal subsets of $\GenC$.
\item If $C\subseteq \resolv_{\Gen_B}(u)$ and $C$ is internal, then $C\cap \resolv_\Alg(u)$ and $C\cap \spec_\Alg(u)$ are internal.
\end{enumerate}
\end{prop}
\begin{proof}
Let $(A_\eps)_\eps$ be a strict repres.\ of $\Alg$ and $u=[(u_\eps)_\eps]$ with $u_\eps\in A_\eps$, for small $\eps$.\\
(1) For $n\in\N$, let $C_{n,\eps}=\{\lambda\in \resolv_{A_\eps}(u_\eps): \norm{(u_\eps - \lambda)^{-1}}\le\eps^{-n}\}$. Then $\bigcup_{n\in\N}[(C_{n,\eps})]=\resolv_\Alg(u)$ by lemma \ref{lemma_GenB_invertible}.\\
(2) For $n\in\N$, let $C_{n,\eps}=\spec_{A_\eps}(u_\eps)\bigcup \{\lambda\in\resolv_{A_\eps}(u_\eps): \norm{(u_\eps - \lambda)^{-1}}\ge \eps^{-n}\}$. Then $\spec_\Alg(u) = \bigcap_{n\in\N}[(C_{n,\eps})_\eps]$ lemma \ref{lemma_GenB_invertible} and by corollary \ref{cor_GenB_strictly_noninvertible}.\\
(3) Let $C=[(C_\eps)_\eps]$. We show that $C\cap\resolv_{\Alg}(u) = [(C_\eps\cap \resolv_{A_\eps}(u_\eps))_\eps]$ and $C\cap\spec_{\Alg}(u) = [(C_\eps\cap \spec_{A_\eps}(u_\eps))_\eps]$.
All four sets are contained in $C\subseteq\resolv_{\Gen_B}(u)$. By lemma \ref{lemma_GenB_invertible}, $(u_\eps - \lambda_\eps)^{-1}\in B$ exist for small $\eps$, and $(\norm{(u_\eps - \lambda_\eps)^{-1}})_\eps$ is moderate, for each $[(\lambda_\eps)_\eps]$ in any of the four sets. Then the equalities follow by lemma \ref{lemma_GenB_invertible} and corollary \ref{cor_GenB_strictly_noninvertible}.
\end{proof}

\begin{df}
Let $B$ be a Banach $\C$-algebra. Let $\Alg$ be a strictly internal subalgebra of $\Gen_B$. For $u\in\Alg$, we define the \defstyle{interior spectrum} of $u$ as
\[
\intspec_{\Alg}(u)=\bigcup_{(A_\eps)_\eps \text{ strict repr.\ of } \Alg} \bigcup_{(u_\eps)_\eps\in A_\eps^{(0,1)} \text{ repr.\ of }u} [(\spec_{A_\eps}(u_\eps))_\eps].
\]
If the algebra is clear from the context, we simply write $\intspec(u)$.
\end{df}

\begin{prop}
Let $B$ be a commutative Banach $\C$-algebra and $\Alg$ a strictly internal subalgebra of $\Gen_B$. Let $u\in\Alg$. Then
\[
\intspec_{\Alg}(u)= \bigcup_{(A_\eps)_\eps\text{ strict rep.\ of }\Alg}[(\spec_{A_\eps} (u_\eps))_\eps],
\]
where $[(\spec_{A_\eps} (u_\eps))_\eps]$ is independent of the representative $(u_\eps)_\eps$ of $u$ with $u_\eps\in A_\eps$, $\forall\eps$.
\end{prop}
\begin{proof}
Let $u=[(u_\eps)_\eps]=[(v_\eps)_\eps]$ with $u_\eps, v_\eps\in A_\eps$, $\forall\eps$. Since $\spec_{A_\eps}(u_\eps)\subseteq \spec_{A_\eps}(u_\eps-v_\eps) + \spec_{A_\eps}(v_\eps)\subseteq \overline B(0, \norm{u_\eps-v_\eps}) + \spec_{A_\eps}(v_\eps)$ \cite[3.2.10]{Kadison}, $[(\spec_{A_\eps}(u_\eps))_\eps]\subseteq [(\spec_{A_\eps}(v_\eps))_\eps]$.
\end{proof}

\begin{thm}\label{thm_non_empty_spectra}
Let $B$ be a Banach $\C$-algebra and $\Alg$ a strictly internal subalgebra of $\Gen_B$. Then
\begin{enumerate}
\item $\intspec_{\Alg}(u) \subseteq \spec_{\Alg}(u)$, $\forall u\in\Alg$.
\item $\intspec_{\Alg}(u) \subseteq\functspec_{\Alg}(u) \subseteq\spec_{\Alg}(u)$, $\forall u\in\Alg$, if $B$ is commutative.
\item $\intspec_{\Alg}(u)\ne\emptyset$, $\forall u\in\Alg$.
\end{enumerate}
\end{thm}
\begin{proof}
(1) Let $\lambda\in\intspec_\Alg(u)$. So $\lambda_\eps\in\spec_{A_\eps}(u_\eps)$, $\forall\eps$, where $(A_\eps)_\eps$ is a strict representative of $\Alg$, $\lambda=[(\lambda_\eps)_\eps]$ and $u=[(u_\eps)_\eps]$ with $u_\eps \in A_\eps$, $\forall\eps$. If $\lambda\notin\spec_\Alg(u)$, then there exists $S\subseteq (0,1)$ with $e_S\ne 0$ such that $u-\lambda$ is invertible w.r.t.\ $S$. By lemma \ref{lemma_GenB_invertible}, $u_\eps-\lambda_\eps$ is invertible in $A_\eps$, for small $\eps\in S$, a contradiction.\\
(2) Let $\lambda\in\intspec_\Alg(u)$. So $\lambda_\eps\in\spec_{A_\eps}(u_\eps)$, $\forall\eps$, where $(A_\eps)_\eps$ is a strict representative of $\Alg$, $\lambda=[(\lambda_\eps)_\eps]$ and $u=[(u_\eps)_\eps]$ with $u_\eps \in A_\eps$, $\forall\eps$. As $A_\eps$ are commutative Banach $\C$-algebras, there exist multiplicative $\C$-linear functionals $m_\eps\ne 0$ with $m_\eps(u_\eps)=\lambda_\eps$, $\forall\eps$ (e.g., \cite[3.2.11]{Kadison}). As $\abs{m_\eps(v)}\le\norm{v}$, $\forall v\in A_\eps$, the map $m$: $\Alg\to\GenC$: $m([(v_\eps)_\eps])=[(m_\eps(v_\eps))_\eps]$ is well-defined and is a multiplicative $\GenC$-linear functional with $m(1)=1$. Hence $m\in\Smlf{\Alg}$ and $m(u)=\lambda$. So $\intspec_\Alg(u)\subseteq\functspec_\Alg(u)$. The other inclusion follows by lemma \ref{lemma_fsp_sub_sp}.\\
(3) Let $u=[(u_\eps)_\eps]\in\Alg=[(A_\eps)_\eps]$ with $A_\eps$ Banach $\C$-algebras and $u_\eps\in A_\eps$, $\forall\eps$. Then there exist $\lambda_\eps\in\spec_{A_\eps}(u_\eps)$ with $\abs{\lambda_\eps}\le\norm{u_\eps}$, $\forall\eps$ (e.g., \cite[3.2.3]{Kadison}).
So $(\lambda_\eps)_\eps\in\Mod_\C$, and $\lambda:=[(\lambda_\eps)_\eps]\in\intspec_\Alg(u)$.
\end{proof}
\begin{rem}
It is necessary to consider only strict representatives of $\Alg$ in the definition of $\intspec_\Alg(u)$ in order to have $\intspec_\Alg(u)\subseteq \spec_\Alg(u)$. E.g., let $B=\Cnt([0,1])$, $u(x)=x+1$, $A_\eps=\C[x]$, $\forall\eps$. By Weierstrass' approximation theorem, $\overline{\C[x]}=\Cnt([0,1])$, so $[(A_\eps)_\eps]=\Gen_B$. Hence $0\notin\spec_{\Gen_B}(u)$. But $0\in \spec_{A_\eps}(x+1)$, for each $\eps$.
\end{rem}

\begin{cor}\label{cor_non_empty_spectra}
Let $B$ be a Banach $\C$-algebra and $\Alg$ a sub-$\GenC$-algebra (with $1$) of $\Gen_B$. Then $\spec_\Alg(u)\ne\emptyset$, $\forall u\in\Alg$.
\end{cor}
\begin{proof}
For $u\in\Alg$, $\emptyset\ne \intspec_{\Gen_B}(u)\subseteq\spec_{\Gen_B}(u)\subseteq\spec_\Alg(u)$.
\end{proof}

The following is a dual statement of theorem \ref{thm_resolv_determines_spec}:
\begin{prop}\label{prop_spec_determines_resolv_internal}
Let $B$ be a Banach $\C$-algebra and $\Alg$ an intersection of strictly internal subalgebras of $\Gen_B$. Let $u\in\Alg$. Then
\begin{align*}
\resolv_\Alg(u) &= \{\lambda\in\GenC: (\forall S\subseteq (0,1) \text{ with } e_S\ne 0) (\lambda e_S \notin \spec_\Alg(u)e_S)\}\\
&= \{\lambda\in\GenC: \abs{\lambda - \mu}\gg 0, \ \forall \mu\in\spec_\Alg(u)\}.
\end{align*}
\end{prop}
\begin{proof}
The second equality follows from lemma \ref{lemma_far_away_set}. For the first equality:\\
$\subseteq$: if $\lambda e_S = \mu e_S$, for some $\mu\in \spec_\Alg(u)$ and $S\subseteq(0,1)$ with $e_S\ne 0$, then $u-\lambda$ is not invertible w.r.t.\ $S$ in $\Alg$ by lemma \ref{lemma_invertible_wrt_S}.\\
$\supseteq$: first, let $\Alg$ be strictly internal with strict representative $(A_\eps)_\eps$. Let $u=[(u_\eps)_\eps]$ with $u_\eps\in A_\eps$, $\forall\eps$. Let $\lambda=[(\lambda_\eps)_\eps]\in\GenC$. If $u-\lambda$ is not invertible in $\Alg$, then by lemma \ref{lemma_GenB_invertible}, we can find a decreasing sequence $(\eps_n)_{n\in\N}$ tending to $0$ such that $u_{\eps_n} - \lambda_{\eps_n}$ is not invertible in $A_\eps$ or $\norm[]{(u_{\eps_n} - \lambda_{\eps_n})^{-1}}\ge\eps_n^{-n}$, $\forall n\in\N$. Let $S=\{\eps_n: n\in\N\}$. Then $e_S\ne 0$. Let $\mu\in\spec_\Alg(u)$ ($\spec_\Alg(u)\ne\emptyset$ by corollary \ref{cor_non_empty_spectra}). By corollary \ref{cor_GenB_strictly_noninvertible}, it follows that $\lambda e_S + \mu e_{\co S}\in \spec_\Alg(u)$. Hence $\lambda e_S\in\spec_\Alg(u)e_S$.\\
Now let $\Alg=\bigcap_{i\in I}\Alg_i$ with $\Alg_i$ strictly internal, for each $i$ and let $\lambda\in\GenC$ such that for each $S\subseteq(0,1)$ with $e_S\ne 0$, $\lambda e_S\notin\spec_{\Alg}(u) e_S$. As $\spec_{\Alg_i}(u)\subseteq\spec_\Alg(u)$, also $\lambda e_S\notin\spec_{\Alg_i}(u) e_S$ for each $i$. By the previous case and lemma \ref{lemma_spec_inclusion}, $\lambda\in \bigcap_{i\in I}\resolv_{\Alg_i}(u) = \resolv_\Alg(u)$.
\end{proof}

\begin{ex}
Let $H$ be an infinite-dimensional separable Hilbert space and $B=\cLin(H)$ the Banach $\C$-algebra of continuous $\C$-linear operators on $H$. Then there exists $u=[(u_\eps)_\eps]=[(v_\eps)_\eps]\in \Gen_B$ with $[(\spec(u_\eps))_\eps] \ne [(\spec(v_\eps))_\eps]$, $[(\specrad(u_\eps))_\eps] \ne [(\specrad(v_\eps))_\eps]$ and $\sharpnorm{[(\specrad(u_\eps))_\eps]} < \specrad_{\Gen_B}(u)$.
\end{ex}
\begin{proof}
By \cite[\S 3.4]{Aupetit}, there exist $v, v_k\in B$ (for each $k\in\N$) with $\norm{v_k-v}\le e^{-k}$, $\spec(v_k)=\{0\}$, for each $k\in\N$ and there exists $\lambda\in\spec(v)$ with $e^{-2}\le\abs\lambda\le 1$. Let $u:=[(v)_\eps]\in \Gen_B$. Then there exists $\lambda\in[(\spec(v))_\eps]\subseteq \spec(u)$ (theorem \ref{thm_non_empty_spectra}) with $e^{-2}\le\abs\lambda\le 1$. Let $u_\eps:=v_k$, if $1/k\le \eps < 1/(k-1)$ ($k\in\N$, $k>1$). Then $\norm{u_\eps-v}\le e^{-1/\eps}$, for each $\eps$. Hence $(\norm{u_\eps-v})_\eps\in\Null_\R$, and $u=[(u_\eps)_\eps]$. Yet $[(\spec(u_\eps))_\eps]=\{0\}$.
\end{proof}

\subsection{The spectral mapping theorem}
\begin{lemma}\label{lemma_product_of_invertible}
Let $B$ be a Banach $\C$-algebra and $\Alg$ a strictly internal subalgebra of $\Gen_B$. Let $u_1,\dots, u_m\in\Alg$ ($m\in\N$). If $u_1\cdots u_m$ is strictly non-invertible in $\Alg$, then there exists $v\in\interl(\{u_1,\dots,u_m\})$ that is strictly non-invertible in $\Alg$.
\end{lemma}
\begin{proof}
Let $\Alg=[(A_\eps)_\eps]$ and $u_j=[(u_{j,\eps})_\eps]$ with $A_\eps$ closed subalgebras of $B$ and $u_{j,\eps}\in A_\eps$, $\forall\eps$, $\forall j$. Let $n\in\N$. By corollary \ref{cor_GenB_strictly_noninvertible}, $u_{1,\eps}\cdots u_{m,\eps}$ is not invertible in $A_\eps$ or $\norm{(u_{1,\eps}\cdots u_{m,\eps})^{-1}}\ge \eps^{-n}$, for $\eps\le\eta_n$. Hence there exists $j\in\{1,\dots,m\}$ such that $u_{j,\eps}$ is not invertible in $A_\eps$ or $\norm[]{\inv u_{j,\eps}}\ge \eps^{-n/m}$, for $\eps\le\eta_n$. W.l.o.g., $(\eta_n)_n$ decreasingly tends to $0$. For $\eta_{n+1}< \eps\le \eta_n$, let $\eps\in S_j$ iff $j$ is the smallest index with the latter property. Then $e_{S_1} u_1 + \cdots + e_{S_m} u_m\in \interl(\{u_1,\dots,u_m\})$ is strictly non-invertible in $\Alg$ by corollary \ref{cor_GenB_strictly_noninvertible}.
\end{proof}

\begin{thm}\label{thm_poly_spec_calc_internal}
Let $B$ be a Banach $\C$-algebra. Let $\Alg$ be the union of an increasing sequence of strictly internal subalgebras of $\Gen_B$. Let $u\in\Alg$. Let $p(z)\in\GenC[z]$ with $1$ as its leading coefficient. Then $p(\spec_\Alg(u))=\spec_\Alg(p(u))$.
\end{thm}
\begin{proof}
First, let $\Alg$ be strictly internal. Let $\lambda\in \spec_\Alg(p(u))$. By lemma \ref{lemma_polynomial_eqn}, $p(z) - \lambda = (z-z_1)\cdots(z-z_m)$, for some $z_j=[(z_{j,\eps})_\eps]\in\GenC$. By lemma \ref{lemma_product_of_invertible}, there exists $v\in\interl(\{u-z_1,\dots, u-z_m\})$ that is strictly non-invertible in $\Alg$. Hence there exists a partition $\{S_1,\dots,S_m\}$ of $(0,1)$ such that $u - (e_{S_1} z_1 + \cdots + e_{S_m} z_m)$ is strictly non-invertible in $\Alg$, i.e., $e_{S_1} z_1 + \cdots + e_{S_m} z_m\in \spec_\Alg(u)$. Since $p(z_1e_{S_1} + \cdots + z_m e_{S_m})=\lambda$, $\lambda\in p(\spec_\Alg(u))$.\\
Now let $\Alg = \bigcup_{n\in\N}\Alg_n$, with $(\Alg_n)_{n\in\N}$ an increasing sequence of strictly internal subalgebras of $\Gen_B$. Let $m\in\N$ such that $u\in \Alg_m$. Let $\lambda\in \spec_\Alg(p(u))$. By the previous case and lemma \ref{lemma_spec_inclusion}, $\spec_\Alg(p(u))= \bigcap_{n\ge m}\spec_{\Alg_n}(p(u))\subseteq \bigcap_{n\ge m}p(\spec_{\Alg_n}(u))$. By proposition \ref{prop_spec_GenB_internal}, $\spec_{\Alg_n}(u) = \bigcap_{k\in\N} C_{n,k}$, where $C_{n,k}\subseteq\GenC$ are internal, for each $n\ge m$. By lemma \ref{lemma_polynomial_eqn}, $\{\mu\in\GenC: p(\mu)=\lambda\}$ is internal and sharply bounded. Further, since $\spec_{\Alg_n}(u)$ are totally ordered, the family of sets consisting of $\{\mu\in\GenC: p(\mu)=\lambda\}$ and $C_{n,k}$ ($k\in\N$, $n\ge m$), has the finite intersection property. Hence by saturation \cite[Thm.~2.12]{OV_internal}, $\lambda\in p\big(\bigcap_{n\ge m} \spec_{\Alg_n}(u)\big) = p(\spec_\Alg(u))$.% by lemma \ref{lemma_spec_inclusion}.
\\
The converse inclusion follows by theorem \ref{thm_poly_spec_calc}.
\end{proof}

\begin{prop}\label{prop_spec_au_GenB}
Let $B$ be a Banach $\C$-algebra and $\Alg$ a strictly internal subalgebra of $\Gen_B$. Let $u\in\Alg$ and $a\in\GenC$. Then $\spec_\Alg(au)=a\spec_\Alg(u)$.
\end{prop}
\begin{proof}
By proposition \ref{prop_spec_au}, proposition \ref{prop_spec_GenB_internal} and corollary \ref{cor_non_empty_spectra}.
\end{proof}
%probably this can also be extended to the union of an increasing sequence of strictly internal algebras

\subsection{The spectral radius formula}
\begin{df}(cf.\ \cite{HV_analytic})
Let $\emptyset\ne A\subseteq \GenR^d$. Then $\tGen(A):=\EMod(A)/\Null(A)$, where
\begin{align*}
\EMod(A)=\,&\big\{(u_\eps)_\eps\in\Cnt[\infty](\R^d)^{(0,1)}: (\forall\alpha\in\N^d) \big(\forall [(x_\eps)_\eps]\in A\big)
\big((\partial^\alpha u_\eps(x_\eps))_\eps\in\Mod_{\C}\big)\big\},\\
\Null(A)=\,&\big\{(u_\eps)_\eps\in\Cnt[\infty](\R^d)^{(0,1)}: (\forall\alpha\in\N^d) \big(\forall [(x_\eps)_\eps]\in A\big)
\big((\partial^\alpha u_\eps(x_\eps))_\eps\in\Null_{\C}\big)\big\}.
\end{align*}
(Here $\forall [(x_\eps)_\eps]\in A$ means: for each representative $(x_\eps)_\eps$ of an element of $A$.)\\
The action of $u=[(u_\eps)_\eps]\in\tGen(A)$ on a generalized point $\tilde x= [(x_\eps)_\eps]\in A$ is defined as $u(\tilde x):= [(u_\eps(x_\eps))_\eps]\in\GenC$. If $A$ is open, then $u$ is completely determined by its action on generalized points, and we thus identify elements of $\tGen(A)$ with particular pointwise maps $A\to\GenC$.\\
Let $\Omega$ be an open subset of $\GenC$. Then $\tGen_H(\Omega)$ is the differential algebra consisting of those $u\in\tGen(\Omega)$ with $\bar \partial u := \frac{1}{2}(\partial_x + i \partial_y) u = 0$. For $u\in\tGen_H(\Omega)$ and $\tilde z\in \Omega$, we write $u'(\tilde z): = \partial_x u(\tilde z) = -i\partial_y u(\tilde z)$. Iterated derivatives are denoted by $D^k$ ($k\in\N$).
\end{df}
Properties of analytic generalized functions on generalized domains are studied in \cite{HV_analytic}.

\begin{lemma}\label{lemma_specrad}
Let $B$ be a Banach $\C$-algebra and $u\in\Gen_B$. Let $s\in\R$, $s < -\ln (\specrad_{\Gen_B}(u))$. Then $\{\tilde z\in\GenC: \abs[]{\tilde z}\ge \caninf^s\}\subseteq\resolv_{\Gen_B}(u)$.
\end{lemma}
\begin{proof}
Let $s'\in\R$ with $s<s'< -\ln(\specrad_{\Gen_B}(u))$. Let $\mu\in\spec_{\Gen_B}(u)$. Since $\sharpnorm{\mu}\le \specrad_{\Gen_B}(u)$, $\abs\mu\le \caninf^{s'}$. So if $\lambda\in\GenC$ and $\abs{\lambda}\ge \caninf^s$, then $\abs{\lambda - \mu}\ge \abs{\lambda} - \abs\mu \ge \caninf^s - \caninf^{s'}\gg 0$. By proposition \ref{prop_spec_determines_resolv_internal}, $\lambda \in\resolv_{\Gen_B}(u)$.
\end{proof}

\begin{prop}\label{prop_analytic_resolvent}
Let $B$ be a Banach $\C$-algebra with topological dual $B'$ and $u\in\Gen_B$. Let $R$: $\resolv_{\Gen_B}(u)\to\Gen_B$: $R(\lambda)=(\lambda - u)^{-1}$. Let $s\in\R$, $s < -\ln (\specrad_{\Gen_B}(u))$. Then $f\comp R\in \tGen_H(\{\tilde z\in\GenC: \abs[]{\tilde z} \gg \caninf^s\})$, for each $f\in\Gen_{B'}$.
\end{prop}
\begin{proof}
First, recall that $f$ acts as a (so-called basic $\GenC$-linear) pointwise map $\Gen_B\to\GenC$ \cite[\S 1.1.2]{GV_Hilbert}. Let $u=[(u_\eps)_\eps]$ and $f=[(f_\eps)_\eps]$ with $f_\eps\in B'$, for each $\eps$. Then $\{z\in\C: \abs z > \eps^s\}\subseteq \resolv_{B}(u_\eps)$, for small $\eps$, since supposing the contrary, we can construct $\tilde z\in \GenC\setminus \resolv_{\Gen_B}(u)$ with $\abs[]{\tilde z}\ge \caninf^s$ using lemma \ref{lemma_GenB_invertible} (moderateness can be ensured, since $\abs z\le \norm{u_\eps}$, for each $z\in\spec_B(u_\eps)$), contradicting lemma \ref{lemma_specrad}. Hence $R_\eps$: $\{z\in\C: \abs{z} > \eps^s\}\to B$: $R_\eps(z) := (z-u_\eps)^{-1}$ (for small $\eps$) is well-defined and $[(f_\eps\comp R_\eps(z_\eps))_\eps]=f\comp R(\tilde z)$, for each $\tilde z= [(z_\eps)_\eps]\in\GenC$ with $\abs{z_\eps} > \eps^s$, $\forall\eps$. For each $k\in\N$, there exists $N\in\N$ such that $\sup_{\eps^s < \abs z\le \eps^{-k}} \abs{f_\eps\comp R_\eps(z)}\le\eps^{-N}$, for small $\eps$, since otherwise, we can construct $(z_\eps)_\eps\in\Mod_\C$ with $\abs{z_\eps}>\eps^s$, $\forall \eps$ and $(f_\eps\comp R_\eps(z_\eps))_\eps\notin\Mod_\C$. Since $f_\eps\comp R_\eps$ are analytic (e.g., \cite[Thm.\ VI.5]{ReedSimon}), $f\comp R\in\tGen_H(\{\tilde z\in\GenC: \abs{\tilde z} \gg \caninf^s\})$ by \cite[Lemma 4.2]{HV_analytic}%lemma_tGen_H_by_representatives
.
\end{proof}

\begin{lemma}\label{lemma_functional}
Let $B$ be a Banach $\C$-algebra, $u\in\Gen_B$ and $a\in\R^+$. If $\limsup_{n\to\infty}\sqrt[n]{\sharpnorm{f(u^n)}}\le a$, for each $f\in\Gen_{B'}$, then $\limsup_{n\to\infty}\sqrt[n]{\pseudonorm{u^n}}\le a$. 
\end{lemma}
\begin{proof}
Let $s\in\R$ with $s < -\ln a$. Then $\sup_{n\in\N}\sharpnorm[]{f(u^n)/\caninf^{ns}}<+\infty$. Applying the Banach-Steinhaus theorem \cite[Thm.~3.21]{Garetto2005} to the Banach $\GenC$-module $\Gen_{B'}$ and the continuous $\GenC$-linear functionals $T_n(f)=f(u^n)/\caninf^{ns}$: $\Gen_{B'}\to\GenC$ (for fixed $s$), we find $C_s\in\R^+$ such that for each $n\in\N$ and $f\in\Gen_{B'}$, $\sharpnorm{T_n(f)}\le C_s \pseudonorm{f}$. Choosing $f_n\in\Gen_{B'}$ with $\pseudonorm{f_n}=1$ and $f_n(u^n)=\norm{u^n}$ \cite[Prop.~3.23]{Garetto2005}, we find $\pseudonorm{u^n}\le C_s e^{-ns}$, $\forall n\in\N$. Hence $\limsup_{n\to\infty}\sqrt[n]{\pseudonorm{u^n}}\le e^{-s}$. Letting $s\to -\ln a$, $\limsup_{n\to\infty}\sqrt[n]{\pseudonorm{u^n}}\le a$.
\end{proof}

\begin{thm}\label{thm_specrad_formula}
Let $B$ be a Banach $\C$-algebra and $\Alg$ a Banach sub-$\GenC$-algebra of $\Gen_B$. Let $u\in\Alg$. Then $\specrad_\Alg(u)=\lim_{n\to\infty}\sqrt[n]{\pseudonorm{u^n}}$.
\end{thm}
\begin{proof}
Let first $\Alg=\Gen_B$ and $u\in\Gen_B$. Let $f\in\Gen_{B'}$. By proposition \ref{prop_lin_maps_on_normed_modules}, $\sharpnorm{f(u^n)}\le \pseudonorm f\pseudonorm{u^n}\le\pseudonorm f\pseudonorm u^n$, for each $n$, so $\limsup_{n\to\infty}\sqrt[n]{\sharpnorm{f(u^n)}}\le \pseudonorm{u}$, and by \cite[Lemma 4.16, Thm. 4.20]{HV_analytic}%lemma_convergence_radius and thm_holomorphic_charac
,
\[
V(\tilde z):=\sum_{n=0}^\infty f(u^n) \tilde z^n \in \tGen_H(\{\tilde z\in\GenC: \sharpnorm{\tilde z} < 1/{\pseudonorm{u}}\}).
\]
Further, if $\sharpnorm{\tilde z}<1/\pseudonorm{u}$, then $\pseudonorm{\tilde z u} < 1$, and $V(\tilde z) = f\big(\sum_{n=0}^\infty u^n \tilde z^n\big)=f((1 - u\tilde z)^{-1})$. Hence $V(\tilde z)$ coincides with $U(\tilde z):=\tilde z^{-1} f((\tilde z^{-1} - u)^{-1})$ as soon as $\tilde z$ is invertible and $\sharpnorm{\tilde z}<1/\pseudonorm{u}$. By proposition \ref{prop_analytic_resolvent} and \cite[Lemma 4.4]{HV_analytic}%lemma_composition_analytic
, $U \in \tGen_H(\{\tilde z\in\GenC: \caninf^m \ll \abs{\tilde z}\ll \caninf^{-s}\})$, for each $m\in\N$ and $s\in\R$ with $s < -\ln(\specrad_{\Gen_B}(u))$. Let $m\in\N$, $m - 1 > -s$ such that $V\in\tGen_H(\{\tilde z\in\GenC: \abs[]{\tilde z}\ll \caninf^{m-1}\})$. Let $(\chi_\eps)_\eps\in\EMod(\C)$ with $\chi_\eps(z)=1$, if $\abs{z}\le 2\eps^m$ and $\chi_\eps(z)=0$, if $\abs{z}\ge \eps^{m-1}/2$, and let $W :=  U + \chi (V - U)$ where $\chi:=[(\chi_\eps)_\eps]\in\tGen(\GenC)$. Then $W\in\tGen_H(\{\tilde z\in\GenC: \abs[]{\tilde z}\ll \caninf^{-s}\})\subseteq \tGen_H(\{\tilde z\in\GenC: \sharpnorm{\tilde z}< e^s\})$, and $W=V$ on a sharp neighbourhood of $0$. By \cite[Thms.\ 4.20 and 4.24]{HV_analytic}% thm_holomorphic_charac thm_acc_point_of_zeroes
, $\limsup_{n\to\infty}\sqrt[n]{\sharpnorm{D^n W(0)}}=\limsup_{n\to\infty}\sqrt[n]{\sharpnorm{f(u^n)}}\le e^{-s}$. 
By lemma \ref{lemma_functional}, $\limsup_{n\to\infty}\sqrt[n]{\pseudonorm{u^n}}\le e^{-s}$. Letting $s\to -\ln(\specrad_{\Gen_B}(u))$, $\limsup_{n\to\infty}\sqrt[n]{\pseudonorm{u^n}}\le \specrad_{\Gen_B}(u)$.\\
Now let $\Alg$ be any Banach sub-$\GenC$-algebra of $\Gen_B$ and $u\in\Alg$. By proposition \ref{prop_spectral_radius_bound},
$\specrad_\Alg(u)\le \lim_{n\to\infty}\sqrt[n]{\pseudonorm{u^n}} \le \specrad_{\Gen_B}(u)\le \specrad_{\Alg}(u)$.
\end{proof}
Consequently, for $u\in\Gen_B$, we can drop the index in $\specrad_\Alg(u)$.

\begin{cor}\label{cor_specrad_calculus}
Let $B$ be a Banach $\C$-algebra. Let $u,v\in\Gen_B$ with $uv=vu$. Then $\specrad(u+v) \le\specrad(u)+\specrad(v)$ and $\specrad(uv) \le\specrad(u)\specrad(v)$.
\end{cor}
\begin{proof}
By theorem \ref{thm_specrad_formula} and the properties of $\pseudonorm{.}$ (e.g., as in \cite[Cor.~3.2.10]{Aupetit}).
\end{proof}

\subsection{Stability of spectra in subalgebras}
As the path-connected components of $\GenC$ in the topological sense are trivial (they are the singletons), we use an alternative way to define a nontrivial notion:
\begin{df}
Let $C\subseteq\GenC$. We call $a,b\in C$ \defstyle{$\Gencnt{[0,1]}$-path-connected} in $C$ iff there exists $f\in\Gen_{\Cnt([0,1])}$ with $f(0)=a$, $f(1)=b$ and $f(\tau)\in C$, for each $\tau\in [0,1]^{(0,1)}$. We call $f$ a $\Gencnt{[0,1]}$-path connecting $a$ and $b$.
\end{df}
The relation $a \sim b$ iff $a,b$ are $\Gencnt{[0,1]}$-path-connected in $C$ is an equivalence relation. Hence we can call the equivalence classes $\Gencnt{[0,1]}$-path-connected components in $C$.

\begin{lemma}\label{lemma_internal_resolvent_subset}
Let B be a Banach $\C$-algebra and $u=[(u_\eps)_\eps]\in\Gen_B$. Let $\emptyset \ne C\subseteq\GenC$ be a sharply bounded internal set with sharply bounded representative $[(C_\eps)_\eps]$.
\begin{enumerate}
\item  If $C\cap\spec_{\Gen_B}(u)=\emptyset$, then there exist $M\in\N$ and $S\subseteq (0,1)$ with $e_S\ne 0$ such that $u_\eps - z$ is invertible in $B$ and $\norm{(u_\eps - z)^{-1}}\le \eps^{-M}$, $\forall\eps \in S$, $\forall z\in C_\eps$.
\item If $C\subseteq\resolv_{\Gen_B}(u)$, then there exists $M\in\N$ and $\eps_0\in (0,1)$ such that $u_\eps - z$ is invertible in $B$ and $\norm{(u_\eps - z)^{-1}}\le \eps^{-M}$, $\forall\eps \le\eps_0$, $\forall z\in C_\eps$.
\end{enumerate}
\end{lemma}
\begin{proof}
(1) Suppose that $(\forall m\in\N) (\exists \eps_m\in (0,1))$ $(\forall \eps\le\eps_m)$ $(\exists z\in C_\eps)$ $(u_\eps- z$ is not invertible in $B$ or $\norm{(u_\eps- z)^{-1}}\ge \eps^{-m})$. Then we can construct $\tilde z\in C$ such that $u- \tilde z$ is strictly not invertible in $\Gen_B$ by lemma \ref{lemma_GenB_invertible}.\\
(2) Similar.
\end{proof}

\begin{prop}\label{prop_spec_path_connect}
Let B be a Banach $\C$-algebra and $\Alg$ a strictly internal subalgebra of $\Gen_B$. Let $u\in\Alg$.
\begin{enumerate}
\item Let $\lambda,\mu$ be $\Gencnt{[0,1]}$-path-connected in $\GenC\setminus \spec_{\Gen_B}(u)$. If $\lambda\in\resolv_\Alg(u)$, then $\mu\notin\spec_\Alg(u)$.
\item Let $\lambda,\mu$ be $\Gencnt{[0,1]}$-path-connected in $\resolv_{\Gen_B}(u)$. Then $\lambda\in\spec_\Alg(u)$ iff $\mu\in\spec_\Alg(u)$.
\end{enumerate}
\end{prop}
\begin{proof}
(1) Let $(A_\eps)_\eps$ be a strict representative of $\Alg$ and $u=[(u_\eps)_\eps]$ with $u_\eps\in A_\eps$, $\forall\eps$. Let $f=[(f_\eps)_\eps]\in\Gen_{\Cnt([0,1])}$ be a $\Gencnt{[0,1]}$-path in $\GenC\setminus \spec_{\Gen_B}(u)$ connecting $\lambda = f(0)$ and $\mu = f(1)$. Applying lemma \ref{lemma_internal_resolvent_subset} to the internal set $f([0,1]^{(0,1)})$ with sharply bounded representative $[(f_\eps([0,1]))_\eps]$, we find $M\in\N$ and $S\subseteq (0,1)$ with $e_S\ne 0$ such that $\norm{(u_\eps - f_\eps(t))^{-1}}\le\eps^{-M}$, $\forall \eps\in S$ and $t\in [0,1]$. Let $\lambda\in\resolv_\Alg(u)$. By lemma \ref{lemma_GenB_invertible}, $(u_\eps - f_\eps(0))^{-1}\in A_\eps$, for small $\eps\in S$. Suppose $\mu\in\spec_\Alg(u)$. By the same lemma, $(u_\eps - f_\eps(1))^{-1}\notin A_\eps$, for small $\eps\in S$. Hence for small $\eps\in S$, $t_\eps:=\sup\{t\in[0,1]: (u_\eps - f_\eps(t))^{-1}\in A_\eps\}$ exists. Let $t_\eps:=0$, if $\eps\notin S$. Then $\tau:=[(t_\eps)_\eps]\in [0,1]^{(0,1)}$. By continuity of $f_\eps$, we find $t_{1,\eps}$ with $\abs{f_\eps(t_{1,\eps})-f_\eps(t_\eps)}\le\eps^{1/\eps}$ and $(u_\eps-f_\eps(t_{1,\eps}))^{-1}\in A_\eps$ and $t_{2,\eps}$ with $\abs{f_\eps(t_{2,\eps})-f_\eps(t_\eps)}\le\eps^{1/\eps}$ and $(u_\eps-f_\eps(t_{2,\eps}))^{-1}\notin A_\eps$, for small $\eps\in S$. Then $f(\tau)= [(f_\eps(t_{1,\eps}))_\eps] = [(f_\eps(t_{2,\eps}))_\eps]$ and $u-f(\tau)$ would be both invertible and not invertible w.r.t.\ $S$ by lemma \ref{lemma_GenB_invertible}.\\
%alternatively, we could have used \cite[Thm.~3.2.11]{Aupetit} for fixed $\eps\in S$.
(2) Similar.
\end{proof}

\begin{prop}\label{prop_stable_spec_subset_GenR}
Let B be a Banach $\C$-algebra and $\Alg$ an intersection of strictly internal subalgebras of $\Gen_B$. Let $u\in\Alg$ such that $\spec_{\Gen_B}(u)\subseteq\GenR$. Then $\spec_{\Alg}(u)=\spec_{\Gen_B}(u)$ and $\resolv_\Alg(u)=\resolv_{\Gen_B}(u)$.
\end{prop}
\begin{proof}
First, let $\Alg$ be strictly internal. Let $\lambda=[(\lambda_\eps)_\eps]\in \GenC\setminus\spec_{\Gen_B}(u)$. Let $M\in\N$ such that $\{\tilde z\in\GenC: \abs[]{\tilde z}\ge \caninf^{-M}\} \subseteq \resolv_\Alg(u)$. If the imaginary part of $\lambda_\eps\ge 0$, let $f_\eps(t)=\lambda_\eps + \eps^{-M }t i$. Otherwise, let $f_\eps(t)=\lambda_\eps - \eps^{-M}ti$. Then $f:=[(f_\eps)_\eps]\in\Gen_{\Cnt([0,1])}$ and $f([0,1]^{(0,1)})\subseteq \{\lambda\}\cup \GenC\setminus\GenR\subseteq\GenC\setminus\spec_{\Gen_B}(u)$. Hence $\lambda$, $f(1)$ are $\Gencnt{[0,1]}$-path-connected in $\GenC\setminus\spec_{\Gen_B}(u)$ and $f(1)\in\resolv_\Alg(u)$, since $\abs{f(1)}\ge \caninf^{-M}$. By proposition \ref{prop_spec_path_connect}, $\lambda\in\GenC\setminus\spec_\Alg(u)$. By proposition \ref{prop_spec_determines_resolv_internal}, also $\resolv_\Alg(u) = \resolv_{\Gen_B}(u)$.\\
If $\Alg=\bigcap_{i\in I} A_i$, with $A_i$ strictly internal subalgebras of $\Gen_B$, then $\resolv_\Alg(u) = \bigcap_{i\in I}\resolv_{\Alg_i}(u) = \resolv_{\Gen_B}(u)$ by lemma \ref{lemma_spec_inclusion}. By theorem \ref{thm_resolv_determines_spec}, also $\spec_\Alg(u) = \spec_{\Gen_B}(u)$.
\end{proof}

\section{Strictly internal $C^*$-algebras}
\begin{lemma}\label{lemma_selfadjoint_rep}
Let $B$ be a $*$-algebra and $\Alg=[(A_\eps)_\eps]$ an internal subalgebra of $\Gen_B$ such that $A_\eps$ are $*$-subalgebras of $B$, $\forall\eps$.
\begin{enumerate}
\item Let $u\in \Alg$ be self-adjoint. Then there exists a representative $(u_\eps)_\eps$ of $u$ with $u_\eps\in A_\eps$ self-adjoint, $\forall\eps$.
\item The set of all self-adjoint elements of $\Alg$ is internal.
\end{enumerate}
\end{lemma}
\begin{proof}
(1) Let $u=[(u_\eps)_\eps]$ with $u_\eps\in A_\eps$, $\forall\eps$. Then $u=[(\frac{u_\eps + u^*_\eps}{2})_\eps]$ with $\frac{u_\eps + u^*_\eps}{2}\in A_\eps$ self-adjoint, $\forall\eps$.\\
(2) By part~(1), $\{u\in\Alg: u=u^*\} = [(\{u\in A_\eps: u_\eps = u_\eps^*\})_\eps]$.
\end{proof}

There exists no analogue for normal elements of the previous lemma:
\begin{ex}
Let $\cLin(H)$ be the $C^*$-algebra of continuous $\C$-linear operators on a separable Hilbert space $H$. Then there exist self-adjoint $u$, $v\in\Gen_{\cLin(H)}$ that commute, but such that no representatives $(u_\eps)_\eps$ of $u$ and $(v_\eps)_\eps$ of $v$ satisfy $u_\eps v_\eps = v_\eps u_\eps$, $\forall\eps$.
\end{ex}
\begin{proof}
Let $(e_k)_{k\ge 1}$ be an orthonormal basis of $H$. Consider the weighted shifts $S_n(e_k):= \min(k/n,1) e_{k+1}$. Then $\norm{S_n}\le 1$ and $\norm{S_n S_n^* - S_n^* S_n}=1/n^2$, for each $n$. Let $A_n= \mathop\mathrm{Re} S_n$, $B_n = \mathop\mathrm{Im} S_n$. Let $u_\eps:= A_{\lceil\eps^{-1/\eps}\rceil}$ and $v_\eps:= B_{\lceil\eps^{-1/\eps}\rceil}$, for each $\eps\in (0,1)$. Then $u:=[(u_\eps)_\eps]$, $v:=[(v_\eps)_\eps]\in \Gen_{\cLin(H)}$ are self-adjoint and $u + i v$ is normal. Hence $u$, $v$ commute.
In \cite{BO81}, it is shown that there are no $A'_n$, $B'_n$ $\in \cLin(H)$ which commute and such that $\lim_{n\to\infty}\norm{A_n - A'_n} + \norm{B_n - B'_n} = 0$. %we took this information from the introduction of [K.\ Davidson, Almost Commuting Hermitian Matrices, Math. Scand.\ 56 (1985)].
It follows that for any $(u'_\eps)_\eps$, $(v'_\eps)_\eps\in\Mod_{\cLin(H)}$, if $u'_\eps v'_\eps = v'_\eps u'_\eps$, $\forall \eps$, then $\lim_{\eps\to 0} \norm{u_\eps - u'_\eps}+ \norm{v_\eps - v'_\eps} \ne 0$. In particular, $u\ne [(u'_\eps)_\eps]$ or $v\ne [(v'_\eps)_\eps]$.
\end{proof}

\begin{cor}
Let $\cLin(H)$ be the $C^*$-algebra of continuous $\C$-linear operators on a separable Hilbert space $H$. Then there exists a normal $w\in\Gen_{\cLin(H)}$ such that no representative $(w_\eps)_\eps$ of $w$ satisfies $w_\eps w_\eps^* = w_\eps^* w_\eps$, $\forall\eps$.
\end{cor}
\begin{proof}
Let $w= u + iv$ with $u$, $v$ as in the previous example.
\end{proof}

\begin{df}
Let $B$ be a Banach $\C$-algebra. Let $\Alg$ be a strictly internal subalgebra of $\Gen_B$. Let $u,v\in\Alg$. If there exist a strict representative $(A_\eps)_\eps$ of $\Alg$ and representatives $(u_\eps)_\eps$ of $u$ and $(v_\eps)_\eps$ of $v$ such that $u_\eps, v_\eps\in A_\eps$ and $u_\eps v_\eps = v_\eps u_\eps$, $\forall\eps$, then we say that $u$, $v$ \defstyle{commute strictly} in $\Alg$. We will refer to $(u_\eps)_\eps$, $(v_\eps)_\eps$ as commuting representatives of $u$, $v$.\\
Let $B$ be a $*$-algebra. Let $\Alg$ be a strictly internal subalgebra of $\Gen_B$. If $\Alg$ has a strict representative $(A_\eps)_\eps$ and $u\in\Alg$ has a representative $(u_\eps)_\eps$ with $u_\eps\in A_\eps$ and $u_\eps u_\eps^* = u_\eps^* u_\eps$, $\forall\eps$, then we call $u$ \defstyle{strictly normal} in $\Alg$. We will refer to $(u_\eps)_\eps$ as a normal representative of $u$.
\end{df}

\begin{thm}\label{thm_spec_of_normal_elts}
Let $B$ be a $C^*$-algebra. Let $\Alg$ be a strictly internal subalgebra of $\Gen_B$. Let $u\in \Alg$ be strictly normal in $\Alg$. Then $\spec_\Alg(u)=\intspec_\Alg(u)=[(\spec_{A_\eps}(u_\eps))_\eps]$ for each strict representative $(A_\eps)_\eps$ of $\Alg$ and normal representative $(u_\eps)_\eps$ of $u$ in $\Alg$ with $u_\eps\in A_\eps$, for each $\eps$. In particular, $\spec_\Alg(u)$ is internal.
\end{thm}
\begin{proof}
By theorem \ref{thm_non_empty_spectra}, it is sufficient to show that $\spec_\Alg(u)\subseteq [(\spec_{A_\eps} (u_\eps))_\eps]$. So let $\lambda=[(\lambda_\eps)_\eps]\in\GenC\setminus [(\spec_{A_\eps}(u_\eps))_\eps]$. By \cite[Prop.~2.1]{OV_internal}, $(d(\lambda_\eps, \spec_{A_\eps}(u_\eps)))_\eps\notin\Null_\R$.
So there exist $m\in\N$ and $S\subseteq (0,1)$ with $e_S\ne 0$ such that $d(\lambda_\eps,\spec_{A_\eps}(u_\eps))\ge \eps^m$, $\forall \eps\in S$. By \cite[Thm.~3.3.5]{Aupetit}, $d(\lambda_\eps,\spec_{A_\eps}(u_\eps)) = 1/\specrad((u_\eps-\lambda_\eps)^{-1}) = 1/\norm{(u_\eps-\lambda_\eps)^{-1}}$, since $(u_\eps-\lambda_\eps)^{-1}\in A_\eps$ is normal, $\forall\eps\in S$. Hence $\norm{(u_\eps-\lambda_\eps)^{-1}}\le\eps^{-m}$, $\forall \eps\in S$. By lemma \ref{lemma_GenB_invertible}, $u-\lambda$ is invertible w.r.t.\ $S$ in $\Alg$, and $\lambda\in\GenC\setminus \spec_\Alg(u)$.
\end{proof}
\begin{rem}
In the previous theorem, $u^*\in \Gen_B$ need not belong to $\Alg$. (In fact, if, instead of $u_\eps$ normal, $\forall\eps$, $\big(\frac{\norm[]{(u_\eps-\lambda_\eps)^{-1}}} {\specrad((u_\eps-\lambda_\eps)^{-1})}\big)_\eps$ is moderate, then also $\spec_\Alg(u)=[(\spec_{A_\eps}(u_\eps))_\eps]$.)
\end{rem}

By theorem \ref{thm_poly_spec_calc}, we obtain:
\begin{cor}
Let $B$ be a commutative $C^*$-algebra. Let $\Alg$ be a strictly internal subalgebra of $\Gen_B$. Let $u,v\in\Alg$. Then $\spec_\Alg(u)=\functspec_\Alg(u)$, $\spec_\Alg(u+v)\subseteq \spec_\Alg(u) + \spec_\Alg(v)$ and $\spec_\Alg(uv)\subseteq \spec_\Alg(u) \spec_\Alg(v)$.
\end{cor}

\begin{prop}\label{prop_Cstar_elementary_GenB}
Let $B$ be a $C^*$-algebra and $\Alg$ a strictly internal subalgebra of $\Gen_B$.
\begin{enumerate}
\item If $u\in\Gen_B$ is normal, then $\specrad(u)=\pseudonorm{u}$. If $u\in\Alg$ is strictly normal in $\Gen_B$, then $\max\{\abs\lambda: \lambda\in\spec_\Alg(u)\}=\norm{u}$.
\item If $u\in\Alg$ is unitary, then $\norm{u}=1$ and $\abs{\lambda}=1$, $\forall\lambda\in\spec_\Alg(u)$.
\item If $u\in\Alg$ is self-adjoint, then $\spec_{\Alg}(u)\subseteq\{\lambda\in\GenR: - \norm{u}\le \lambda\le \norm{u}\}$ and there exists $S\subseteq(0,1)$ such that $(e_S - e_{\co S})\norm{u}\in\spec_{\Alg}(u)$.
\item Every $m\in\Smlf{\Alg}$ is hermitian.
\end{enumerate}
\end{prop}
\begin{proof}
(1) The first assertion follows by proposition \ref{prop_Cstar_elementary} and theorem \ref{thm_specrad_formula}, the second assertion by proposition \ref{prop_GenB_invertible} and by the fact that $[(\spec_B(u_\eps))_\eps]\subseteq \spec_{\Gen_B}(u) \subseteq\spec_\Alg(u)$.\\
(2) Since $\norm{u}^2 = \norm{u u^*} = 1$ in $\GenR$ and $\norm{u}\ge 0$, we find $\norm{u}=1$. Let $\lambda\in\spec_\Alg(u)$. By proposition \ref{prop_GenB_invertible}, $\abs\lambda\le 1$. By theorem \ref{thm_poly_spec_calc}, $\inv\lambda\in\spec(\inv u)=\spec(u^*)$, hence also $\abs{\inv\lambda}\le 1$, i.e., $\abs\lambda\ge 1$.\\
(3) Let $f(z)=e^{iz}=\sum_{n=0}^\infty (iu)^n/n!$. By theorem \ref{thm_holomorphic_spec_calc}, $f(u)$ exists, for $u\in\Alg$ with $\pseudonorm{u}<1$ and $f(\spec(u))\subseteq \spec(f(u))$. Now $f(u)^*=\sum_{n=0}^\infty (-iu)^n/n!=e^{-iu}$ if $u$ is self-adjoint. By multiplication of the power series, it follows that $f(u)$ is unitary. Let $\lambda=[(\lambda_\eps)_\eps]=\mu + i\nu \in\spec(u)$. By the convergence of power series, since $\sharpnorm{\lambda}<1$, $\sum_{n=0}^\infty (i\lambda)^n/n! = [(e^{i\lambda_\eps})_\eps]=:e^{i\lambda}$. By part 2, $1 = \abs[]{e^{i\lambda}} = e^{-\nu}$. Hence also $e^{\nu}=1$. Since $1=e^{\pm\nu}\ge 1 \pm\nu$, $\nu = 0$ and $\lambda\in\GenR$.
If $\pseudonorm{u}\ge 1$, then $\pseudonorm{u\caninf^m}<1$ for some $m\in\N$. Hence $\spec(u) =\caninf^{-m}\spec(u\caninf^m) \subseteq\GenR$. The first assertion follows by proposition \ref{prop_GenB_invertible}. By lemma \ref{lemma_selfadjoint_rep}, there exists $S\subseteq (0,1)$ with $(e_S - e_{\co S})\norm u\in \intspec_{\Gen_B}(u)\subseteq\spec_\Alg(u)$.\\%the second assertion of (3) trivially also holds if $\Alg$ is an arbitrary sub-$\GenC$-algebra (with $1$) of $\Gen_B$
%one cannot directly apply theorem \ref{thm_spec_of_normal_elts}, even though lemma \ref{lemma_selfadjoint_rep} ensures the existence of a self-adjoint representative: the self-adjoint representative $(u_\eps)_\eps$ might fail to satisfy $u_\eps\in A_\eps$, $\forall\eps$ (unless we know that $A_\eps$ can be chosen $C^*$-subalgebras of $B$, $\forall\eps$).
(4) If $u\in\Alg$ is self-adjoint, then $\functspec_\Alg(u)\subseteq \spec_\Alg(u)\subseteq \GenR$ by part~(3). The result follows by lemma \ref{lemma_hermitian_mlf}.
\end{proof}

\begin{thm}\label{thm_spec_of_Cstar_subalgebra}
Let B be a $C^*$-algebra and $\Alg$ an intersection of strictly internal subalgebras of $\Gen_B$. Let $u,u^*\in\Alg$. Then $\spec_{\Alg}(u) = \spec_{\Gen_B}(u)$ and $\resolv_{\Alg}(u) = \resolv_{\Gen_B}(u)$.
\end{thm}
\begin{proof}
By proposition \ref{prop_Cstar_elementary_GenB}, $\spec_{\Gen_B}(uu^*) \subseteq\GenR$. By proposition \ref{prop_stable_spec_subset_GenR}, $\resolv_\Alg(uu^*) = \resolv_{\Gen_B}(uu^*)$. First, let $u$ be invertible in $\Gen_B$ with inverse $v\in\Gen_B$. Then also $(uu^*)^{-1} = v^* v\in \Gen_B$, so $0\in \resolv_{\Gen_B}(uu^*)=\resolv_\Alg(uu^*)$. Hence $uu^*$ is invertible in $\Alg$. By unicity of inverses, $v^* v\in \Alg$. Then also $v = u^* v^* v\in \Alg$, and $u$ is invertible in $\Alg$.\\
Now let $\lambda\in\resolv_{\Gen_B}(u)$. Since $u-\lambda$, $(u-\lambda)^*\in\Alg$, the previous reasoning shows that $\lambda\in\resolv_\Alg(u)$. Hence $\resolv_\Alg(u) = \resolv_{\Gen_B}(u)$. By theorem \ref{thm_resolv_determines_spec}, also $\spec_{\Alg}(u) = \spec_{\Gen_B}(u)$.
\end{proof}

By proposition \ref{prop_spec_Gencbd}, proposition \ref{prop_Gensharp_is_intersection_of_strictly_internal} and theorem \ref{thm_spec_of_Cstar_subalgebra}, we obtain:
\begin{cor}\label{cor_spec_Gensharp}
Let $X$ be a compact metric space. Let $u\in\Gensharp{X}$. Then $\spec(u)= \{u(\tilde x): \tilde x\in \widetilde X\}$.
\end{cor}

\subsection{Order structure}
\begin{df}
Let $\Alg$ be a faithful Colombeau $C^*$-algebra and $u\in \Alg$. Then $u$ is \defstyle{positive} (notation: $u\ge 0$) if $u$ is self-adjoint and $\spec_\Alg(u)\subseteq [0,\infty)\sptilde = \{\lambda\in\GenR: \lambda\ge 0\}$. We denote $\Alg^+ :=\{u\in\Alg : u\ge 0\}$.
\end{df}

Example \ref{ex_spec_Cstar_subalgebra} shows that $u^* u$ need not be positive, for an element $u$ of a general Colombeau $C^*$-subalgebra of $\Gen_B$ ($B$ a $C^*$-algebra).

Let $B$ be a $C^*$-algebra. For an intersection of strictly internal Colombeau $C^*$-subalgebras $\Alg$ of $\Gen_B$,
%i.e., an intersection of strictly internal subalgebras that are closed under $*$,
$\Alg^+=\Gen_B^+\cap\Alg$ by theorem \ref{thm_spec_of_Cstar_subalgebra}.
\begin{lemma}\label{lemma_positive_rep}
Let $B$ be a $C^*$-algebra and let $\Alg=[(A_\eps)_\eps]$ be a strictly internal subalgebra of $\Gen_B$ with the property that $A_\eps$ is a $C^*$-algebra, $\forall\eps$. Let $u\in\Alg$. Then $u\in\Alg^+$ iff there exists a representative $(u_\eps)_\eps$ of $u$ with $u_\eps\in A_\eps^+$, $\forall\eps$.
\end{lemma}
\begin{proof}
$\Rightarrow$: by lemma \ref{lemma_selfadjoint_rep}, there exists a representative $(u_\eps)_\eps$ with $u_\eps\in A_\eps$ and $u_\eps= u_\eps^*$, $\forall\eps$. By theorem \ref{thm_spec_of_normal_elts}, $[(\spec(u_\eps))_\eps]=\spec(u)\subseteq [0,\infty)\sptilde$. Since $\spec(u_\eps)\subseteq\R$, $\forall\eps$, we can find a negligible net of real constants $(a_\eps)_\eps$ such that $\spec(u_\eps + a_\eps)\subseteq [0,\infty)$, $\forall\eps$.\\
$\Leftarrow$: by theorem \ref{thm_spec_of_normal_elts}, $\spec(u) = [(\spec(u_\eps))_\eps]\subseteq [([0,\infty))_\eps]=[0,\infty)\sptilde$.
\end{proof}

\begin{prop}\label{prop_positivity_elementary}
Let $B$ be a $C^*$-algebra and $\Alg=\Gen_B$. Then
\begin{enumerate}
\item $\Alg^+ =\{v\in\Alg: v=v^*$ and $\norm[\big]{v - \norm{v}}\le \norm{v}\}$. In particular, $\Alg^+$ is closed.
\item If $u\in\Alg^+$ and $\lambda\in\GenR$ with $\lambda\ge 0$, then $\lambda u\in \Alg^+$.
\item If $u,v\in\Alg^+$, then $u+v\in\Alg^+$.
%\item If $u$, $v\in \Alg^+$ and there exist representatives such that $u_\eps v_\eps = v_\eps u_\eps$ and $u_\eps$, $v_\eps$ are self-adjoint, for all $\eps$, then $uv\in\Alg^+$.%We would like to have $u$, $v\in \Alg^+$ with $uv=vu$ as a condition instead!
\item If $u, -u\in\Alg^+$, then $u=0$.
\item If $u\in\Alg^+$ is invertible, then also $\inv u\in\Alg^+$.
\end{enumerate}
\end{prop}
\begin{proof}
As for classical $C^*$-algebras \cite[4.2.1--4.2.2]{Kadison}, using proposition \ref{prop_Cstar_elementary_GenB} for parts 1 and 4, proposition \ref{prop_spec_au_GenB} for part 2 and theorem \ref{thm_poly_spec_calc} for part 5.
\end{proof}

\begin{lemma}\label{lemma_norm_of_selfadjoint_product}
Let $B$ be a $C^*$-algebra. Let $u,v\in\Gen_B$ be self-adjoint. Then $\norm{uvu}\le\norm{v u^2}$.
\end{lemma}
\begin{proof}
By lemma \ref{lemma_selfadjoint_rep}, there exist representatives $(u_\eps)_\eps$ of $u$ and $(v_\eps)_\eps$ of $v$ with $u_\eps,v_\eps\in B$ self-adjoint, $\forall\eps$. Then $\norm{u_\eps v_\eps u_\eps} = \specrad_B(u_\eps v_\eps u_\eps) = \specrad_B(v_\eps u_\eps^2)\le \norm{v_\eps u_\eps^2}$, $\forall\eps$ (e.g., \cite[Prop.\ 3.2.8]{Kadison}).
\end{proof}

\begin{prop}\label{prop_sqrt}
Let $B$ be a $C^*$-algebra and let $\Alg$ be an intersection of strictly internal subalgebras $\Alg_i=[(A_{i,\eps})_\eps]$ ($i\in I$) of $\Gen_B$ with $A_{i,\eps}$ a $C^*$-algebra, $\forall i$, $\eps$. Then:
\begin{enumerate}
\item If $u\in\Alg^+$, then there exists a unique $v\in\Alg^+$ such that $u = v^2$.
\item $\Alg^+ = \{u^* u : u\in\Alg\}$.
\item If $u\in\Alg^+$, $v\in\Alg$, then $v^* u v\in \Alg^+$.
\end{enumerate}
\end{prop}
\begin{proof}
First, let $\Alg=[(A_\eps)_\eps]\subseteq \Gen_B$ be strictly internal with $A_\eps$ a $C^*$-algebra, $\forall\eps$.\\
(1, existence): let $u=[(u_\eps)_\eps]$. By lemma \ref{lemma_positive_rep}, w.l.o.g.\ $u_\eps\in A_\eps^+$, $\forall\eps$. By the classical theory, there exist $v_\eps\in A_\eps^+$ with $u_\eps = v_\eps^2$ and $\norm{v_\eps}=\sqrt{\norm[]{u_\eps}}$, $\forall\eps$. Hence $v:=[(v_\eps)_\eps]\in\Gen_B$, $u=v^2$ and by lemma \ref{lemma_positive_rep}, $v\in\Alg^+$.\\
(2, $\subseteq$): by part 1 (existence).\\
(2, $\supseteq$): let $u = [(u_\eps)_\eps]\in\Alg$ with $u_\eps\in A_\eps$, $\forall\eps$. Then $(u_\eps^* u_\eps)_\eps$ is a normal representative of $u^* u$. Hence by theorem \ref{thm_spec_of_normal_elts}, $\spec(u^* u) = [(\spec(u_\eps^* u_\eps))]\subseteq[([0,\infty))_\eps] = [0,\infty)\sptilde$.\\
(3) By part 1 (existence), there exists $w\in \Alg^+$ with $u = w^2$. Then $v^* u v = (wv)^* (w v)\in \Alg^+$ by part 2.\\
(1, unicity): Let first $u$ be invertible in $\Alg$. Let $v,w\in \Alg^+$ with $v^2 = w^2 = u$. By part 1 (existence), there exists $x\in\Alg^+$ with $w = x^2$. Then also $v,w,x$ are invertible in $\Alg$ and $\inv v, \inv w, \inv x\in \Alg^+$. Further, $\norm{v\inv w}^2=\norm{(v\inv w)^*v\inv w} = \norm{\inv w v^2 \inv w} = 1$. By lemma \ref{lemma_norm_of_selfadjoint_product}, $\norm{\inv x v \inv x}\le \norm{v(\inv x)^2} = \norm{v \inv w} = 1$. Hence $1- \inv x v\inv x\in\Alg^+$. By part 3, also $x (1- \inv x v\inv x) x = w - v\in \Alg^+$. By symmetry, also $v-w\in\Alg^+$. Hence $v = w$ by proposition \ref{prop_positivity_elementary}(5). We denote the unique $v$ by $\sqrt{u}$.\\
Now let $u\in\Alg^+$ arbitrary. By multiplying with $\caninf^a$ (for a suitable $a\in\R$), we may assume that $\norm{u}\le 1/2$. Let $v\in\Alg^+$ with $u = v^2$. Since $\spec_\Alg(u + \caninf^n)= \caninf^n + \spec_\Alg(u)$, $0\in\resolv_\Alg(u + \caninf^n)$ by proposition \ref{prop_spec_determines_resolv_internal} for $n\in\N$. Hence $\sqrt{u+\caninf^n}$ is uniquely determined by the previous case. Let $v=[(v_\eps)_\eps]$ with $v_\eps\in A_\eps^+$, $\forall\eps$. Then $\sqrt{u + \caninf^n} = [(\sqrt{v_\eps^2 + \eps^n})_\eps]$, for $n\in\N$. For small $\eps$, $\norm{v_\eps}\le 1$, so by the classical theory (e.g., \cite[Thm.\ 4.1.6]{Kadison}), $\norm[]{\sqrt{v_\eps^2 + \eps^n} - v_\eps}=\specrad(\sqrt{v_\eps^2 + \eps^n} - v_\eps) = \sup_{t\in \spec(v_\eps)\subseteq [0,1]}\abs{\sqrt{t^2 + \eps^n} - t}\le \max(\sup_{t^2\in [0,\eps^n]} \sqrt{t^2 + \eps^n}, \sup_{t^2\in [\eps^n,1], s\in [0,\eps^n]}\eps^n \abs[\big]{\frac{1}{2\sqrt{t^2 + s}}}) = \sqrt{2\eps^n}$. Hence $v=\lim_{n\to\infty} \sqrt{u+\caninf^n}$, and $v$ is uniquely determined.\\
Now let $\Alg = \bigcap_{i\in I}\Alg_i$ with $\Alg_i=[(A_{i,\eps})_\eps]$ and $A_{i,\eps}$ a $C^*$-algebra, $\forall i$, $\eps$.\\
(1) Let $\sqrt u$ denote the unique element in $\Gen_B^+$ such that $u=\sqrt{u}^2$. By the previous case, $\sqrt{u}\in\Alg_i$, for each $i$. Hence $\sqrt u\in \Alg\cap\Gen_B^+ = \Alg^+$. Any element $v\in\Alg^+$ with $u=v^2$ also belongs to $\Gen_B^+$, so equals $\sqrt u$ by unicity in $\Gen_B$.\\
(2, $\supseteq$): let $u\in\Alg$. By the previous case, $u^* u\in\Alg\cap\Gen_B^+=\Alg^+$.\\
(2, $\subseteq$) and (3) now follow as in the previous case.
\end{proof}
\begin{cor}
Let $B$ be a $C^*$-algebra. If $\Alg$ is an intersection of strictly internal Colombeau $C^*$-subalgebras of $\Gen_B$, then $\Alg$ is symmetric.
\end{cor}
\begin{proof}
Let $u\in \Alg$. By theorem \ref{thm_spec_of_Cstar_subalgebra} and proposition \ref{prop_sqrt}, $u^* u\in \Alg\cap \Gen_B^+ = \Alg^+$. Hence $-1 \in \resolv_\Alg(u^* u)$ by proposition \ref{prop_spec_determines_resolv_internal}.
\end{proof}

\subsection{Positive linear functionals}
\begin{df}
Let $\Alg$ be a faithful Colombeau $C^*$-algebra. Let $f$: $\Alg\to\GenC$ be a $\GenC$-linear functional. Then $f$ is called \defstyle{positive} if $f(u)\ge 0$ (in $\GenR$), for each $u\in\Alg^+$. If additionally $f(1)=1$, then $f$ is called a \defstyle{state}.
\end{df}
Clearly, if $\Alg^+=\{u^*u: u\in\Alg\}$, then $f$ is positive iff $f(u^*u)\ge 0$, for each $u\in\Alg$.

\begin{prop}\label{prop_positive_is_hermitian}
Let $\Alg$ be a faithful Colombeau $C^*$-algebra with $\Alg^+=\{u^*u: u\in\Alg\}$. Let $f$ be a positive $\GenC$-linear functional on $\Alg$. Then
\begin{enumerate}
\item $f(u^*)=\overline{f(u)}$, for each $u\in\Alg$.
\item $\abs{f(v^* u)}^2\le f(u^* u) f(v^* v)$, for each $u,v\in\Alg$.
\end{enumerate}
\end{prop}
\begin{proof}
(1) Since $f((u+1)^*(u+1))\in\GenR$, it follows that $f(u) + f(u^*)\in\GenR$, hence $\Im f(u) + \Im f(u^*) = 0$. Since $f((u+i)^*(u+i))$ $\in\GenR$, it follows that $i f(u) - i f(u^*)\in\GenR$, hence $\Re f(u) - \Re f(u^*) = 0$.\\
(2) By part 1, $\inner{u}{v}:=f(v^*u)$ defines a $\GenC$-sesquilinear map on the $\GenC$-module $\Alg$ with $\inner{u}{v} = \overline{\inner{v}{u}}$ and $\inner{u}{u}\ge 0$ (in $\GenR$), for each $u,v\in\Alg$. The result follows by proposition \ref{prop_C-S}.
\end{proof}

We recall the following fact about $\GenR$ \cite[Lemma 2.19]{GV_Hilbert}:
\begin{lemma}\label{lemma_scratch}
Let $a,b,c\in\wt{\R}$ with $a,b,c\ge 0$. If $b\le a$ and $ab\le ac$ then $b\le c$.
\end{lemma}

\begin{prop}\label{prop_positive_functional_char}
Let $B$ be a $C^*$-algebra and let $\Alg=[(A_\eps)_\eps]$ be a strictly internal subalgebra of $\Gen_B$ with the property that $A_\eps$ is a $C^*$-algebra, $\forall\eps$. Let $f$ be a $\GenC$-linear functional on $\Alg$. Then $f$ is positive iff $f(1)\ge 0$ (in $\GenR$) and $\abs{f(u)}\le f(1)\norm{u}$, $\forall u\in\Alg$ (in particular, $f$ is continuous, cf.\ \cite[Prop.\ 1.7]{GV_Hilbert}).
\end{prop}
\begin{proof}
$\Rightarrow$: by propositions \ref{prop_sqrt} and \ref{prop_positive_is_hermitian}(2), with $v=1$, $\abs{f(u)}^2\le f(1) f(u^* u)$, $\forall u\in\Alg$. Further, for $u\in\Alg$, $\norm{u^* u} - u^* u$ is self-adjoint and $\spec(\norm{u^*u} - u^*u) = \norm{u^* u} - \spec(u^* u)\subseteq [0,\infty)\sptilde$ by proposition \ref{prop_Cstar_elementary_GenB}. Hence $f(\norm{u^*u}-u^* u)\ge 0$, and $f(u^*u)\le f(1) \norm{u^*u}$. Thus $\abs{f(u)}^2\le f(1)^2 \norm{u}^2$, and $\abs{f(u)}\le f(1) \norm{u}$ (since $\sqrt{\,.\,}$ is well-defined on $\{\tilde x\in\GenR: x\ge 0\}$).\\
$\Leftarrow$: let first $u\in\Alg$ be self-adjoint. We show that $f(u)\in\GenR$. Let $f(u)=\alpha + i\beta$, $\alpha,\beta\in\GenR$. Let $v:= t - iu$, $t\in\GenR$. Then $\norm{v}^2 = \norm{v^* v} = \norm{t^2 + u^2}\le t^2 + \norm{u}^2$ and $\abs{f(v)}^2 = \abs{tf(1) - if(u)}^2 = t^2 f(1)^2 + 2 t f(1) \beta + \alpha^2 + \beta^2$. Hence $f(1)^2\norm{v}^2\le f(1)^2 t^2 + f(1)^2\norm{u}^2 \le \abs{f(v)}^2 - 2tf(1)\beta + f(1)^2\norm{u}^2 \le f(1)^2\norm{v}^2 - 2tf(1)\beta + f(1)^2\norm{u}^2$, and $2tf(1)\beta \le f(1)^2\norm{u}^2$. Choosing $t=\caninf^{-n}\mathop{\mathrm{sgn}}(\beta)$ ($n\in\N$), we find $2 f(1)\abs{\beta}\le f(1)^2\norm{u}^2\caninf^n$. Since $\abs\beta\le \abs{f(u)}\le f(1) \norm{u}\le \caninf^{-N}f(1)$, for some $N\in\N$, we can apply lemma \ref{lemma_scratch} and obtain $2\abs{\beta}\le f(1)\norm{u}^2\caninf^n$. As $n$ is arbitrary, $\beta = 0$.\\
Now let $u\in\Alg^+$. Since $f(\norm{u} - u)\in\GenR$, $f(u) = f(1) \norm{u} - f(\norm{u} - u)\ge f(1) \norm{u} - f(1)\,\norm[\big]{\norm{u} - u}\ge 0$ by proposition \ref{prop_positivity_elementary}(1).
\end{proof}

\begin{prop}\label{prop_state_with_norm_prop}
Let $B$ be a $C^*$-algebra. Let $\Alg$ be a Colombeau $C^*$-subalgebra of $\Gen_B$. Let $u\in\Alg$. Then there exists a state $s$ on $\Alg$ with $s(u^*u)=\norm{u}^2$.
\end{prop}
\begin{proof}
Let $u=[(u_\eps)_\eps]$. By the classical theory, there exist states $s_\eps$ on $B$ with $s_\eps(u_\eps^* u_\eps)=\norm{u_\eps}_B^2$, $\forall\eps$. Since $\abs{s_\eps(v)}\le \norm{v}_B$, for each $v\in B$, $(s_\eps)_\eps$ defines a basic $\GenC$-linear functional \cite[\S 1.1.2]{GV_Hilbert} $s$ on $\Gen_B$, with $s(1)=1$, $s(u^* u) = \norm{u}^2$ and $\abs{s(v)}\le \norm{v}$, $\forall v\in\Gen_B$. By proposition \ref{prop_positive_functional_char}, $s$ is a state on $\Gen_B$, and the restriction of $s$ to $\Alg$ is a state on $\Alg$.
\end{proof}

We could now proceed to prove an analogue of the Gelfand-Neumark theorem (e.g., \cite[Thm.\ 4.5.6]{Kadison}) for Colombeau $C^*$-subalgebras of $\Gen_B$ ($B$ a classical $C^*$-algebra) along the lines of the classical proof and show that such an algebra is isomorphic (as a Colombeau $C^*$-algebra) with a Colombeau $C^*$-subalgebra of $\mathcal B(\mathcal H)$, for some Hilbert $\GenC$-module $\mathcal H$. It is easier to obtain this result using the classical Gelfand-Neumark theorem on the representatives:
\begin{prop}\label{prop_GNS}
Let $B$ be a $C^*$-algebra and let $\Alg$ be a Colombeau $C^*$-subalgebra of $\Gen_B$. Then $\Alg$ is isomorphic (as a Colombeau $C^*$-algebra) with a Colombeau $C^*$-subalgebra of $\Gen_{\cLin(H)}$ (hence, in particular, of $\mathcal B(\Gen_H)$), for some Hilbert space $H$.
\end{prop}
\begin{proof}
By the classical Gelfand-Neumark theorem, there exists a Hilbert space $H$ and an isometric $*$-isomorphism $T$: $B\to\cLin(H)$. By \cite[1.1.2]{GV_Hilbert}, $T$ defines a so-called basic (hence continuous) $\GenC$-linear map $\Gen_B\to \Gen_{\cLin(H)}$ by means of $T([(u_\eps)_\eps]):= [(T(u_\eps))_\eps]$. By proposition \ref{prop_G_B(H)_is_C_star_subalgebra}, $\Gen_{\cLin(H)}$ is a Colombeau $C^*$-subalgebra of $\mathcal B(\Gen_H)$ and it is straightforward to check that $T$ is an isometric $*$-morphism. In particular, this also holds for the restriction of $T$ to $\Alg$.
\end{proof}

\end{document}